\documentclass[a4paper,11pt]{article}

\usepackage[ngerman, american]{babel} 
\usepackage[utf8]{inputenc}
\usepackage[a4paper, left=3cm, right=3cm, top=2.5cm, bottom=2.5cm]{geometry}
\usepackage{graphicx}
\usepackage{float}
\usepackage{color,psfrag}
\usepackage{amssymb}
\usepackage{amsmath} 
\usepackage{amsthm}
\usepackage{dsfont}
\usepackage{mathrsfs}
\usepackage{color}
\usepackage{setspace}
\usepackage{comment}
\usepackage{tikz}
\usetikzlibrary{decorations.pathreplacing,decorations.markings}
\tikzset{
  on each segment/.style={
    decorate,
    decoration={
      show path construction,
      moveto code={},
      lineto code={
        \path [#1]
        (\tikzinputsegmentfirst) -- (\tikzinputsegmentlast);
      },
      curveto code={
        \path [#1] (\tikzinputsegmentfirst)
        .. controls
        (\tikzinputsegmentsupporta) and (\tikzinputsegmentsupportb)
        ..
        (\tikzinputsegmentlast);
      },
      closepath code={
        \path [#1]
        (\tikzinputsegmentfirst) -- (\tikzinputsegmentlast);
      },
    },
  },
  mid arrow/.style={postaction={decorate,decoration={
        markings,
        mark=at position .5 with {\arrow[#1]{stealth}}
      }}},
}
\usepackage{caption}
\usepackage{calligra}
\DeclareMathAlphabet{\mathcalligra}{T1}{calligra}{m}{n}
\usepackage{hyperref}
\usepackage[nottoc]{tocbibind}

\newcommand{\NN}{\nonumber}
\newcommand{\eps}{{\varepsilon}}        
          
\renewcommand{\phi}{{\varphi}}          
\newcommand{\om}{\omega}

\newcommand{\cA}{\mathcal{A}}

\newcommand{\cE}{\mathcal{E}}
\newcommand{\cF}{\mathcal{F}}
\newcommand{\cG}{\mathcal{G}}

\newcommand{\cM}{\mathcal{M}}        
         
\newcommand{\cO}{\mathcal{O}}        
\newcommand{\cP}{\mathcal{P}}

\newcommand{\cS}{\mathcal{S}}

\newcommand{\C}{\mathds{C}}
\newcommand{\R}{\mathds{R}}
\newcommand{\N}{\mathds{N}}
\newcommand{\Z}{\mathds{Z}}

\newcommand{\m}{\mathfrak{m}}

\newcommand{\ri}{{\mathrm i}}
\newcommand{\dx}{{\mathrm d}}
\newcommand{\I}{\mathds{1}}
\newcommand{\1}{\mathbf{1}}

\newcommand{\E}{\mathds{E}}
\renewcommand{\P}{\mathds{P}}
\newcommand{\Cov}{\mathrm{Cov}}

\newcommand{\diag}[1]{\mathrm{diag}(#1)}

\newcommand{\Tr}{\mathop{\mathrm{Tr}}}

\newcommand{\Id}{\mathrm{Id}}

\newcommand{\niton}{\not\owns}

\newcommand{\NCA}{\smash{\overrightarrow{NCP}}}

\newcommand{\nsp}{\negmedspace}

\newtheorem{theorem}{Theorem}[section]         
\newtheorem{lemma}[theorem]{Lemma}             
\newtheorem{corollary}[theorem]{Corollary}
\newtheorem{definition}[theorem]{Definition}

\newtheorem{assumption}[theorem]{Assumption}

\newtheoremstyle{myrem}
  {}
  {}
  {}
  {}
  {\bfseries}
  {.}
  { }
  {\thmname{#1}\thmnumber{ #2}\thmnote{\normalfont{ (#3)}}}
\theoremstyle{myrem}
\newtheorem*{remark}{Remark}
\newtheorem{example}[theorem]{Example}

\numberwithin{equation}{section} 

\begin{document}
\title{\vspace{-2cm} Fluctuation Moments for Regular Functions of Wigner Matrices}
\author{Jana Reker\thanks{IST Austria, Am Campus 1, 3400 Klosterneuburg, Austria. E-Mail: jana.reker$@$ist.ac.at.}}
\maketitle

\begin{abstract}
We compute the deterministic approximation for mixed fluctuation moments of products of deterministic matrices and general Sobolev functions of Wigner matrices. Restricting to polynomials, our formulas reproduce recent results of Male, Mingo, Peché, and Speicher~\cite{MaleMingoPecheSpeicher2020}, showing that the underlying combinatorics of non-crossing partitions and annular non-crossing permutations continue to stay valid beyond the setting of second-order free probability theory. The formulas obtained further characterize the variance in the functional central limit theorem obtained recently in the companion paper~\cite{JRmain}.
\end{abstract}


\section{Introduction}\label{sect-intro}		

In his seminal work~\cite{Wigner1955}, Wigner established that the empirical spectral measure of certain random matrix ensembles converges, as the dimension goes to infinity, to the semicircle distribution. Since then, many variations and extensions of this result have been considered, yielding a variety of asymptotic phenomena for a wide range of random matrix models. One particular example is the fact that the resolvent $G(z)=(W-z)^{-1}$ of a large Hermitian random matrix $W$ tends to concentrate around a deterministic matrix $M = M (z)$ for spectral parameters $z\in\C$ even just slightly away from the real axis (see, e.g.,~\cite{CES-optimalLL} and references therein for a collection of recent results). It was recently shown (see~\cite{CES-thermalization, CES-optimalLL}) that a similar concentration holds for alternating products of the form
\begin{equation}\label{eq-altproduct}
F_{[1,k]}:=f_1(W)A_1\dots f_k(W)A_k.
\end{equation}
Here, $A_1,\dots,A_k$ are bounded deterministic matrices and $f_1,\dots,f_k$ are regular test functions, allowing in particular for $f_j(W)=G(z_j)$. Products of the form~\eqref{eq-altproduct} with $f_j(W)$ replaced by (polynomials of) the random matrix itself play a key role in free probability theory, as they characterize the joint non-commutative probability distribution of Wigner and deterministic matrices.

\medskip
We remind the reader that a (tracial \emph{first-order}) non-commutative probability space is a pair $(\cA,\phi_1)$ consisting of a complex unital algebra $\cA$ and a tracial linear functional $\phi_1:\cA\rightarrow\C$ with $\phi_1(1_{\cA})=1$, where $1_{\cA}$ is the unit element of the algebra. One particular example is the space $(\cA,\phi_1)=(\cM_{N\times N}(L^{\infty-}(\Omega,\P)),\E\langle\cdot\rangle)$ of $N\times N$ random matrices, where $(\Omega,\P)$ is a classical probability space, $\cM_{N\times N}(S)$ denotes the $N\times N$-matrices with entries in $S$, the space
\begin{displaymath}
L^{\infty-}(\Omega,\P):=\bigcap_{1\leq p<\infty}L^p(\Omega,\P)
\end{displaymath}
contains all random variables with all finite moments, and $\langle\cdot\rangle$ denotes the normalized trace. Note that this definition includes deterministic and Wigner matrices. In this context, the non-commutative probability distribution of $a\in\cA$ is characterized in terms of its moments $(\phi_1(a^k))_k$ with the joint distribution of multiple elements of $\cA$ being defined analogously. Recent work by Cipolloni, Erd\H{o}s and Schröder~\cite{CES-thermalization} established that the structure of the limit of $\E\langle F_{[1,k]}\rangle$ as in~\eqref{eq-altproduct} matches the formulas obtained in free probability, and reproduces known results for the alternating moments $\E\langle W_1D_1\dots W_kD_k\rangle$ of a finite family of independent Wigner matrices $(W_j)_j$ and a finite family of deterministic matrices $(D_j)_j$ (see, e.g.,~\cite[Sect.~4.4]{MSBook}) in the case $f_j(x)=x$. More precisely, in the large $N$ limit, the leading-order term $\m[F_{[1,k]}]$ of $ \E\langle F_{[1,k]}\rangle$ is of the form
\begin{equation}\label{eq-1storderlimit}
\m[F_{[1,k]}]:=\sum_{\pi\in NCP([k])}\Big(\prod_{B\in\pi}\Big\langle \prod_{j\in B}A_j\Big\rangle\Big)\Phi^{(1)}_\pi(f_1,\dots,f_k),
\end{equation}
where $NCP([k])$ denotes the non-crossing partitions of the cyclically ordered set $\{1,\dots,k\}$ and the functions $\smash{\Phi^{(1)}_\pi}$ only depend on $f_1,\dots,f_k$ and $\pi\in NCP([k])$. Hence, the right-hand side of~\eqref{eq-1storderlimit} is a sum of terms that factorize into a contribution of the deterministic matrices resp. the test functions appearing in the product~\eqref{eq-altproduct} with the underlying combinatorics matching the results obtained for the case $f_j(x)=x$ in free probability theory. Note, however, that resolvents and functions with an $N$-dependent mesoscopic scaling are typically not accessible in free probability as many of the standard techniques rely on explicit moment computations for polynomials. The results in~\cite{CES-thermalization} thus show that the underlying combinatorics continue to apply in a more general context.

\medskip
After considering the concentration of~\eqref{eq-altproduct}, the next natural step is to study the fluctuations around the deterministic value. It is well-known that the linear statistics $\Tr f(W)=\sum_{j=1}^Nf(\lambda_j)$ with a regular test function $f:\R\rightarrow\R$ have a variance of order one (first observed in~\cite{KhorunzhyKhoruzhenkoPastur1995}) and, in fact, satisfy a central limit theorem (CLT) with a Gaussian limit, as shown, e.g., in~\cite{KhorunzhyKhoruzhenkoPastur1996} for the Wigner case and in~\cite{Johansson1998} for invariant ensembles. By now, the statistics $\Tr f(W)$ are well-studied on both macroscopic and mesoscopic scales (see, e.g.,~\cite{Guionnet2002, BaiYao2005, LytovaPastur2009,LytovaPasturCLT, Shcherbina2011, SosoeWong2013, BorotGuionnet2013, Shcherbina2013, HeKnowles2020,BaoHe2021, LandonSosoe2022} for the Wigner case and~\cite{CES-functCLT, JRmain} for further references on previous results for Wigner matrices and other models). However, while the fluctuations of $\Tr [f(W)A]$ are known for general regular functions $f$ (see~\cite{Lytova2013} and~\cite{CES-functCLT}), traces of products of the form~\eqref{eq-altproduct} for $k\geq2$ have so far only been studied for $f_j$ being polynomials in the context of second-order freeness (see, e.g., \cite[Ch.~5]{MSBook} or~\cite{CollinsMingoSniadySpeicher2007, Male2021, MaleMingoPecheSpeicher2020}). 

\medskip
We remind the reader that a \emph{second-order} non-commutative probability space is a triplet $(\cA,\phi_1,\phi_2)$, where the functional $\phi_2:\cA\times\cA\rightarrow\C$ is bilinear, tracial in both arguments, symmetric under the interchanging of its arguments, and satisfies ${\phi_2(a,1_{\cA})=\phi_2(1_{\cA},a)=0}$ for all $a\in\cA$. The second-order probability distribution of $a\in\cA$ is characterized in terms of $(\phi_2(a^k,a^\ell))_{k,\ell}$, called the \emph{fluctuation moments}, with the joint moments of multiple elements again being defined analogously. As a canonical example, we remark that $\cM_{N\times N}(L^{\infty-}(\Omega,\P))$ may be endowed with the functional $\phi_2(\cdot,\cdot)=\Cov(\Tr(\cdot),\Tr(\cdot))$, to make it a second-order probability space. In contrast to the first-order structure, the fluctuation moments are sensitive to the symmetry class of the underlying Wigner matrix and explicitly involve the fourth cumulant of the entry distribution (see~\cite[Thm.~6]{MaleMingoPecheSpeicher2020}, as well as~\cite{Redelmeier2012, Redelmeier2018, CES-functCLT}). In particular, we observe a breaking of universality compared to the first-order problem of computing $\E\langle \cdot\rangle$. The joint fluctuation moments of Wigner and deterministic matrices are explicitly known (cf.~\cite[Thm.~13 of Ch.~5]{MSBook} for the GUE case and~\cite[Thm.~6]{MaleMingoPecheSpeicher2020} for general Wigner matrices). 

\medskip
A functional CLT for traces of products of the form~\eqref{eq-altproduct} has recently been established in the companion paper~\cite{JRmain} and the limiting covariance is derived using a recursion. In the present paper, we supply the combinatorial argument necessary to obtain the solution to the recursion and compute the limiting covariance explicitly. More precisely, we show that if $W$ is a GUE matrix, the leading order term $\m[F_{[1,k]}|F_{[k+1,k+\ell]}]$ of the covariance of $\Tr( F_{[1,k]})$ and $\Tr(F_{[k+1,k+\ell]})$ (with $F_{[k+1,k+\ell]}=f_{k+1}A_{k+1}\dots f_{k+\ell}A_{k+\ell}$ of the same build as~\eqref{eq-altproduct}) is given by
\begin{align}
&\m[F_{[1,k]}|F_{[k+1,k+\ell]}]=\sum_{\pi\in \NCA(k,\ell)}\Big(\prod_{B\in\pi}\Big\langle \prod_{j\in B}A_j\Big\rangle\Big)\Phi^{(2)}_\pi(f_1,\dots,f_{k+\ell})\label{eq-2ndorderlimit}\\
&\quad+\sum_{\pi_1\times\pi_2\in NCP(k)\times NCP(\ell)}\Big(\prod_{B_1\in\pi_1,B_2\in\pi_2}\Big\langle \prod_{j\in B_1} A_j\Big\rangle\Big\langle \prod_{j\in B_2} A_j\Big\rangle\Big)\Phi^{(2)}_{\pi_1\times\pi_2}(f_1,\dots,f_{k+\ell}).\NN
\end{align}
Here, $\NCA(k,\ell)$ denotes the non-crossing permutations of the $(k,\ell)$-annulus and the functions $\smash{\Phi^{(2)}_\pi}$ resp. $\smash{\Phi^{(2)}_{\pi_1\times\pi_2}}$ only depend on $f_1,\dots,f_{k+\ell}$ and the underlying permutation resp. partition. Similar to~\eqref{eq-1storderlimit}, we thus obtain a sum of terms that factorize into a contribution of the deterministic matrices resp. the test functions appearing in the product~\eqref{eq-altproduct} with the underlying combinatorics again matching the results obtained for the case $f_j(x)=x$ in free probability theory (see~\cite{MingoSpeicher2006}). Moreover, we show that the overall structure of~\eqref{eq-2ndorderlimit} continues to hold if $W$ is chosen to be a Wigner matrix with $\smash{W_{ij}\overset{d}{=}N^{-1/2}\chi_{od}}$ for~${i<j}$ and $\smash{W_{jj}\overset{d}{=}N^{-1/2}\chi_{d}}$ for general entry distributions $\chi_{od}$ and $\chi_d$. In the general case, however, the sum in the first line of the right-hand side of~\eqref{eq-2ndorderlimit} splits into four summands $\smash{\Phi^{(GUE)}_{\pi}}$, $\smash{\kappa_4\Phi^{(\kappa)}_{\pi}}$, $\smash{\sigma\Phi^{(\sigma)}_{\pi}}$, and $\smash{\widetilde{\omega}_2\Phi^{(\omega)}_{\pi}}$ which have different prefactors in terms of the deterministic matrices $A_1,\dots,A_{k+\ell}$. Here, $\smash{\Phi^{(GUE)}_{\pi}}$ corresponds to the GUE case in~\eqref{eq-2ndorderlimit} and the remaining contributions are associated with the parameters
\begin{equation}\label{eq-parameters}
\kappa_4=\E|\chi_{od}|^4-2,\quad \sigma=\E\chi_{od}^2,\quad \widetilde{\omega}_2=\E\chi_d^2-1-\sigma
\end{equation}
of the Wigner matrix $W$. A similar decomposition is also observed for $\smash{\Phi^{(2)}_{\pi_1\times\pi_2}}$ in~\eqref{eq-2ndorderlimit}. In particular, we find that the closed expression obtained from solving the recursion in~\cite{JRmain} has the same overall structure as the formulas in~\cite[Thm.~6]{MaleMingoPecheSpeicher2020}. This shows that the analogies~\cite{CES-thermalization} established in the first-order setting have a counterpart for the second-order structures. Our combinatorial approach further allows us to give the functions in~\eqref{eq-2ndorderlimit} in a closed form, thus yielding a fully explicit formula for the limiting covariance in the GUE case.

\medskip
We remark that the main results of the present paper, i.e., combinatorial formulas for $\m[F_{[1,k]}|F_{[k+1,k+\ell]}]$ such as~\eqref{eq-2ndorderlimit}, are applied to obtain an explicit limiting covariance structure for the multi-point functional CLT~\cite[Thm.~2.7]{JRmain}. Here, replacing the recursive definition of the limiting variance by a closed formula allows for an easier application of the theorem, e.g., to thermalization problems in physics (cf.~\cite[Cor.~2.12]{JRmain}). We emphasize that the main results in the companion paper~\cite{JRmain} are of analytic nature and that their main technical difficulty lies in including functions with a mesoscopic scaling of the form
\begin{equation}\label{eq-mesofunct}
f_j(x)=g_j(N^\gamma(x-E))
\end{equation}
where $g_j$ is a regular $N$-independent function, $E\in\R$ lies in the bulk of the limiting spectrum of $W$, and $N^{-\gamma}$ is larger than the typical eigenvalue spacing around $E$. In contrast, we assume all test functions to be $N$-independent in the present paper and focus on the combinatorial structures arising in the multi-point functional CLT. While an extension to the functions in~\eqref{eq-mesofunct} is possible using the techniques from~\cite{JRmain}, restricting to the macroscopic regime allows for a cleaner presentation of the results. It further facilitates working with more general assumptions on the Wigner matrix $W$. Note that~\cite[As.~1.1]{JRmain} corresponds to setting $\sigma=\widetilde{\omega}_2=0$ in~\eqref{eq-parameters}, while Assumption~\ref{as-Wigner2} below matches the setting of~\cite{MaleMingoPecheSpeicher2020, CES-functCLT} with general $\sigma\in[-1,1]$ and $\widetilde{\omega}_2\geq-2$, thus generalizing the formulas from~\cite{JRmain}.

\medskip
We conclude the section with a brief overview of the paper. After introducing some commonly used notations, the assumptions on the Wigner matrix $W$ are given in Assumption~\ref{as-Wigner2}. We then give a brief overview of the combinatorics needed to identify the deterministic approximation of $\langle T_1\dots T_k\rangle$ where $T_j:=G(z_j)A_j$, and the multi-resolvent local laws needed for the analysis of the fluctuations (Section~\ref{sect-prelim1}) as well as the definitions from free probability that are used to characterize the limiting covariance of $\langle T_1\dots T_k\rangle-\E\langle T_1\dots T_k\rangle$ and $\langle T_{k+1}\dots T_{k+\ell}\rangle-\E\langle T_{k+1}\dots T_{k+\ell}\rangle$ (Section~\ref{sect-prelim2}). To prepare for the statements of our main results, we give a CLT for the case that all functions $f_j$ are resolvents (Theorem~\ref{thm-resolventCLT2}). The role of the limiting covariance in the theorem is played by a recursively defined set function $\m[\cdot|\cdot]$ (Definition~\ref{def-M}), which is the main object of interest in the present paper. We study the recursion in detail in Section~\ref{sect-formulas} and obtain explicit combinatorial formulas for its solution (Theorems~\ref{thm-main},~\ref{thm-structurekap} - \ref{thm-structureom}). In Section~\ref{sect-CLT2}, we extend the CLT to more general test functions (Theorem~\ref{thm-functCLT} and Corollary~\ref{cor-covarianceLL}) to discuss the connection to free probability theory in detail. In particular, we apply the results to the case $f_j(x)=x$ and show that the limiting covariance in the functional CLT reduces to the formula for the joint fluctuation moments of GUE and deterministic matrices (Corollary~\ref{cor-fpapplication}) as given in~\cite{MaleMingoPecheSpeicher2020}. Lastly, the proofs are given in Sections~\ref{sect-mainproof} and~\ref{sect-restproof}. To keep the presentation concise, some routine calculations are deferred to the appendix.

\medskip
\textbf{Acknowledgements:} I am very grateful to László Erd\H{o}s for suggesting the topic and many valuable discussions during my work on the project. Partially supported by ERC Advanced Grant "RMTBeyond" No.~101020331.

\subsection{General Notation}\label{sect-prelim}
We start by introducing some notation used throughout the paper. For two positive quantities $f,g$, we write $f\lesssim g$ and $f\sim g$ whenever there exist (deterministic, $N$-independent) constants $c,C>0$ such that $f\leq Cg$ and $cg\leq f\leq Cg$, respectively. We denote the Hermitian conjugate of a matrix $A$ by $A^*$ and the complex conjugate of a scalar~$z\in\C$ by $\overline{z}$. Moreover, $\|\cdot\|$ denotes the operator norm, $\mathrm{Tr}(\cdot)$ is the usual trace and $\langle\cdot\rangle=N^{-1}\Tr(\cdot)$. We further denote the covariance of two complex random variables $X_1,X_2$ by $\Cov(X_1,X_2)$ and follow the convention
\begin{displaymath}
\Cov(Y_1,Y_2)=\E(Y_1-\E Y_1)\overline{(Y_2-\E Y_2)},
\end{displaymath}
i.e., the covariance is linear in the first and anti-linear in the second entry. For $k,a,b\in \N$ with $a\leq b$, we set $[k]=\{1,\dots,k\}$ and adopt the interval notation $[a,b]=\{a,a+1,\dots,b\}$. We further write $\langle a,b]$ or $[a,b\rangle$ to indicate that $a$ or $b$ are excluded from the interval, respectively. Ordered sets are denoted by $(\dots)$ instead of $\{\dots\}$.

\medskip
Given a matrix $A\in\C^{N\times N}$, the traceless part of $A$ is denoted by $\mathring{A}:=A-\langle A\rangle \Id$ where $\Id$ denotes the identity matrix. Further, $\mathbf{a}:=\diag{A}$ denotes the diagonal matrix obtained from extracting only the diagonal entries of $A$ and $A_1\odot A_2$ denotes the entry-wise (or Hadamard) product of two matrices $A_1$ and $A_2$. For a Hermitian matrix $W$ and $z_1,\dots,z_k\in\C\setminus\R$, we write the corresponding resolvents as $G_j=G(z_j):=(W-z_j)^{-1}$ and index products of resolvents using the interval notation
\begin{displaymath}
G_{[a,b]}:=G_aG_{a+1}\dots G_b
\end{displaymath}
for $a,b\in\N$ with $a\leq b$. Recalling that angled brackets indicate that an edge point of the interval is excluded, we write $G_{\langle a,b]}$ and $G_{[a,b\rangle}$ to exclude $G_a$ or $G_b$ from the product, respectively. Moreover, $G_{\emptyset}$ is interpreted as zero. Note that this notation differs slightly from~\cite{CES-thermalization,CES-optimalLL}. As we often consider alternating products of resolvents with deterministic matrices $A_1,\dots,A_k$, define $T_j:=G_jA_j$ and apply the same interval notation as above to write
\begin{equation}\label{eq-reschain}
T_{[k]}:=T_1\dots T_k=G_1A_1\dots G_kA_k,\quad T_{[a,b]}:=T_aT_{a+1}\dots T_b.
\end{equation}
Again, angled brackets are used to exclude $T_a$ or $T_b$ from the product, respectively, and $T_{\emptyset}$ is interpreted as zero. We call a product of the type~\eqref{eq-reschain} \emph{resolvent chain} of length $k$.

\medskip
Throughout the paper, we assume $W$ to be an $N\times N$ real or complex Wigner matrix satisfying the following assumptions.
\begin{assumption}\label{as-Wigner2}
The matrix elements of $W$ are independent up to Hermitian symmetry $\smash{W_{ij}=\overline{W_{ji}}}$ and we assume identical distribution in the sense that there is a centered real random variable $\chi_d$ and a centered real or complex random variable $\chi_{od}$ such that $\smash{W_{ij}\overset{d}{=}N^{-1/2}\chi_{od}}$ for~${i<j}$ and $\smash{W_{jj}\overset{d}{=}N^{-1/2}\chi_{d}}$, respectively. We further assume that $\E|\chi_{od}|^2=1$ as well as the existence of all moments of $\chi_d$ and $\chi_{od}$, i.e., there exist constants $C_p>0$ for any $p\in\N$ such that
\begin{displaymath}
\E|\chi_d|^p+\E|\chi_{od}|^p\leq C_p.
\end{displaymath}
\end{assumption}

We remark that Assumption~\ref{as-Wigner2} matches the model considered in~\cite{CES-functCLT} and~\cite{MaleMingoPecheSpeicher2020}. Compared to the conditions $\E\chi_{od}^2=0$ and $\E\chi_d^2=1$ in~\cite{JRmain}, we allow for arbitrary values of the parameters $\sigma=\E\chi_{od}^2\in[-1,1]$ and $\omega_2=\E\chi_d^2\geq0$. This description includes real symmetric Wigner ensembles such as GOE ($\sigma=1$) as well as matrices of the form $W=D+\ri S$ where $D$ is a diagonal matrix and $S$ is skew-symmetric ($\sigma=-1)$. We further introduce the notation
\begin{equation}\label{eq-defkappa4}
\kappa_4:=\E|\chi_{od}|^4-2
\end{equation}
for the normalized fourth cumulant of the off-diagonal entries as well as
\begin{equation}\label{eq-defomega2t}
\widetilde{\om_2}:=\om_2-1-\sigma.
\end{equation}
The eigenvalue density profile of $W$ is described by the semicircle law
\begin{equation}\label{eq-scdensity}
\rho_{sc}(x):=\frac{\sqrt{x^2-4}}{2\pi}\I_{[-2,2]}(x)
\end{equation}
which mainly enters our analysis in the form of its Stieltjes transform
\begin{equation}\label{eq-defm}
m(z):=\int\frac{\rho_{sc}(x)}{z-x}\dx x,\quad z\in\C\setminus\R.
\end{equation}
We remind the reader that $m(z)$ is the unique solution of the Dyson equation
\begin{equation}\label{eq-mselfcon}
-\frac{1}{m(z)}=m(z)+z,\quad \Im z\Im m(z)>0
\end{equation}
and that its derivative satisfies
\begin{equation}\label{eq-mselfconderived}
m'(z)=\frac{m(z)^2}{1-m(z)^2}.
\end{equation}
Given fixed $z_1,\dots,z_k\in\C\setminus\R$, set $m_j=m(z_j)$ and $m_j'=m'(z_j)$, respectively. We further introduce
\begin{equation}\label{eq-defq}
q_{i,j}=\frac{m_im_j}{1-m_im_j},
\end{equation}
possibly setting $q_{j,j}=m_j'$ whenever $i=j$.

\subsection{Preliminaries Part 1: First-Order Quantities}\label{sect-prelim1}
In this section, we briefly summarize the definitions and results from~\cite{CES-optimalLL,CES-thermalization} which are needed to characterize the deterministic approximation of $\langle T_{[1,k]}\rangle$.

\begin{definition}[Non-crossing partitions]\label{def-NCdisk}
Let $S$ be a finite (cyclically) ordered set of integers. We call a partition $\pi$ of the set $S$ \emph{crossing} if there exist blocks $B\neq B'$ in $\pi$ with $a,b\in B$, $c,d\in B'$, and $a<c<b<d$, otherwise we call it \emph{non-crossing}. The set of non-crossing partitions is denoted by $NCP(S)$ and we abbreviate $NCP(k):=NCP([k])$. For each non-crossing partition $\pi=\{B_1,\dots,B_n\}$, set $|\pi|:=n$ for the number of blocks in the partition.
\end{definition}

Recall that non-crossing partitions have an alternative geometrical definition: Arrange the elements of $S$ equidistantly in clockwise order on the circle and for each $\pi\in B$ consider the convex hull $P_B$ of the points $s\in B$. Then $\pi$ is non-crossing if and only if the polygons $\{P_B|B\in\pi\}$ are pair-wise disjoint. Because of this, we also call the elements of $NCP(k)$ \emph{disk non-crossing} to distinguish them from their annulus analog defined below. We further recall the definition of the Kreweras complement (see Fig.~\ref{fig-disknc} for an example).

\begin{definition}[Kreweras complement, disk case]\label{def-Kdisk}
Let $S\subset \N$ be a finite set of integers equidistantly arranged in clockwise order on the circle and label the midpoints of the arcs between the points $s\in S$ also by the elements of $S$. We arrange the new labels such that the arc $s$ follows the point $s$ in clockwise order. Let $\pi\in NCP(S)$. Then the \emph{(disk) Kreweras complement} of $\pi$, denoted by $K(\pi)$, is the element of $NCP(S)$ such that $r,s$ belong to the same block of $K(\pi)$ if and only if the arcs labeled $r,s$ are in the same connected component in the complement $D\setminus \cup_{B\in\pi}P_B$ of the polygons $\{P_B|B\in \pi\}$ in the labeled disk $D$.
\end{definition}

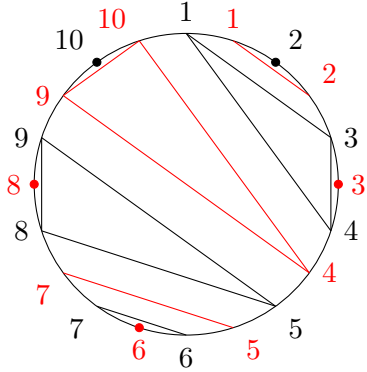
\begin{figure}[H]
\begin{center}
\begin{tikzpicture}[scale=2] 
\draw (0,1) node[above=1pt] {1};
\draw (0.309,0.951) node[above=1pt] {\color{red}1\color{black}};
\draw (0.5878,0.809) node[above right=1pt] {2};
\draw (0.809,0.5878) node[above right=1pt] {\color{red}2\color{black}};
\draw (0.951,0.309) node[right=1pt] {3};
\draw (1,0) node[right=1pt] {\color{red}3\color{black}};
\draw (0.951,-0.309) node[right=1pt] {4};
\draw (0.809,-0.5878) node[right=1pt] {\color{red}4\color{black}};
\draw (0.5878,-0.809) node[below right=1pt] {5};
\draw (0.309,-0.951) node[below right=1pt] {\color{red}5\color{black}};
\draw (0,-1) node[below=1pt] {6};
\draw (-0.309,-0.951) node[below=1pt] {\color{red}6\color{black}};
\draw (-0.5878,-0.809) node[below left=1pt] {7};
\draw (-0.809,-0.5878) node[below left=1pt] {\color{red}7\color{black}};
\draw (-0.951,-0.309) node[left=1pt] {8};
\draw (-1,0) node[left=1pt] {\color{red}8\color{black}};
\draw (-0.951,0.309) node[left=1pt] {9};
\draw (-0.809,0.5878) node[left=1pt] {\color{red}9\color{black}};
\draw (-0.5878,0.809) node[above left=1pt] {10};
\draw (-0.309,0.951) node[above left=1pt] {\color{red}10\color{black}};

\draw (0,0) circle (1cm);

\filldraw [black] (-0.5878,0.809) circle (0.75pt);
\filldraw [black] (0.5878,0.809) circle (0.75pt);

\draw[black] (0,1) -- (0.951,0.309);
\draw[black] (0.951,0.309) -- (0.951,-0.309);
\draw[black] (0.951,-0.309) -- (0,1);

\draw[black] (-0.951,0.309) -- (-0.951,-0.309);
\draw[black] (-0.951,-0.309) -- (0.5878,-0.809);
\draw[black] (-0.951,0.309) -- (0.5878,-0.809);

\draw[black] (-0.5878,-0.809) -- (0,-1);

\filldraw [red] (1,0) circle (0.75pt);
\filldraw [red] (-0.309,-0.951) circle (0.75pt);
\filldraw [red] (-1,0) circle (0.75pt);

\draw[red] (0.309,0.951) -- (0.809,0.5878);

\draw[red] (0.309,-0.951) -- (-0.809,-0.5878);

\draw[red] (0.809,-0.5878) -- (-0.809,0.5878);
\draw[red] (0.809,-0.5878) -- (-0.309,0.951);
\draw[red] (-0.809,0.5878) -- (-0.309,0.951);
\end{tikzpicture}
\end{center}
\captionof{figure}{The non-crossing partition $\pi=\{\{1,3,4\},\{2\},\{5,8,9\},\{6,7\},\{10\}\}$ (black) and its Kreweras complement $K(\pi)=\{\{1,2\},\{3\},\{4,9,10\},\{5,7\},\{6\},\{8\}\}$ (red).}\label{fig-disknc}
\end{figure}

Observe that $D\setminus \cup_{B\in\pi}P_B$ has $|S|-|\pi|+1$ connected components, hence $|\pi|+|K(\pi)|=|S|+1$. Further, $K^2=K\circ K$ recovers $\pi$ up to a rotation of $D$, i.e., $K^2(\pi)$ is the partition where for $S=\{s_1,\dots,s_k\}$ the elements in each block of $\pi$ are shifted by $s_1\mapsto s_2\mapsto\dots\mapsto s_k\mapsto s_1$. In particular, taking the Kreweras complement is invertible as a map on $NCP(S)$.

\begin{definition}[Free cumulant function]\label{def-circ}
Fix $k\in\N$, denote the power set of $[k]$ by $\cP([k])$ and let $f:\cP([k])\rightarrow\C$ be a function mapping subsets of $[k]$ to scalars. The \emph{(first-order) free cumulant function} of $f$ is the unique map $f_\circ: \cP([k])\rightarrow\C$ satisfying
\begin{equation}\label{eq-mcrelation1}
f[S]=\sum_{\pi\in NCP(S)}\prod_{B\in\pi}f_{\circ}[B]
\end{equation}
for any $S\subseteq[k]$.
\end{definition}
We emphasize that Definition~\ref{def-circ} does not require $f$ to have any particular symmetries. However, in the free probability literature, $f$ usually arises from tracial functional and is hence symmetric under the cyclic permutation of its entries (cf., e.g.,~\cite[Ch.~2]{MSBook}). The implicit relation in~\eqref{eq-mcrelation1} can be recursively turned into an explicit definition of $f_\circ$. Alternatively, we may also invert~\eqref{eq-mcrelation1} explicitly using the Möbius function associated with the lattice of non-crossing partitions. Recall that $NCP(S)$ is a lattice with respect to the refinement order, i.e., the partial order in which $\pi\leq\nu$ if and only if for each $B\in\pi$ there exists $B'\in\nu$ with $B\subset B'$. Moreover, there are unique maximal and minimal elements given by $0_S:=\{\{s\}|s\in S\}$ and $1_S:=\{S\}$, respectively. The free cumulant function can then be written as
\begin{align}
f_\circ[S]=\sum_{\pi\in NC[S]}\mu(\pi,\1_S)\prod_{B\in\pi}f[B],\quad \mu(\pi,\nu):=\begin{cases}1,\ &\pi=\nu,\\ -\sum_{\pi<\tau\leq\nu}\mu(\tau,\nu),\ &\pi<\nu,\end{cases}\label{eq-defMoebius}
\end{align}
using the Möbius function $\mu:\{(\pi,\nu)|\pi\leq\nu\in NCP(S)\}\rightarrow\Z$ that is recursively defined by~\eqref{eq-defMoebius}. We remark that $\mu(\pi,1_S)$ can be given in a closed form using the Catalan numbers~(see, e.g.,~\cite[Lem.~2.16]{CES-thermalization}).

\medskip
The following choice for the function $f$ is of particular interest.

\begin{definition}[Divided differences]\label{def-divdif}
For finite multi-sets $\{z_1,\dots, z_k\}\subset\C\setminus\R$ we recursively define
\begin{displaymath}
m[z_1,\dots,z_k]:=\frac{m[z_2,\dots,z_k]-m[z_1,\dots,z_{k-1}]}{z_k-z_1}
\end{displaymath}
whenever there are two distinct $z_1\neq z_k$ among $z_1,\dots,z_k$ and otherwise
\begin{displaymath}
m[\underbrace{z,\dots,z}_{k\text{ times}}]:=\frac{m^{(k-1)}(z)}{(k-1)!}
\end{displaymath}
where $m^{(k-1)}$ is the $(k-1)$-th derivative of the function $m$ in~\eqref{eq-defm}. Note that this is well-defined in the sense that $m[z_1,\dots,z_k]$ is independent of the ordering of the multi-set $\{z_1,\dots,z_k\}$. We abbreviate $m[1,\dots,k]:=m[z_1,\dots,z_k]$.
\end{definition}

We emphasize that $m[\cdot]$, and hence $m_{\circ}[\cdot]$, have full permutation symmetry, which is much more than what was assumed for $f$ in Definition~\ref{def-circ}. The following example illustrates the combinatorial formulas~\eqref{eq-mcrelation1} and~\eqref{eq-defMoebius} for $f=m[\cdot]$.

\begin{example}[First-order free cumulants]\label{ex-circ}
In the case $k=1$ we simply have $m[1]=m(z_1)$. For $k=2$, the only non-crossing partitions are $(12)$ and $(1)(2)$ such that
\begin{displaymath}
m_{\circ}[1,2]=m[1,2]-m_1m_2,\quad m_j:=m[j]=m[z_j]
\end{displaymath}
while for $k=3$ we have
\begin{displaymath}
m_{\circ}[1,2,3]=m[1,2,3]-m_1m[2,3]-m_2m[1,3]-m_3m[1,2]+2m_1m_2m_3.
\end{displaymath}
\end{example}

The quantities $m$ and $m_\circ$ were studied in detail in~\cite{CES-thermalization}, yielding a close connection to \emph{non-crossing graphs}. We recall the definition and give an example in Fig.~\ref{fig-NCG} below. These graphs are planar. For later convenience, we use a slightly more general notion of planar graphs throughout the paper than the standard literature by allowing for self-connections (loops) and multi-edges.

\begin{definition}[Disk non-crossing graphs]\label{def-NCG}
Let $S\subset\N$ be a finite (cyclically) ordered set of integers equidistantly arranged in clockwise order on the circle. We call an undirected planar graph $(S,E)$ on the vertex set $S$ without loops or multi-edges \emph{(disk) crossing} if there exist two edges $(a,b),(c,d)\in E$ with $a<c<b<d$, otherwise we call it \emph{(disk) non-crossing}\footnote{The edges of a disk non-crossing graph $\Gamma$ can be drawn in the interior of the disk without intersecting, i.e., $\Gamma$ is a planar graph drawn inside a labeled disk.}. The set of all \emph{(disk) non-crossing graphs} with vertex set $S$ is denoted by $NCG(S)$ and we denote the subset of connected graphs as $NCG_c(S)$. Whenever $S=[k]$, abbreviate $NCG(k):=NCG([k])$.
\end{definition}

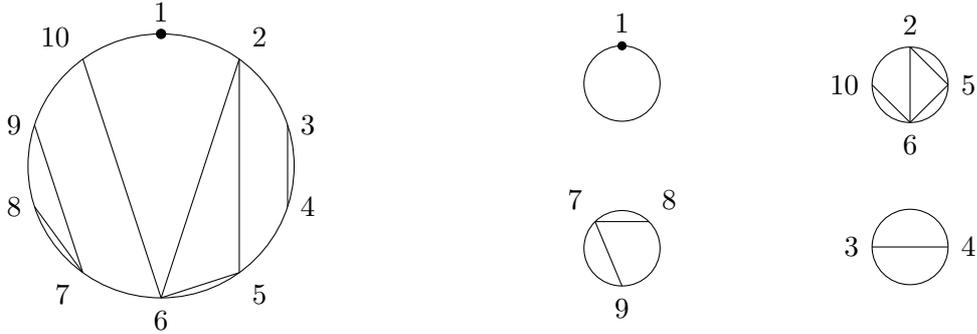
\begin{figure}[H]
\parbox{0.5\textwidth}{
\begin{center}
\begin{tikzpicture}[scale=1.75]
\draw (0,1) node[above=1pt] {1};
\draw (0.5878,0.809) node[above right=1pt] {2};
\draw (0.951,0.309) node[right=1pt] {3};
\draw (0.951,-0.309) node[right=1pt] {4};
\draw (0.5878,-0.809) node[below right=1pt] {5};
\draw (0,-1) node[below=1pt] {6};
\draw (-0.5878,-0.809) node[below left=1pt] {7};
\draw (-0.951,-0.309) node[left=1pt] {8};
\draw (-0.951,0.309) node[left=1pt] {9};
\draw (-0.5878,0.809) node[above left=1pt] {10};
\draw (0,0) circle (1cm);

\filldraw [black] (0,1) circle (1pt);

\draw[black] (0.951,-0.309) -- (0.951,0.309);

\draw[black] (-0.5878,-0.809) -- (-0.951,-0.309);
\draw[black] (-0.951,0.309) -- (-0.5878,-0.809);

\draw[black] (-0.5878,0.809) -- (0,-1);
\draw[black] (0,-1) -- (0.5878,-0.809);
\draw[black] (0,-1) -- (0.5878,0.809);
\draw[black] (0.5878,-0.809) -- (0.5878,0.809);
\end{tikzpicture}
\end{center}
}
\parbox{0.5\textwidth}{
\begin{center}
\begin{tikzpicture}[scale=0.5]
\draw (0,1) node[above=1pt] {1};
\draw (0,-1) node[below=1pt] {\color{white}blank\color{black}};
\draw (1,0) node[right=1pt] {\color{white}blank\color{black}};
\draw (-1,0) node[left=1pt] {\color{white}blank\color{black}};
\draw (0,0) circle (1cm);

\filldraw [black] (0,1) circle (3pt);
\end{tikzpicture}\hspace{0.75cm}
\begin{tikzpicture}[scale=0.5]
\draw (0,1) node[above=1pt] {2};
\draw (1,0) node[right=1pt] {5};
\draw (0,-1) node[below=1pt] {6};
\draw (-1,0) node[left=1pt] {10};
\draw (0,0) circle (1cm);

\draw[black] (-1,0) -- (0,-1);
\draw[black] (0,1) -- (0,-1);
\draw[black] (1,0) -- (0,-1);
\draw[black] (0,1) -- (1,0);
\end{tikzpicture}\\
\begin{tikzpicture}[scale=0.5]
\draw (-0.707,0.707) node[above left=1pt] {7};
\draw (0,-1) node[below=1pt] {9};
\draw (0.707,0.707) node[above right=1pt] {8};
\draw (0,1) node[above=1pt] {\color{white}blank\color{black}};
\draw (1,0) node[right=1pt] {\color{white}blank\color{black}};
\draw (-1,0) node[left=1pt] {\color{white}blank\color{black}};
\draw (0,0) circle (1cm);

\draw[black] (-0.707,0.707) -- (0,-1);
\draw[black] (-0.707,0.707) -- (0.707,0.707);
\end{tikzpicture}\hspace{0.75cm}
\begin{tikzpicture}[scale=0.5]
\draw (0,-1) node[below=1pt] {\color{white}blank\color{black}};
\draw (0,1) node[above=1pt] {\color{white}blank\color{black}};
\draw (1,0) node[right=1pt] {4};
\draw (-1,0) node[left=1pt] {\color{white}0\color{black}3};
\draw (0,0) circle (1cm);

\draw[black] (-1,0) -- (1,0);
\end{tikzpicture}
\end{center}
}
\captionof{figure}{An element of $NCG(10)$ and its connected subgraphs.}\label{fig-NCG}
\end{figure}

Emphasizing that Definition~\ref{def-NCG} lives on a disk is important, as we later introduce a non-crossing property on the annulus. Whenever both definitions are used together, we use the specifications \emph{disk} non-crossing and \emph{annular} non-crossing to distinguish the underlying geometry. By construction, every $\Gamma\in NCG(S)$ induces a non-crossing partition with blocks representing the vertices in the connected components of $\Gamma$. Further, any connected component of $\Gamma$ is itself a (disk) non-crossing graph.

\medskip
Lemma~5.2 of~\cite{CES-thermalization} proves the representations
\begin{align}
m[S]&=\Big(\prod_{s\in S}m_s\Big)\sum_{\Gamma\in NCG(S)}\prod_{(i,j)\in E(\Gamma)}q_{i,j},\label{eq-mgraphs}\\
m_{\circ}[S]&=\Big(\prod_{s\in S}m_s\Big)\sum_{\Gamma\in NCG_c(S)}\prod_{(i,j)\in E(\Gamma)}q_{i,j}\label{eq-mcircgraphs}
\end{align}
in terms of the \emph{weights} $q_{i,j}$ in~\eqref{eq-defq}. Here, $E(\Gamma)$ is the edge set of the graph $\Gamma$. Note that~\eqref{eq-mgraphs} and~\eqref{eq-mcircgraphs} are still well-defined if $S$ is an ordered multi-set, i.e., if some elements are repeated. In this case, we consider $NCG(|S|)$ instead of $NCG(S)$ and use the one-to-one correspondence between the (possibly repeated) labels $\{s|s\in S\}$ and $\{1,\dots,|S|\}$ to obtain a uniquely defined right-hand side.

\medskip
A key technical tool in the proof of our main results is the optimal multi-resolvent local law~\cite[Thm.~2.5]{CES-optimalLL}. As we only work on macroscopic scales, i.e., with $N$-independent spectral parameters, in the present paper, we state the result in the form of a global law and omit the dependence on $\eta_*=\min |\Im z_j|$. Recall the commonly used definition of stochastic domination.

\begin{definition}[Stochastic domination]
Let
\begin{displaymath}
X=\Big\{X^{(N)}(u)\Big|N\in\N,u\in U^{(N)}\Big\} \ \text{and}\ Y=\Big\{Y^{(N)}(u)\Big|N\in\N,u\in U^{(N)}\Big\}
\end{displaymath}
be two families of non-negative random variables that are indexed by $N$ and possibly some other parameter $u$. We say that $X$ is \emph{stochastically dominated} by $Y$, denoted by $X\prec Y$ or $X=\cO_{\prec}(Y)$, if, for all $\eps,C>0$ we have 
\begin{displaymath}
\sup_{u\in U^{(N)}}\P\Big(X^{(N)}(u)>N^\eps Y^{(N)}(u)\Big)\leq N^{-C}
\end{displaymath}
for large enough $N\geq N_0(\eps,C)$.
\end{definition}

\begin{theorem}[Macroscopic version of {\cite[Thm.~2.5]{CES-optimalLL}}]\label{thm-multiG-LL}
Fix $k\in \N$ and pick spectral parameters $z_1,\dots,z_k\in\C\setminus\R$ with $|\Im z_j|\gtrsim1$ and $\max_j|z_j|\leq N^{100}$ as well as deterministic matrices $A_1,\dots,A_k\in\C^{N\times N}$ with $\|A_i\|\lesssim 1$. Define\footnote{The noncommutative product $\prod_{j\in B}A_j$ for $B=(j_1,\dots,j_r)$ is defined as $A_{j_1}\dots A_{j_r}$ in the order inherited from $B\subseteq S$.}
\begin{equation}\label{eq-LLformula}
M_{[k]}:=\sum_{\pi\in NCP(k)}\Big(\prod_{\substack{B\in K(\pi),\\k\notin B}}\Big\langle \prod_{j\in B}A_j\Big\rangle\prod_{i\in B(k)\setminus\{k\}}A_i\Big)\Big(\prod_{B\in\pi}m_{\circ}[B]\Big),
\end{equation}
where $B(k)$ is the block in $K(\pi)$ that contains $k$. Recalling that $T_j=G_jA_j$, we have the averaged local law
\begin{equation}\label{eq-multiGaveraged}
\langle T_{[1,k]}\rangle=\langle M_{[k]}A_k\rangle+\cO_{\prec}\Big(\frac{1}{N}\Big),
\end{equation}
and for $\mathbf{x},\mathbf{y}\in \C^N$ with $\|\mathbf{x}\|,\|\mathbf{y}\|\lesssim 1$ we have the isotropic local law
\begin{equation}\label{eq-multiGisotropic}
\langle \mathbf{x},T_{[1,k\rangle}G_k\mathbf{y}\rangle=\langle\mathbf{x},M_{[k]}\mathbf{y}\rangle+\cO_{\prec}\Big(\frac{1}{\sqrt{N}}\Big).
\end{equation}
\end{theorem}
 
As we frequently encounter $\langle M_{[k]}A_k\rangle$ in the following sections, we introduce the notation
\begin{equation}\label{eq-defM}
\m[T_1,\dots,T_k]=\m[z_1,A_1,\dots,z_k,A_k]:=\langle M_{[k]}A_k\rangle.
\end{equation}
In particular,
\begin{equation}\label{eq-formulaM}
\m[T_1,\dots,T_k]=\sum_{\pi\in NCP(k)}\Big(\prod_{B\in K(\pi)}\Big\langle \prod_{j\in B}A_j\Big\rangle\Big)\Big(\prod_{B\in\pi}m_{\circ}[B]\Big)
\end{equation}
and we have $\m[G_1,\dots,G_k]=m[1,\dots,k]$ as a consequence of~\eqref{eq-mcrelation1}. We further apply~\eqref{eq-multiGaveraged} or~\eqref{eq-multiGisotropic} for a product $T_{s_1}\dots T_{s_{k-1}}G_{s_k}$ that is indexed by a (cyclically) ordered set $S=(s_1,\dots,s_k)$ instead of an interval. In this case, the deterministic approximation is denoted~as
\begin{displaymath}
M_S=M_{(s_1,\dots,s_k)}
\end{displaymath}
with the same definition as in~\eqref{eq-LLformula}.

\begin{remark}
The quantities $m[1,\dots,k]$, $m_{\circ}[1,\dots,k]$, $\m[T_1,\dots,T_k]$, $(M_{[k]})_{ij}$, and $\|M_{[k]}\|$ are of order one for any $k\in\N$ and $i,j\in[N]$ in the macroscopic regime (cf.~Lemma~2.4 and Appendix~A of~\cite{CES-optimalLL}). Theorem~\ref{thm-multiG-LL} asserts that the deterministic $M_{[k]}$ is the leading order approximation of $T_{[1,k\rangle}G_k$. In particular, the error terms in~\eqref{eq-multiGaveraged} and~\eqref{eq-multiGisotropic} are smaller than the natural upper bound on their leading term by a factor of $1/N$ and $1/\sqrt{N}$, respectively.
\end{remark}

We further need a generalization of the averaged local law~\eqref{eq-multiGaveraged} that includes transposes.

\begin{theorem}[Global law for resolvent chains with transposes]\label{thm-skewedLL}
Let $k\in\N$ and pick spectral parameters $z_1,\dots,z_k$ with $|\Im(z_j)|\gtrsim 1$ and $\max_j|z_j|\leq N^{100}$ as well as deterministic matrices $A_1,\dots,A_k$ with $\|A_j\|\lesssim1$. Moreover, let $G_j^\sharp$ denote either the resolvent $G_j=G(z_j)$ or its transpose $G_j^t$ and denote by $\#$ the binary vector that has a one in $j$-th position if $\smash{G_j^\sharp=G_j^t}$ and a zero otherwise. Then,
\begin{equation}\label{eq-defMsharp}
\langle G_1^\sharp A_1\dots G_k^\sharp A_k\rangle=\sum_{\pi\in NCP(k)}\Big(\prod_{B\in K(\pi)}\Big\langle \prod_{j\in B}A_j\Big\rangle\Big)\Big(\prod_{B\in\pi}m_{\circ}^{\#,\sigma}[B]\Big)+\cO_\prec\Big(\frac{1}{N}\Big)
\end{equation}
where $m_\circ^{\#,\sigma}[\cdot]$ denotes the free cumulants associated with the set function $m^{\#,\sigma}[\cdot]$ by Definition~\ref{def-circ}. Here, $m^{\#,\sigma}[\cdot]$ is defined to satisfy $m^{\#,\sigma}[\emptyset]=0$ as well as the recursion
\begin{align}
m^{\#,\sigma}[1,\dots,k]=m_1(1+q_{1,k}^\sharp)\Big(m^{\#,\sigma}[2,\dots,k]+\sum_{j=2}^kc_{1,j}m^{\#,\sigma}[1,\dots,j]m^{\#,\sigma}[j,\dots,k]\Big)\label{eq-mhashrecursion}
\end{align}
with $q_{1,k}^\sharp=q_{1,k}=\frac{m_1m_k}{1-m_1m_k}$ whenever $\#_1=\#_k$, i.e., either both $G_1$ and $G_k$ occur as transposes in the product $G_1^\sharp\dots G_k^\sharp$ or neither of them, and $q_{1,k}^\sharp=\frac{\sigma m_1m_k}{1-\sigma m_1m_k}$ otherwise. Similarly, $c_{1,j}=1$ whenever $\#_1=\#_j$ and $c_{1,j}=\sigma$ otherwise. Recall that $\sigma=\E\chi_{od}^2$ where $\chi_{od}$ is the real or complex random variable that specifies the distribution of the off-diagonal entries of the Wigner matrix $W$.
\end{theorem}
The proof of Theorem~\ref{thm-skewedLL} is, modulo careful bookkeeping of the transposes, similar to the proof of the averaged local law in~\cite[Thm.~3.4]{CES-thermalization}. For the convenience of the reader, a brief sketch of the argument is included in Appendix~\ref{app-skewedLLproof}. We remark that the same result may be obtained on mesoscopic scales with optimal error bounds following the strategy of~\cite{CES-optimalLL} (cf.~\cite[Rem.~2.2]{CES-optimalLL}) and that several examples in the cases $k\in\{2,3\}$ are considered in Proposition~3.4 and Remark 3.5 of~\cite{CES-ETH} as well as in Propositions~3.3 and~3.4 of~\cite{CES-functCLT}.

\medskip
Note that $\sigma=1$ implies that the matrix $W$ is real and its resolvent satisfies $G_j^t=G_j$. Hence, the statement of Theorem~\ref{thm-skewedLL} reduces to that of an averaged global law for real symmetric Wigner matrices in this case. Due to the structural similarity between~\eqref{eq-formulaM} and~\eqref{eq-defMsharp}, we will slightly abuse notation and write the right-hand side of~\eqref{eq-defMsharp} as
\begin{equation}\label{eq-mtransposes}
\m[G_1^\sharp A_1,\dots, G_k^\sharp A_k]:=\sum_{\pi\in NCP(k)}\Big(\prod_{B\in K(\pi)}\Big\langle \prod_{j\in B}A_j\Big\rangle\Big)\Big(\prod_{B\in\pi}m_{\circ}^{\#,\sigma}[B]\Big).
\end{equation}
Moreover, by Definition~\ref{def-circ}, we have $\m[G_1^\sharp,\dots,G_k^\sharp]=m^{\#,\sigma}[1,\dots,k]$ and~\eqref{eq-mhashrecursion} reduces to the divided differences in Definition~\ref{def-divdif} whenever $\#$ is the zero vector.

\subsection{Preliminaries Part 2: Second-Order Quantities}\label{sect-prelim2}
In this section, we give an overview of the definitions from free probability that are used in later sections as well as some related quantities appearing in the CLTs.

\medskip
Recall that the key picture for describing the expectation of $\langle T_{[1,k]}\rangle$ is a disk with the labels $1,\dots,k$ organized in clockwise order along its boundary. In a very similar spirit, the key picture for describing the corresponding second-order object, i.e., the covariance of $\langle T_{[1,k]}\rangle$ and $\langle T_{[k+1,k+\ell]}\rangle$, consists of two concentric labeled circles. Let $k,\ell\in\N$ and arrange the numbers $1,\dots,k$ equidistantly in clockwise order on the outer circle and the numbers ${k+1,\dots,k+\ell}$ equidistantly in counter-clockwise order on the inner circle. We refer to the planar domain between these two circles together with the labeled points on its boundary as \emph{the $(k,\ell)$-annulus} (see Fig.~\ref{fig-klannulus} below). The labeled points will often serve as vertices of a graph. In this case, any edges connecting two points are drawn inside the annulus.

\begin{figure}[H]
\begin{center}
\begin{tikzpicture}[scale=1.25]
\draw (0,1) node[above=1pt] {1};
\filldraw [black] (0,1) circle (1pt);
\draw (1,0) node[right=1pt] {2};
\filldraw [black] (1,0) circle (1pt);
\draw (0,-1) node[below=1pt] {3};
\filldraw [black] (0,-1) circle (1pt);
\draw (-1,0) node[left=1pt] {4};
\filldraw [black] (-1,0) circle (1pt);

\draw (0.3536,0.3536) node[below left=0.3536pt] {5};
\filldraw [black] (0.3536,0.3536) circle (1pt);
\draw (-0.3536,0.3536) node[below right=0.3536pt] {6};
\filldraw [black] (-0.3536,0.3536) circle (1pt);
\draw (0,-0.5) node[above=0.75pt] {7};
\filldraw [black] (0,-0.5) circle (1pt);

\draw (0,0) circle (1cm);
\draw (0,0) circle (0.5cm);
\end{tikzpicture}
\end{center}
\captionof{figure}{The labels of the $(4,3)$-annulus.}\label{fig-klannulus}
\end{figure}
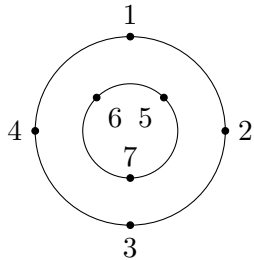

\begin{definition}[Annular non-crossing permutations]\label{def-NCA}
Let $k,\ell\in\N$. We call a permutation of $[k+\ell]$ an \emph{annular non-crossing permutation} if we can draw its cycles\footnote{Recall that every permutation has a unique cycle decomposition. We represent a cycle $(abc\dots x)$ by an oriented graph with edge set $\{(a,b),(b,c),\dots,(x,a)\}$.} on the $(k,\ell)$-annulus such that the following conditions (see~\cite[Def.~5 in Ch.~5]{MSBook}) are satisfied:
\begin{itemize}
\item[(i)] Non-crossing property: The cycles do not cross.
\item[(ii)] Standardness: Each cycle encloses a region in the annulus that is homeomorphic to the disk with boundary oriented clockwise (in particular, the cycles follow the orientation of the numbering of the circles).
\item[(iii)] Connectedness: At least one cycle connects both circles.
\end{itemize}
The set of annular non-crossing permutations is denoted by $\NCA(k,\ell)$. Any cycle that connects both circles is referred to as \emph{connecting cycle}.
\end{definition}
We remark that $\NCA(k,\ell)$ can be fully characterized by the avoidance of certain crossing patterns (cf. analogous geometric characterization of $NCP(k)$ below Definition~\ref{def-NCdisk}) and an algebraic analog of the standardness condition. This equivalent definition is discussed, e.g., in~\cite[Sect.~3]{MingoNica2004}, but we will not use it here.

\begin{definition}[Annular non-crossing partitions]
Let $k,\ell\in\N$. We call the partitions induced by the cycles of $\NCA(k,\ell)$ \emph{annular non-crossing partitions}. The set of annular non-crossing partitions is denoted by $NCP(k,\ell)$. A block that arises from a connecting cycle is referred to as \emph{connecting block}.
\end{definition}

While there is a one-to-one correspondence between the non-crossing partitions of the disk in Definition~\ref{def-NCdisk} and the permutations of $[k]$ avoiding the same crossing pattern, there is a crucial difference between non-crossing partitions and permutations on the $(k,\ell)$-annulus. In particular, there is no bijective mapping between a permutation in $\NCA(k,\ell)$ and the partition of $[k+\ell]$ induced by its cycles, as, e.g., both permutations $(123)$ and $(132)$ correspond to the partition $\{\{1,2,3\}\}$, but give rise to different pictures due to the orientation induced by Definition~\ref{def-NCA}(ii) (see Fig.~\ref{fig-NCAvsNC} below). In general, a permutation always uniquely determines the underlying partition, but a partition can be obtained from more than one permutation. This happens if and only if there is exactly one connecting block (cf.~\cite[Sect.~4]{MingoNica2004}).

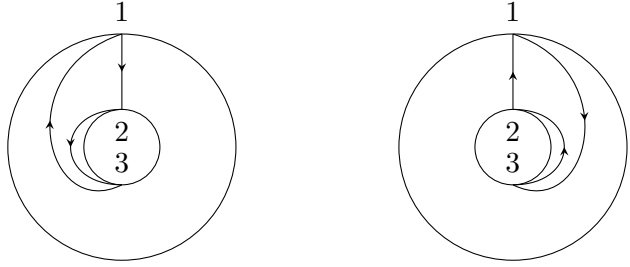
\begin{figure}[H]
\begin{center}
\begin{tikzpicture}[scale=1] 
\draw (0,1.5) node[above=1pt] {1};

\draw (0,0.5) node[below=1pt] {2};
\draw (0,-0.5) node[above=1pt] {3};

\draw (0,0) circle (1.5cm);
\draw (0,0) circle (0.5cm);

\draw[black,postaction={on each segment={mid arrow=black}}] (0,1.5) -- (0,0.5);
\draw[black,postaction={on each segment={mid arrow=black}}] (0,0.5) .. controls (-0.9,0.5) and (-0.9,-0.5) .. (0,-0.5);
\draw[black,postaction={on each segment={mid arrow=black}}] (0,-0.5) .. controls (-1,-1) and (-1.5,1) .. (0,1.5);
\end{tikzpicture}\hspace{2cm}
\begin{tikzpicture}[scale=1] 
\draw (0,1.5) node[above=1pt] {1};

\draw (0,0.5) node[below=1pt] {2};
\draw (0,-0.5) node[above=1pt] {3};

\draw (0,0) circle (1.5cm);
\draw (0,0) circle (0.5cm);

\draw[black,postaction={on each segment={mid arrow=black}}] (0,1.5) .. controls (1.5,1) and (1,-1) .. (0,-0.5);
\draw[black,postaction={on each segment={mid arrow=black}}] (0,-0.5) .. controls (0.9,-0.5) and (0.9,0.5) .. (0,0.5);
\draw[black,postaction={on each segment={mid arrow=black}}] (0,0.5) -- (0,1.5);
\end{tikzpicture}
\end{center}
\captionof{figure}{The non-crossing permutations $(123)$ and $(132)$ on the $(1,2)$-annulus. They are different as permutations, but their cycles induce the same non-crossing partition.}\label{fig-NCAvsNC}
\end{figure}

Lastly, we consider partitions arising from permutations that respect the non-crossing property and standardness condition but do not have a connecting cycle. In this case, we may consider the permutation restricted to each circle separately, i.e., as an element of $NCP(k)\times NCP(\ell)$, and introduce an artificial connection by \emph{marking} one block on each circle.

\begin{definition}[Marked non-crossing partition]
Consider $\pi\in NCP(k)\times NCP(\ell)$ that naturally splits into $\pi=\pi_1\times \pi_2$ with $\pi_1\in NCP(k)$, $\pi_2\in NCP(\ell)$. We pick one block of $\pi_1$ and one of $\pi_2$, respectively, and mark them by underlining. The resulting object is referred to as a \emph{marked non-crossing partition}.
\end{definition}

Marking a block on each circle allows us to artificially introduce a connecting block by considering the union of the two marked blocks. As a consequence, any marked non-crossing partition can be associated with a unique element of $NCP(k,\ell)$. We further note that there are $|\pi_1|\cdot|\pi_2|$ possibilities to mark the blocks of $\pi=\pi_1\times\pi_2\in NCP(k)\times NCP(\ell)$. For example, $\{\{\underline{1}\},\{2\}\}\times\{\{\underline{3}\}\}$ and $\{\{1\},\{\underline{2}\}\}\times\{\{\underline{3}\}\}$ are considered different marked partitions although both arise from $\{\{1\},\{2\}\}\times\{\{3\}\}\in NCP(2)\times NCP(1)$ (see Fig.~\ref{fig-marked} below).

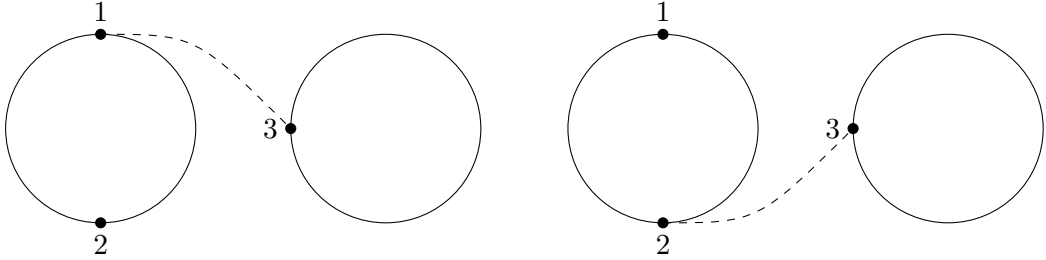
\begin{figure}[H]
\begin{center}
\begin{tikzpicture}[scale=1.25]
\draw (0,1) node[above=1pt] {1};
\filldraw [black] (0,1) circle (1.5pt);
\draw (0,-1) node[below=1pt] {2};
\filldraw [black] (0,-1) circle (1.5pt);
\draw (0,0) circle (1.0);

\draw (2,0) node[left=1pt] {3};
\filldraw [black] (2,0) circle (1.5pt);
\draw (3,0) circle (1.0);

\draw[dashed] (0,1) .. controls (1,1) .. (2,0);
\end{tikzpicture}\hspace{1cm}
\begin{tikzpicture}[scale=1.25]
\draw (0,1) node[above=1pt] {1};
\filldraw [black] (0,1) circle (1.5pt);
\draw (0,-1) node[below=1pt] {2};
\filldraw [black] (0,-1) circle (1.5pt);
\draw (0,0) circle (1.0);

\draw (2,0) node[left=1pt] {3};
\filldraw [black] (2,0) circle (1.5pt);
\draw (3,0) circle (1.0);

\draw[dashed] (0,-1) .. controls (1,-1) .. (2,0);
\end{tikzpicture}
\end{center}
\captionof{figure}{A visualization of $\{\{\underline{1}\},\{2\}\}\times\{\{\underline{3}\}\}$ and $\{\{1\},\{\underline{2}\}\}\times\{\{\underline{3}\}\}$. Marking is indicated by a dashed line.}\label{fig-marked}
\end{figure}

We further recall from~\cite[Ch.~5]{MSBook} that there is a second-order analog of Definition~\ref{def-circ}.

\begin{definition}[Second-order free cumulant function]\label{def-circcirc}
Let $f$ be a function mapping tuples $(S_1,S_2)$ of two finite (cyclically) ordered sets of integers to scalars. Assume further that $f$ is symmetric under the interchanging of its two arguments, i.e., that $f[S_1|S_2]=f[S_2|S_1]$, and cyclic in the sense that $f[S_1|\{s_1,\dots,s_j\}]=f[S_1|\{s_2,\dots,s_j,s_1\}$. Moreover, we assume that $f[S_1|\emptyset]=f[\emptyset|S_2]=0$, i.e., $f$ vanishes if either argument is the empty set. We implicitly define the \emph{second-order free cumulant function} of $f$ as the unique map $f_{\circ\circ}$ defined on pairs of finite (cyclically) ordered sets $(U_1,U_2)$ that satisfies
\begin{align}
f[S_1|S_2]&=\sum_{\pi\in \NCA(|S_1|,|S_2|)}\prod_{B\in\pi}f_{\circ}[B]+\sum_{\substack{\pi_1\times\pi_2\in NCP(|S_1|)\times NCP(|S_2|),\\ U_1\in\pi_1,U_2\in\pi_2\text{ marked}}} f_{\circ\circ}[U_1|U_2]\prod_{\substack{B\in \pi_1\setminus U_1\\ \cup\pi_2\setminus U_2}}f_{\circ}[B]\label{eq-mcrelation2}
\end{align}
for any finite $S_1,S_2$. Here, $f_\circ$ is the first-order free cumulant function introduced in Definition~\ref{def-circ}.
\end{definition}
Note that we use a set function $f$ that is symmetric under the interchanging of its arguments instead of its skew-symmetric version $f[S_1|S_2]=\overline{f[S_2|S_1]}$ typically used in the free probability literature to mimic the covariance functional (cf.~\cite[Ch.~5]{MSBook}). This choice will simplify the computations by reducing the number of complex conjugates arising in the intermediate steps.

\medskip
Similar to~\eqref{eq-mcrelation1}, the implicit relation~\eqref{eq-mcrelation2} may be turned into an explicit definition of $f_{\circ\circ}$ by recursion. Note that the term $f_{\circ\circ}[[k]\,|\, [k+1,k+\ell]]$ in formula~\eqref{eq-mcrelation2} with $f[[k]\,|\,[k+1,k+\ell]]$ on the left-hand side only occurs for the marked partition $\{\{\underline{1,\dots,k}\}\}\times\{\{\underline{k+1,\dots,k+\ell}\}\}$ and hence always has coefficient one, so we can express it in terms of $f$, $f_\circ$, and the previously identified values of $f_{\circ\circ}$. This shows that $f_{\circ\circ}$ is well-defined. Although we will not rely on Möbius inversion to express $f_{\circ\circ}$, we remark that it is possible to include both $\NCA(k,\ell)$ and the marked elements of $NCP(k)\times NCP(\ell)$ into one common definition, the \emph{non-crossing partitioned permutations}, which can be endowed with a partial ordering and hence render it suitable for Möbius inversion. This would allow to rewrite the right-hand side of~\eqref{eq-mcrelation2} to a structure similar to~\eqref{eq-mcrelation1} and obtain a closed formula similar to~\eqref{eq-defMoebius}. We refer to Sections~4 and~5 of~\cite{CollinsMingoSniadySpeicher2007} for the full construction.

\medskip
Similar to~\eqref{eq-mcrelation1} above, the relation~\eqref{eq-mcrelation2} is applied for one particular choice of $f$ built up from the Stieltjes transform $m(z)$. We will later see that the set function $m[\cdot|\cdot]$ defined below arises as the deterministic approximation of the (appropriately scaled) covariance of $\langle G_{[1,k]}\rangle$ and $\langle G_{[k+1,k+\ell]}\rangle$ in a similar way that the divided differences $m[\cdot]$ arise for the expectation of $\langle G_{[1,k]}\rangle$. In particular, $m[\cdot|\cdot]$ satisfies the symmetry and cyclicity assumption in Definition~\ref{def-circcirc}. We give a recursive definition of $m[\cdot|\cdot]$ for now, however, closed formulas are later obtained in Section~\ref{sect-formulas}.

\begin{definition}\label{def-m}
Let $S_1=(z_1,\dots,z_{k'})\subset\C\setminus\R$ and $S_2=(z_{k'+1},\dots,z_{k'+\ell'})\subset\C\setminus\R$ be two finite ordered multi-sets. We define $m[\cdot|\cdot]$ to be the set function taking values in~$\C$ with the properties (i)-(iii) listed below. Similar to $m[\cdot]$ in Definition~\ref{def-divdif}, we interpret $m[\cdot|\cdot]$ as a function of the indices of the spectral parameters.
\begin{itemize}
\item[(i)] Symmetry: $m[\cdot|\cdot]$ is symmetric under the interchanging of its arguments, i.e., for any sets $B_1\subseteq S_1,B_2\subseteq S_2$ we have
\begin{displaymath}
m[(i,z_i\in B_1)|(j,z_j\in B_2)]=m[(j,z_j\in B_2)|(i,z_i\in B_1)].
\end{displaymath}
\item[(ii)] Initial condition: For any sets $B_1\subseteq S_1,B_2\subseteq S_2$ we have
\begin{equation}\label{eq-minitial}
m[(i,z_i\in B_1)|\emptyset]=m[\emptyset|(j,j\in B_2)]=0.
\end{equation}
\item[(iii)] Recursion: Let $B_1\subseteq S_1$ and $B_2\subseteq S_2$ be ordered subsets with $|B_1|=k\leq k'$ and $|B_2|=\ell\leq\ell'$ elements, respectively. For simplicity, we index them by $[k]$ and $[k+1,k+\ell]$. The function $m[\cdot|\cdot]$ satisfies the following linear recursion
\begin{align}
&m[1,\dots,k|k+1,\dots,k+\ell]\NN\\
&=\frac{m_1}{1-m_1m_k}\Bigg(m[2,\dots,k|k+1,\dots,k+\ell]\NN\\
&\quad+\sum_{j=1}^{k-1}m[1,\dots,j|k+1,\dots,k+\ell]m[j,\dots,k]\NN\\
&\quad+\sum_{j=2}^km[1,\dots,j]m[j,\dots,k|k+1,\dots,k+\ell]+s_{GUE}+s_\kappa+s_\sigma+s_\omega\Bigg)\label{eq-mrecursion}
\end{align}
where the source terms in the last line are given by
\begin{align*}
s_{GUE}&:=\sum_{j=1}^\ell m[1,\dots,k,k+j,\dots,k+\ell,k+1,\dots,k+j]\NN\\
s_{\kappa}&:=\kappa_4\sum_{r=1}^k\sum_{k+1\leq s\leq t\leq k+\ell}m[1,\dots,r]m[r,\dots,k]m[s,\dots,t]m[t,\dots,k+\ell,k+1,\dots, s]\NN\\
s_{\sigma}&:=\sigma\sum_{j=1}^\ell m^{\#,\sigma}[1,\dots,k,k+j,\dots,k+\ell,k+1,\dots,k+j]\NN\\
s_\omega&:=\widetilde{\omega_2}\sum_{j=1}^\ell m[1,\dots,k]m[k+j,\dots,k+\ell,k+1,\dots k+j].\NN
\end{align*}
Here, we wrote out the underlying multi-set in the definition of $s_\kappa$ to indicate that it evaluates to $m[s,s]$ instead of $m_s$ if $t=s$ and the vector $\#\in\{0,1\}^{k+\ell+1}$ is given by $\#_1=\dots=\#_k=0$ and $\#_{k+1}=\dots=\#_{k+\ell+1}=1$. Recall that $m[\cdot]$ denotes the divided differences as introduced in Definition~\ref{def-divdif} and that $m^{\#,\sigma}[\cdot]$ was introduced in Theorem~\ref{thm-skewedLL}.
\end{itemize}
\end{definition}

Note that the recursion for $m[\cdot|\cdot]$ is linear with different types of source terms in the last line of~\eqref{eq-mrecursion}. Therefore, we may introduce the decomposition
\begin{equation}\label{eq-mdecomp}
m[\cdot|\cdot]=m_{GUE}[\cdot|\cdot]+\kappa_4m_{\kappa}[\cdot|\cdot]+\sigma m_{\sigma}[\cdot|\cdot]+\widetilde{\omega}_2m_{\omega}[\cdot|\cdot]
\end{equation}
where $m_{GUE}[\cdot|\cdot]$ satisifes~\eqref{eq-mrecursion} for $\kappa_4=\sigma=\widetilde{\omega}_2=0$ and $\kappa_4m_{\kappa}[\cdot|\cdot]$, $\sigma m_\sigma[\cdot|\cdot]$, and $\widetilde{\omega}_2m_\omega[\cdot|\cdot]$ satisfy~\eqref{eq-mrecursion} with $s_{\kappa}$, $s_{\sigma}$, and $s_{\omega}$ as the only source term, respectively.

\medskip
Note that the right-hand side of \eqref{eq-mrecursion} only contains divided differences and $m[B_1|B_2]$ for $|B_1|+|B_2|<k+\ell$, so~\eqref{eq-mrecursion} indeed defines $m[\cdot|\cdot]$ recursively. The symmetry assumption in~(i) then extends~\eqref{eq-mrecursion} to the second entry of $m[\cdot|\cdot]$. Moreover, all source terms in the last line of~\eqref{eq-mrecursion} are fully expressable as a function of $m_1,\dots,m_{k+\ell}$ by~\eqref{eq-mgraphs}, making $m[\cdot|\cdot]$ eventually a function of $m_1,\dots,m_{k+\ell}$ as well.

\medskip
As an example, setting $\sigma=\widetilde{\omega}_2=0$ and applying the recursion once gives
\begin{align}
m[1|2]&=\frac{m_1^2m_2^2}{(1-m_1^2)(1-m_2^2)(1-m_1m_2)^2}+\kappa_4\frac{m_1^3m_2^3}{(1-m_1^2)(1-m_2^2)}\NN\\
&=\frac{m_1'm_2'}{(1-m_1m_2)^2}+\kappa_4m_1m_1'm_2m_2'\label{eq-msmallest}
\end{align}
with $m_i'=m'(z_i)$. We remark that $m_{GUE}[1|2]$, seen as a function of $(z_1,z_2)$, is sometimes referred to as the second-order Cauchy transform of the GUE ensemble in the free probability literature (cf.~\cite{DiazMingo2022}). The corresponding first-order object is $-m(z)$, which is obtained by applying the usual Cauchy transform to the semicircle law.\footnote{The Cauchy transform of a probability measure $\mu$ is given by $c(z)=\int_{\R}\frac{\mu(\dx x)}{z-x}$ and hence only differs from the Stieltjes transform by a sign.}

\medskip
We consider another special case in the following example.
\begin{example}
Whenever $\kappa_4=\sigma=\widetilde{\omega}_2=0$ and one argument of $m_{GUE}[\cdot|\cdot]$ is a singleton set, only the fourth line of~\eqref{eq-mrecursion} gives a non-zero contribution. Rewriting this term using~\eqref{eq-mgraphs} (cf. Lemma~\ref{lem-doubledindex} below) yields the closed formula
\begin{align}
m_{GUE}[1|2,\dots,\ell+1]&=\frac{m_1}{1-m_1^2}\sum_{j=2}^{\ell+1} m[1,\dots,j,\dots,\ell+1,j]\label{eq-msingleton}\\
&=\Big(\prod_{s=1}^{\ell+1} m_s\Big)\Big(\sum_{\Gamma\in NCG([\ell+1])}\prod_{(a,b)\in E(\Gamma)}q_{a,b}\Big)\sum_{j=2}^{\ell+1} \frac{m_1'}{m_1}\frac{m_j'}{m_j}\Big(1+\sum_{i\neq j}q_{i,j}\Big).\NN
\end{align}
For $\ell=2$, we hence obtain
\begin{displaymath}
m_{GUE}[1|2,3]=\frac{m_1'(m_2'm_3(1-m_1m_3)+m_2m_3'(1-m_1m_2))}{(1-m_1m_2)^2(1-m_1m_3)^2(1-m_2m_3)}.
\end{displaymath}
Note that the right-hand side of~\eqref{eq-msingleton} is fully expressed terms of non-crossing graphs on a labeled disk. This is because $\NCA(1,\ell)$ and marked elements of $NCP(1)\times NCP(\ell)$ can be reduced to disk non-crossing partitions in this special case. In particular, the orientation of the circles is not relevant for this example.
\end{example}

The proof of~\eqref{eq-msingleton} is immediate from the following combinatorial lemma which may be of independent interest. We give its proof in Appendix~\ref{app-doubledindex}.

\begin{lemma}\label{lem-doubledindex}
For $j\in\{1,\dots,k\}$, $k\geq1$, we have
\begin{equation}\label{eq-doubledindex}
m[1,\dots,j,\dots,k,j]=m[1,\dots,k]\Big(1+\sum_{l\in [k]\setminus \{j\}}q_{j,l}\Big)\frac{m_j'}{m_j}
\end{equation}
with $q_{i,j}$ as in~\eqref{eq-defq}.
\end{lemma}

We also give some examples to illustrate the combinatorial formula~\eqref{eq-mcrelation2} for the choice $f[\cdot|\cdot]=m_{GUE}[\cdot|\cdot]$.

\begin{example}[Second-order free cumulants]\label{ex-circcirc}
Let $\kappa_4=\sigma=\widetilde{\omega}_2=0$. In the case $k=\ell=1$, the only non-crossing annular permutation is $(12)$ and there is also only one option for marking $\{\{1\}\}\times\{\{2\}\}$, namely $\smash{\{\{\underline{1}\}\}\times\{\{\underline{2}\}\}}$. Rearranging~\eqref{eq-mcrelation2}, we thus get
\begin{displaymath}
m_{\circ\circ}[1|2]=m[1|2]-m[1,2]+m_1m_2.
\end{displaymath}
Similarly, considering $k=1$ and $\ell=2$ yields
\begin{align*}
m_{\circ\circ}[1|2,3]&=m[1|2,3]-m[1|2]m_3-m[1|3]m_2-2m[1,2,3]\\
&\quad+2m_1m[2,3]+2m_2m[1,3]+2m_3m[1,2]-4m_1m_2m_3.
\end{align*}
In the case $k=\ell=2$, there are 18 distinct non-crossing annular permutations (see Fig.~\ref{fig-NCAcollection} below) and 4 elements in $NCP(2)\times NCP(2)$. However, the second sum in~\eqref{eq-mcrelation2} consists of 9 terms in total due to the marking of the blocks as, e.g., $\smash{\{\{\underline{1}\},\{2\}\}\times\{\{\underline{3,4}\}\}}$ and $\smash{\{\{1\},\{\underline{2}\}\}\times\{\{\underline{3,4}\}\}}$ correspond to $m_{\circ\circ}[1|3,4]m_2$ and $m_{\circ\circ}[2|3,4]m_1$, respectively, which do not need to coincide. In total, the formula defining $m_{\circ\circ}[1,2|3,4]$ has 27 terms on the right-hand side of~\eqref{eq-mcrelation2}.
\end{example}

\begin{figure}[H]
\begin{center}\footnotesize
\begin{tikzpicture}[scale=0.75] 
\draw (0,1.5) node[above=1pt] {1};
\draw (0,-1.5) node[below=1pt] {2};

\draw (0,0.5) node[below=1pt] {3};
\draw (0,-0.5) node[above=1pt] {4};

\draw (0,0) circle (1.5cm);
\draw (0,0) circle (0.5cm);

\draw[black,postaction={on each segment={mid arrow=black}}] (0,1.5) -- (0,0.5);
\draw[black,postaction={on each segment={mid arrow=black}}] (0,0.5) .. controls (-0.5,0.75) .. (0,1.5);
\draw[black,postaction={on each segment={mid arrow=black}}] (0,-0.5) .. controls (0.5,-1) and (-0.5,-1) .. (0,-0.5); 
\draw[black,postaction={on each segment={mid arrow=black}}] (0,-1.5) .. controls (-0.5,-1) and (0.5,-1) .. (0,-1.5); 
\end{tikzpicture}\hspace{0.5cm}
\begin{tikzpicture}[scale=0.75] 
\draw (0,1.5) node[above=1pt] {1};
\draw (0,-1.5) node[below=1pt] {2};

\draw (0,0.5) node[below=1pt] {3};
\draw (0,-0.5) node[above=1pt] {4};

\draw (0,0) circle (1.5cm);
\draw (0,0) circle (0.5cm);

\draw[black,postaction={on each segment={mid arrow=black}}] (0,1.5) .. controls (-1,0.5) and (-0.75,-0.5) .. (0,-0.5);
\draw[black,postaction={on each segment={mid arrow=black}}] (0,-0.5) .. controls (-1,-1) and (-1.5,1) .. (0,1.5);
\draw[black,postaction={on each segment={mid arrow=black}}] (0,-1.5) .. controls (-0.5,-1) and (0.5,-1) .. (0,-1.5); 
\draw[black,postaction={on each segment={mid arrow=black}}] (0,0.5) .. controls (-0.5,1) and (0.5,1) .. (0,0.5); 
\end{tikzpicture}\hspace{0.5cm}
\begin{tikzpicture}[scale=0.75] 
\draw (0,1.5) node[above=1pt] {1};
\draw (0,-1.5) node[below=1pt] {2};

\draw (0,0.5) node[below=1pt] {3};
\draw (0,-0.5) node[above=1pt] {4};

\draw (0,0) circle (1.5cm);
\draw (0,0) circle (0.5cm);

\draw[black,postaction={on each segment={mid arrow=black}}] (0,1.5) .. controls (0.5,1) and (-0.5,1) .. (0,1.5); 
\draw[black,postaction={on each segment={mid arrow=black}}] (0,-1.5) .. controls (-1.5,-1) and (-1,1) .. (0,0.5);
\draw[black,postaction={on each segment={mid arrow=black}}] (0,0.5) .. controls (-0.75,0.5) and (-1,-0.5) .. (0,-1.5);
\draw[black,postaction={on each segment={mid arrow=black}}] (0,-0.5) .. controls (0.5,-1) and (-0.5,-1) .. (0,-0.5); 
\end{tikzpicture}\hspace{0.5cm}
\begin{tikzpicture}[scale=0.75] 
\draw (0,1.5) node[above=1pt] {1};
\draw (0,-1.5) node[below=1pt] {2};

\draw (0,0.5) node[below=1pt] {3};
\draw (0,-0.5) node[above=1pt] {4};

\draw (0,0) circle (1.5cm);
\draw (0,0) circle (0.5cm);
\draw[black,postaction={on each segment={mid arrow=black}}] (0,1.5) .. controls (0.5,1) and (-0.5,1) .. (0,1.5); 
\draw[black,postaction={on each segment={mid arrow=black}}] (0,-1.5) -- (0,-0.5);
\draw[black,postaction={on each segment={mid arrow=black}}] (0,-0.5) .. controls (0.5,-0.75) .. (0,-1.5);
\draw[black,postaction={on each segment={mid arrow=black}}] (0,0.5) .. controls (-0.5,1) and (0.5,1) .. (0,0.5); 
\end{tikzpicture}\hspace{0.5cm}
\begin{tikzpicture}[scale=0.75] 
\draw (0,1.5) node[above=1pt] {1};
\draw (0,-1.5) node[below=1pt] {2};

\draw (0,0.5) node[below=1pt] {3};
\draw (0,-0.5) node[above=1pt] {4};

\draw (0,0) circle (1.5cm);
\draw (0,0) circle (0.5cm);

\draw[black,postaction={on each segment={mid arrow=black}}] (0,1.5) -- (0,0.5);
\draw[black,postaction={on each segment={mid arrow=black}}] (0,0.5) .. controls (-0.5,0.75) .. (0,1.5);
\draw[black,postaction={on each segment={mid arrow=black}}] (0,-1.5) -- (0,-0.5);
\draw[black,postaction={on each segment={mid arrow=black}}] (0,-0.5) .. controls (0.5,-0.75) .. (0,-1.5);
\end{tikzpicture}\hspace{0.5cm}
\begin{tikzpicture}[scale=0.75] 
\draw (0,1.5) node[above=1pt] {1};
\draw (0,-1.5) node[below=1pt] {2};

\draw (0,0.5) node[below=1pt] {3};
\draw (0,-0.5) node[above=1pt] {4};

\draw (0,0) circle (1.5cm);
\draw (0,0) circle (0.5cm);

\draw[black,postaction={on each segment={mid arrow=black}}] (0,-0.5) .. controls (0.75,-0.5) and (1,0.5) .. (0,1.5);
\draw[black,postaction={on each segment={mid arrow=black}}] (0,1.5) .. controls (1.5,1) and (1,-1) .. (0,-0.5);
\draw[black,postaction={on each segment={mid arrow=black}}] (0,-1.5) .. controls (-1.5,-1) and (-1,1) .. (0,0.5);
\draw[black,postaction={on each segment={mid arrow=black}}] (0,0.5) .. controls (-0.75,0.5) and (-1,-0.5) .. (0,-1.5);
\end{tikzpicture}\hspace{0.5cm}
\begin{tikzpicture}[scale=0.75] 
\draw (0,1.5) node[above=1pt] {1};
\draw (0,-1.5) node[below=1pt] {2};

\draw (0,0.5) node[below=1pt] {3};
\draw (0,-0.5) node[above=1pt] {4};

\draw (0,0) circle (1.5cm);
\draw (0,0) circle (0.5cm);

\draw[black,postaction={on each segment={mid arrow=black}}] (0,1.5) .. controls (1.5,0.5) and (1.5,-0.5) .. (0,-1.5);
\draw[black,postaction={on each segment={mid arrow=black}}] (0,-1.5) .. controls (1,-0.5) and (0.75,0.5) .. (0,0.5);
\draw[black,postaction={on each segment={mid arrow=black}}] (0,0.5) -- (0,1.5);
\draw[black,postaction={on each segment={mid arrow=black}}] (0,-0.5) .. controls (0.5,-1) and (-0.5,-1) .. (0,-0.5); 
\end{tikzpicture}
\hspace{0.5cm}
\begin{tikzpicture}[scale=0.75] 
\draw (0,1.5) node[above=1pt] {1};
\draw (0,-1.5) node[below=1pt] {2};

\draw (0,0.5) node[below=1pt] {3};
\draw (0,-0.5) node[above=1pt] {4};

\draw (0,0) circle (1.5cm);
\draw (0,0) circle (0.5cm);

\draw[black,postaction={on each segment={mid arrow=black}}] (0,1.5) -- (0,0.5);
\draw[black,postaction={on each segment={mid arrow=black}}] (0,0.5) .. controls (-0.75,0.5) and (-1,-0.5) .. (0,-1.5);
\draw[black,postaction={on each segment={mid arrow=black}}] (0,-1.5) .. controls (-1.5,-0.5) and (-1.5,0.5) .. (0,1.5);
\draw[black,postaction={on each segment={mid arrow=black}}] (0,-0.5) .. controls (0.5,-1) and (-0.5,-1) .. (0,-0.5); 
\end{tikzpicture}\hspace{0.5cm}
\begin{tikzpicture}[scale=0.75] 
\draw (0,1.5) node[above=1pt] {1};
\draw (0,-1.5) node[below=1pt] {2};

\draw (0,0.5) node[below=1pt] {3};
\draw (0,-0.5) node[above=1pt] {4};

\draw (0,0) circle (1.5cm);
\draw (0,0) circle (0.5cm);

\draw[black,postaction={on each segment={mid arrow=black}}] (0,1.5) .. controls (1.5,0.5) and (1.5,-0.5) .. (0,-1.5);
\draw[black,postaction={on each segment={mid arrow=black}}] (0,-1.5) -- (0,-0.5);
\draw[black,postaction={on each segment={mid arrow=black}}] (0,-0.5) .. controls (0.75,-0.5) and (1,0.5) .. (0,1.5);
\draw[black,postaction={on each segment={mid arrow=black}}] (0,0.5) .. controls (-0.5,1) and (0.5,1) .. (0,0.5); 
\end{tikzpicture}\hspace{0.5cm}
\begin{tikzpicture}[scale=0.75] 
\draw (0,1.5) node[above=1pt] {1};
\draw (0,-1.5) node[below=1pt] {2};

\draw (0,0.5) node[below=1pt] {3};
\draw (0,-0.5) node[above=1pt] {4};

\draw (0,0) circle (1.5cm);
\draw (0,0) circle (0.5cm);

\draw[black,postaction={on each segment={mid arrow=black}}] (0,1.5) .. controls (-1,0.5) and (-0.75,-0.5) .. (0,-0.5);
\draw[black,postaction={on each segment={mid arrow=black}}] (0,-0.5) -- (0,-1.5);
\draw[black,postaction={on each segment={mid arrow=black}}] (0,-1.5) .. controls (-1.5,-0.5) and (-1.5,0.5) .. (0,1.5);
\draw[black,postaction={on each segment={mid arrow=black}}] (0,0.5) .. controls (-0.5,1) and (0.5,1) .. (0,0.5); 
\end{tikzpicture}\hspace{0.5cm}
\begin{tikzpicture}[scale=0.75] 
\draw (0,1.5) node[above=1pt] {1};
\draw (0,-1.5) node[below=1pt] {2};

\draw (0,0.5) node[below=1pt] {3};
\draw (0,-0.5) node[above=1pt] {4};

\draw (0,0) circle (1.5cm);
\draw (0,0) circle (0.5cm);

\draw[black,postaction={on each segment={mid arrow=black}}] (0,1.5) .. controls (1.5,1) and (1,-1) .. (0,-0.5);
\draw[black,postaction={on each segment={mid arrow=black}}] (0,-0.5) .. controls (0.9,-0.5) and (0.9,0.5) .. (0,0.5);
\draw[black,postaction={on each segment={mid arrow=black}}] (0,0.5) -- (0,1.5);
\draw[black,postaction={on each segment={mid arrow=black}}] (0,-1.5) .. controls (-0.5,-1) and (0.5,-1) .. (0,-1.5); 
\end{tikzpicture}\hspace{0.5cm}
\begin{tikzpicture}[scale=0.75] 
\draw (0,1.5) node[above=1pt] {1};
\draw (0,-1.5) node[below=1pt] {2};

\draw (0,0.5) node[below=1pt] {3};
\draw (0,-0.5) node[above=1pt] {4};

\draw (0,0) circle (1.5cm);
\draw (0,0) circle (0.5cm);

\draw[black,postaction={on each segment={mid arrow=black}}] (0,1.5) -- (0,0.5);
\draw[black,postaction={on each segment={mid arrow=black}}] (0,0.5) .. controls (-0.9,0.5) and (-0.9,-0.5) .. (0,-0.5);
\draw[black,postaction={on each segment={mid arrow=black}}] (0,-0.5) .. controls (-1,-1) and (-1.5,1) .. (0,1.5);
\draw[black,postaction={on each segment={mid arrow=black}}] (0,-1.5) .. controls (-0.5,-1) and (0.5,-1) .. (0,-1.5); 
\end{tikzpicture}\hspace{0.5cm}
\begin{tikzpicture}[scale=0.75] 
\draw (0,1.5) node[above=1pt] {1};
\draw (0,-1.5) node[below=1pt] {2};

\draw (0,0.5) node[below=1pt] {3};
\draw (0,-0.5) node[above=1pt] {4};

\draw (0,0) circle (1.5cm);
\draw (0,0) circle (0.5cm);

\draw[black,postaction={on each segment={mid arrow=black}}] (0,-1.5) .. controls (-1.5,-1) and (-1,1) .. (0,0.5);
\draw[black,postaction={on each segment={mid arrow=black}}] (0,-0.5) -- (0,-1.5);
\draw[black,postaction={on each segment={mid arrow=black}}] (0,0.5) .. controls (-0.9,0.5) and (-0.9,-0.5) .. (0,-0.5);
\draw[black,postaction={on each segment={mid arrow=black}}] (0,1.5) .. controls (0.5,1) and (-0.5,1) .. (0,1.5); 
\end{tikzpicture}\hspace{0.5cm}
\begin{tikzpicture}[scale=0.75] 
\draw (0,1.5) node[above=1pt] {1};
\draw (0,-1.5) node[below=1pt] {2};

\draw (0,0.5) node[below=1pt] {3};
\draw (0,-0.5) node[above=1pt] {4};

\draw (0,0) circle (1.5cm);
\draw (0,0) circle (0.5cm);

\draw[black,postaction={on each segment={mid arrow=black}}] (0,-1.5) -- (0,-0.5);
\draw[black,postaction={on each segment={mid arrow=black}}] (0,-0.5) .. controls (0.9,-0.5) and (0.9,0.5) .. (0,0.5);
\draw[black,postaction={on each segment={mid arrow=black}}] (0,0.5) .. controls (1,1) and (1.5,-1) .. (0,-1.5);
\draw[black,postaction={on each segment={mid arrow=black}}] (0,1.5) .. controls (0.5,1) and (-0.5,1) .. (0,1.5); 
\end{tikzpicture}\hspace{0.5cm}
\begin{tikzpicture}[scale=0.75] 
\draw (0,1.5) node[above=1pt] {1};
\draw (0,-1.5) node[below=1pt] {2};

\draw (0,0.5) node[below=1pt] {3};
\draw (0,-0.5) node[above=1pt] {4};

\draw (0,0) circle (1.5cm);
\draw (0,0) circle (0.5cm);

\draw[black,postaction={on each segment={mid arrow=black}}] (0,1.5) .. controls (1.5,0.5) and (1.5,-0.5) .. (0,-1.5);
\draw[black,postaction={on each segment={mid arrow=black}}] (0,-1.5) .. controls (-1.5,-1) and (-1,1) .. (0,0.5);
\draw[black,postaction={on each segment={mid arrow=black}}] (0,0.5) .. controls (-0.9,0.5) and (-0.9,-0.5) .. (0,-0.5);
\draw[black,postaction={on each segment={mid arrow=black}}] (0,-0.5) .. controls (0.75,-0.5) and (1,0.5) .. (0,1.5);
\end{tikzpicture}\hspace{0.5cm}
\begin{tikzpicture}[scale=0.75] 
\draw (0,1.5) node[above=1pt] {1};
\draw (0,-1.5) node[below=1pt] {2};

\draw (0,0.5) node[below=1pt] {3};
\draw (0,-0.5) node[above=1pt] {4};

\draw (0,0) circle (1.5cm);
\draw (0,0) circle (0.5cm);

\draw[black,postaction={on each segment={mid arrow=black}}] (0,1.5) .. controls (1.5,0.5) and (1.5,-0.5) .. (0,-1.5);
\draw[black,postaction={on each segment={mid arrow=black}}] (0,0.5) -- (0,1.5);
\draw[black,postaction={on each segment={mid arrow=black}}] (0,-0.5) .. controls (0.9,-0.5) and (0.9,0.5) .. (0,0.5);
\draw[black,postaction={on each segment={mid arrow=black}}] (0,-1.5) -- (0,-0.5);
\end{tikzpicture}\hspace{0.5cm}
\begin{tikzpicture}[scale=0.75] 
\draw (0,1.5) node[above=1pt] {1};
\draw (0,-1.5) node[below=1pt] {2};

\draw (0,0.5) node[below=1pt] {3};
\draw (0,-0.5) node[above=1pt] {4};

\draw (0,0) circle (1.5cm);
\draw (0,0) circle (0.5cm);

\draw[black,postaction={on each segment={mid arrow=black}}] (0,1.5) -- (0,0.5);
\draw[black,postaction={on each segment={mid arrow=black}}] (0,0.5) .. controls (-0.9,0.5) and (-0.9,-0.5) .. (0,-0.5);
\draw[black,postaction={on each segment={mid arrow=black}}] (0,-0.5) -- (0,-1.5);
\draw[black,postaction={on each segment={mid arrow=black}}] (0,-1.5) .. controls (-1.5,-0.5) and (-1.5,0.5) .. (0,1.5);
\end{tikzpicture}\hspace{0.5cm}
\begin{tikzpicture}[scale=0.75] 
\draw (0,1.5) node[above=1pt] {1};
\draw (0,-1.5) node[below=1pt] {2};

\draw (0,0.5) node[below=1pt] {3};
\draw (0,-0.5) node[above=1pt] {4};

\draw (0,0) circle (1.5cm);
\draw (0,0) circle (0.5cm);

\draw[black,postaction={on each segment={mid arrow=black}}] (0,1.5) .. controls (1.5,1) and (1,-1) .. (0,-0.5);
\draw[black,postaction={on each segment={mid arrow=black}}] (0,-0.5) .. controls (0.9,-0.5) and (0.9,0.5) .. (0,0.5);
\draw[black,postaction={on each segment={mid arrow=black}}] (0,0.5) .. controls (-0.75,0.5) and (-1,-0.5) .. (0,-1.5);
\draw[black,postaction={on each segment={mid arrow=black}}] (0,-1.5) .. controls (-1.5,-0.5) and (-1.5,0.5) .. (0,1.5);
\end{tikzpicture}
\end{center}\normalsize
\captionof{figure}{The 18 elements of $\protect\NCA(2,2)$.}\label{fig-NCAcollection}
\end{figure}
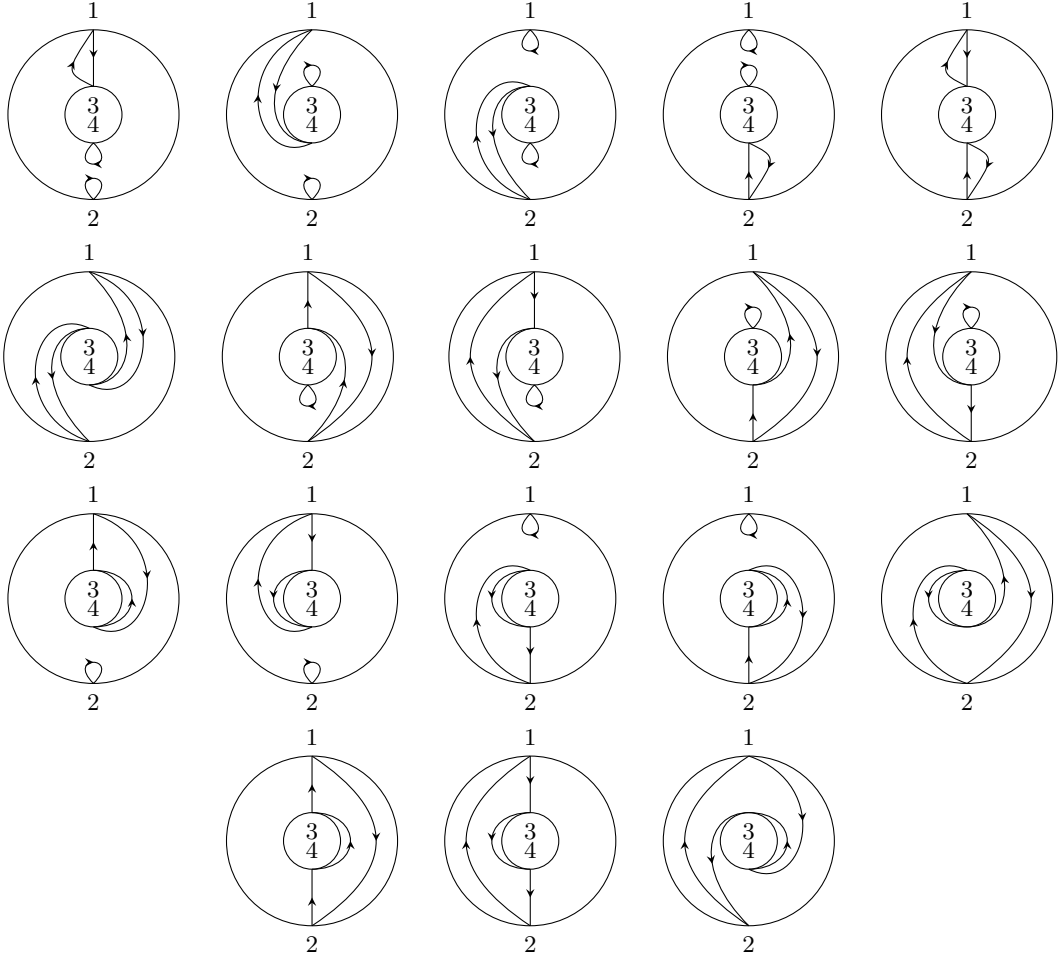

We conclude this section by recalling the Kreweras complement for annular non-crossing permutations from~\cite{MingoNica2004} (see Fig.~\ref{fig-annulusK} below for an example). Similarly to the disk case, taking the Kreweras complement is an invertible map on $\NCA(k,\ell)$.

\begin{definition}[Kreweras complement, annulus case]\label{def-Kreweras2}
Consider the $(k,\ell)$-annulus and label the midpoints of the arcs between the points $1,\dots,k+\ell$ (black in Fig.~\ref{fig-annulusK}) also by $1,\dots,k+\ell$ (red in Fig.~\ref{fig-annulusK}). Respecting the orientation of the two circles, we arrange the new labels such that the arc $s$ follows the point $s$. Let $\pi\in \NCA(k,\ell)$ be visualized on the $(k,\ell)$-annulus as in Definition~\ref{def-NCA}. The \emph{(annular) Kreweras complement} $K(\pi)\in \NCA(k,\ell)$ is defined as the maximal annular non-crossing permutation on $[k+\ell]$ that can be drawn using only the labels at the midpoints of the arcs and without intersecting the cycles of $\pi$. In particular, each cycle of $K(\pi)$ again encloses a region in the annulus that is homeomorphic to the disk with boundary oriented clockwise. In this context, we consider an annular non-crossing permutation maximal if none of its cycles can be extended (by merging cycles) without inducing a crossing.
\end{definition}

\begin{figure}[H]
\begin{center}
\begin{tikzpicture}[scale=1.5]
\draw (0,2) node[above=1pt] {1};
\draw (1.414,1.414) node [above right=1pt] {\color{red}1\color{black}};
\draw (2,0) node[right=1pt] {2};
\draw (1.414,-1.414) node [below right=1pt] {\color{red}2\color{black}};
\draw (0,-2) node[below=1pt] {3};
\draw (-1.414,-1.414) node [below left=1pt] {\color{red}3\color{black}};
\draw (-2,0) node[left=1pt] {4};
\draw (-1.414,1.414) node [above left=1pt] {\color{red}4\color{black}};

\draw (0.3536,0.3536) node[below left=0.1pt] {5};
\draw (0,0.5) node [below=0.1pt] {\color{red}5\color{black}};
\draw (-0.3536,0.3536) node[below right=0.1pt] {6};
\draw (-0.3536,-0.3536) node [above right=0.1pt] {\color{red}6\color{black}};
\draw (0,-0.5) node[above=0.1pt] {7};
\draw (0.3536,-0.3536) node [above left=0.1pt] {\color{red}7\color{black}};

\draw (0,0) circle (2cm);
\draw (0,0) circle (0.5cm);

\draw[black,postaction={on each segment={mid arrow=black}}] (0,2) -- (2,0);
\draw[black,postaction={on each segment={mid arrow=black}}] (2,0) .. controls (1,-1) .. (0,-0.5);
\draw[black,postaction={on each segment={mid arrow=black}}] (0,-0.5) .. controls (1.5,-1.2) and (1.5,0.5) .. (0.3536,0.3536);
\draw[black,postaction={on each segment={mid arrow=black}}] (0.3536,0.3536) -- (0,2);

\draw[black,postaction={on each segment={mid arrow=black}}] (0,-2) -- (-2,0);
\draw[black,postaction={on each segment={mid arrow=black}}] (-2,0) .. controls (-1,-0.5) .. (0,-2);

\draw[black,postaction={on each segment={mid arrow=black}}] (-0.3536,0.3536) .. controls (-1.2,0.4) and (-0.2,1.1) .. (-0.3536,0.3536);

\draw[red,postaction={on each segment={mid arrow=red}}] (1.414,1.414) .. controls (1.6,0.8) and (0.7,1.2) .. (1.414,1.414);
\draw[red,postaction={on each segment={mid arrow=red}}] (-1.414,-1.414) .. controls (-1.5,-0.7) and (-0.6,-1.3) .. (-1.414,-1.414);
\draw[red,postaction={on each segment={mid arrow=red}}] (0.3536,-0.3536) .. controls (1,0) and (1.2,-0.9) .. (0.3536,-0.3536);

\draw[red,postaction={on each segment={mid arrow=red}}] (-1.414,1.414) .. controls (0,1) .. (0,0.5);
\draw[red,postaction={on each segment={mid arrow=red}}] (0,0.5) .. controls (-0.5,1.5) and (-1.5,0.5) .. (-0.3536,-0.3536);
\draw[red,postaction={on each segment={mid arrow=red}}] (-0.3536,-0.3536) .. controls (0,-1) .. (1.414,-1.414);
\draw[red,postaction={on each segment={mid arrow=red}}] (1.414,-1.414) .. controls (1,-2) and (-2,0) .. (-1.414,1.414);

\end{tikzpicture}
\end{center}
\captionof{figure}{The annular non-crossing permutation $\pi=(1275)(34)(6)$ in black and its Kreweras complement $K(\pi)=(1)(2456)(3)(7)$ in red.}\label{fig-annulusK}
\end{figure}
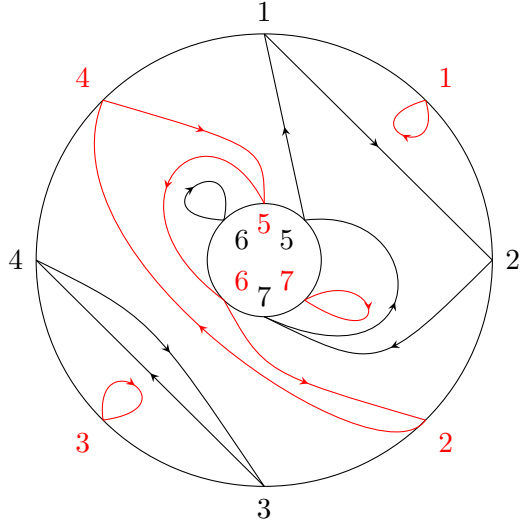

Note that $|\pi|+|K(\pi)|=k+\ell$ for any $\pi\in\NCA(k,\ell)$ (see, e.g.,~\cite[Sect.~6]{MingoNica2004}). We remark that while defining the annular Kreweras complement on the level of partitions would also be possible, the resulting map does not have the same properties as in the disk case (see, e.g.,~\cite[Sect.~1]{MingoNica2004} for a discussion). Therefore, we will only consider the annular Kreweras complement for permutations. Note that one can further assign a unique Kreweras complement to any marked non-crossing partition $\pi$ arising from some element $\pi_1\times\pi_2\in NCP(k)\times NCP(\ell)$ by applying Definition~\ref{def-Kdisk} circle-wise. In this case, we write $K(\pi)=K(\pi_1)\times K(\pi_2)$.

\section{Main Results}\label{sect-results}
The main focus of the present paper lies in determining the limiting covariance structure arising in the CLT for the centered statistics
\begin{align}
X_{\alpha}&:=\langle T_{[1,k]}\rangle-\E \langle T_{[1,k]}\rangle =\langle G_1A_1\dots G_kA_k\rangle-\E \langle G_1A_1\dots G_kA_k\rangle,\label{eq-defX}\\
Y_{\alpha}&:=\langle f_1(W)A_1\dots f_k(W)A_k\rangle-\E\langle f_1(W)A_1\dots f_k(W)A_k\rangle.\label{eq-defY}
\end{align}
Here, $\alpha=((z_1,A_1),\dots,(z_k,A_k))$ resp. $\alpha:=((f_1,A_1),\dots,(f_k,A_k))$ is a multi-index containing bounded deterministic matrices $A_1,\dots,A_k$ and either the spectral parameters $z_1,\dots,z_k\in\C\setminus\R$ with $|\Im z_j|\gtrsim1$ appearing in the resolvents or the test functions $f_1,\dots,f_k\in H^{k+1}(\R)$ with $\|f_j\|\lesssim1$. Whenever we need to refer to the number $k$ of resolvents (resp. test functions) in the product $X_\alpha$ (resp.~$Y_\alpha$), we carry the parameter $k$ as a superscript and write $\smash{X^{(k)}_{\alpha}}$ (resp.~$\smash{Y^{(k)}_{\alpha}}$). Recall that we set $T_j=G_jA_j=G(z_j)A_j$ as well as $T_{[i,j]}=T_iT_{i+1}\dots T_j$. Similarly, we introduce $F_j:=f_j(W)A_j$ and use the interval notation
\begin{displaymath}
F_{[i,j]}:=f_i(W)A_i\dots f_j(W)A_j
\end{displaymath}
for $i<j$ as well as $F_{\emptyset}=0$.

\subsection{Resolvent Central Limit Theorem and Recursion}\label{sect-CLT1}
We start by identifying the joint distribution of multiple $\smash{X^{(k_i)}_{\alpha_i}}$ with different $k_i$ and $\alpha_i$. To state the limiting covariance structure, we introduce a recursively defined set function $\m[\cdot|\cdot]$, which we later identify as the deterministic approximation of the (appropriately scaled) covariance of $\langle T_{[1,k]}\rangle$ and $\langle T_{[k+1,k+\ell]}\rangle$ similar to $M_{[k]}$ and $\m[\cdot]$ arising for the expectation of $T_{[1,k\rangle}G_k$ (see Theorem~\ref{thm-multiG-LL} as well as~\eqref{eq-defM} and~\eqref{eq-formulaM})\footnote{Note the similarity between the notations $\m[\cdot]$ and $\m[\cdot|\cdot]$, which take one and two resolvent chains as arguments, respectively.}. Note that $\alpha=((z_1,A_1),\dots,(z_k,A_k))$ contains the same information on the spectral parameters and deterministic matrices involved as the set of matrices $(T_j,j\in[k])$. We will, therefore, occasionally abuse notation and use $(z_j,A_j)$ and $T_j=G_jA_j$ interchangeably. In particular, we write
\begin{align*}
\m[\alpha|\beta]=\m[T_1,\dots,T_k|T_{k+1},\dots,T_{k+\ell}]
\end{align*}
where the two multi-indices $\alpha$ and $\beta$ index the spectral parameters and deterministic matrices in $T_1,\dots,T_k$ and $T_{k+1},\dots,T_{k+\ell}$, respectively. At this point, we only give a recursive definition for $\m[\cdot|\cdot]$, however, explicit formulas are later obtained in Section~\ref{sect-formulas}. Note that the case $\sigma=\widetilde{\omega}_2=0$ of Definition~\ref{def-M} was already given in~\cite{JRmain}.

\begin{definition}\label{def-M}
Let $S_1=(T_1,\dots,T_{k'})$ and $S_2=(T_{k'+1},\dots,T_{k'+\ell'})$ be two (ordered) finite sets of complex $N\times N$-matrices of the form $T_j=G_jA_j$. We define $\m[\cdot|\cdot]$ as the (deterministic) function of pairs of sets $S_1,S_2$ with values in $\C$ and the following properties:
\begin{itemize}
\item[(i)] Symmetry: $\m[\cdot|\cdot]$ is symmetric under the interchanging of its arguments, i.e., for any sets $B_1\subseteq S_1,B_2\subseteq S_2$ we have
\begin{displaymath}
\m[(T_i,i\in B_1)|(T_j, j\in B_2)]=\m[(T_j, j\in B_2)|(T_i,i\in B_1)].
\end{displaymath}
\item[(ii)] Initial condition: For any sets $B_1\subseteq S_1,B_2\subseteq S_2$ we have
\begin{equation}\label{eq-Minitial}
\m[(T_i,i\in B_1)|\emptyset]=\m[\emptyset|(T_j, j\in B_2)]=0.
\end{equation}
\item[(iii)] Recursion: Let $B_1\subseteq S_1$ and $B_2\subseteq S_2$ be ordered subsets with $|B_1|=k\leq k'$ and $|B_2|=\ell\leq\ell'$ elements, respectively. We index the matrices in $B_1$ by $[k]$ and the matrices in $B_2$ by $[k+1,k+\ell]$. The function $\m[\cdot|\cdot]$ satisfies the following linear recursion
\begin{align}
&\m[T_1,\dots,T_k|T_{k+1},\dots,T_{k+\ell}]\NN\\
&=m_1\Bigg(\m[T_2,\dots,T_{k-1},G_kA_kA_1|T_{k+1},\dots,T_{k+\ell}]\NN\\
&\quad+q_{1,k}\m[T_2,\dots,T_{k-1},G_kA_1|T_{k+1},\dots,T_{k+\ell}]\langle A_k\rangle\label{eq-Mrecursion}\\
&\quad+\sum_{j=1}^{k-1}\m[T_1,\dots,T_{j-1},G_j|T_{k+1},\dots,T_{k+\ell}]\big(\m[T_j,\dots,T_k]+q_{1,k}\m[T_j,\dots,T_{k-1},G_k]\langle A_k\rangle\big)\NN\\
&\quad+\sum_{j=2}^k\m[T_1,\dots,T_{j-1},G_j]\Big(\m[T_j,\dots,T_k|T_{k+1},\dots,T_{k+\ell}]\NN\\
&\quad\quad+q_{1,k}\m[T_j,\dots,T_{k-1},G_k|T_{k+1},\dots,T_{k+\ell}]\langle A_k\rangle\Big)+\mathfrak{s}_{GUE}+\mathfrak{s}_\kappa+\mathfrak{s}_\sigma+\mathfrak{s}_\omega\Bigg)\NN
\end{align}
where the source terms $\mathfrak{s}_{GUE}$, $\mathfrak{s}_\kappa$, $\mathfrak{s}_{\sigma}$, and $\mathfrak{s}_{\omega}$ are given by
\begin{align}
\mathfrak{s}_{GUE}&:=\sum_{j=1}^\ell \Big(\m[T_1,\dots,T_k,T_{k+j},\dots,T_{k+j-1},G_{k+j}]\NN\\
&\quad\quad +q_{1,k}\m[T_1,\dots,T_{k-1},G_k,T_{k+j},\dots,T_{k+j-1},G_{k+j}]\langle A_k\rangle\Big)\label{eq-sourceGUE}\\
\mathfrak{s}_{\kappa}&:=\kappa_4\sum_{r=1}^k\sum_{s=k+1}^{k+\ell}\Big(\sum_{t=k+1}^s\langle M_{[r]}\odot M_{(s,\dots,k+\ell,k+1,\dots,t)}\rangle\langle(M_{[r,k]}A_k)\odot M_{[t,s]}\rangle\NN\\
&\quad\quad+\sum_{t=s}^{k+\ell}\langle M_{[r]}\odot M_{[s,t]}\rangle\langle(M_{[r,k]}A_k)\odot M_{(t,\dots,k+\ell,k+1,\dots,s)}\rangle\Big)\NN\\
&\quad+\kappa_4q_{1,k}\sum_{r=1}^k\sum_{s=k+1}^{k+\ell}\Big(\sum_{t=k+1}^s\langle M_{[r]}\odot M_{(s,\dots,k+\ell,k+1,\dots,t)}\rangle\langle M_{[r,k]}\odot M_{[t,s]}\rangle\NN\\
&\quad\quad+\sum_{t=s}^{k+\ell}\langle M_{[r]}\odot M_{[s,t]}\rangle\langle M_{[r,k]}\odot M_{(t,\dots,k+\ell,k+1,\dots,s)}\rangle\Big)\langle A_k\rangle.\label{eq-sourcekappa}\\
\mathfrak{s}_\sigma&:=\sigma\sum_{j=1}^\ell \m[T_1\dots,T_k,G_{k+j}^tA_{k+j-1}^t,\dots,G_{k+\ell}^tA_{k+1}^t,\dots,G_{k+j-1}^tA_{k+j}^t,G_{k+j}]\label{eq-sourcesigma}\\
&\quad +q_{1,k}\sigma\sum_{j=1}^\ell \m[T_1\dots,T_{k-1},G_k,G_{k+j}^tA_{k+j-1}^t,\dots,G_{k+\ell}^tA_{k+1}^t,\dots,G_{k+j-1}^tA_{k+j}^t,G_{k+j}]\langle A_k\rangle\NN\\
\mathfrak{s}_\omega&:=\widetilde{\omega_2}\sum_{j=1}^\ell \langle (M_{[k]}A_k)\odot M_{(k+j,\dots,k+\ell,k+1,\dots k+j)}\rangle\label{eq-sourceomega}\\
&\quad+q_{1,k}\widetilde{\omega_2}\sum_{j=1}^\ell \langle M_{[k]}\odot M_{(k+j,\dots,k+\ell,k+1,\dots k+j)}\rangle\langle A_k\rangle\NN
\end{align}
Recall that $\odot$ denotes the Hadamard product, $q_{1,k}$ was defined in~\eqref{eq-defq}, and $M_{(\dots)}$ was defined in Theorem~\ref{thm-multiG-LL}. Moreover, recall that $\m[\cdot]$ was defined in~\eqref{eq-defM} and the notation with transposes was introduced in~\eqref{eq-mtransposes}.
\end{itemize}
\end{definition}
Note that setting $A_1=\dots=A_{k+\ell}=\Id$ reduces~\eqref{eq-Mrecursion} to~\eqref{eq-mrecursion}, showing that 
\begin{displaymath}
\m[G_1,\dots,G_k|G_{k+1},\dots,G_{k+\ell}]=m[1,\dots,k|k+1,\dots,k+\ell].
\end{displaymath}

\medskip
We use the linearity of the recursion and the different types of source terms to introduce the decomposition
\begin{equation}\label{eq-Mdecomp}
\m[\cdot|\cdot]=\m_{GUE}[\cdot|\cdot]+\kappa_4\m_\kappa[\cdot|\cdot]+\sigma\m_\sigma[\cdot|\cdot]+\widetilde{\omega}_2\m_\omega[\cdot|\cdot],
\end{equation}
where $\m_{GUE}[\cdot|\cdot]$ satisfies~\eqref{eq-Mrecursion} for $\kappa_4=\sigma=\widetilde{\omega}_2=0$, and $\kappa_4\m_\kappa[\cdot|\cdot]$, $\sigma\m_\sigma[\cdot|\cdot]$ resp. $\widetilde{\omega}_2\m_\omega[\cdot|\cdot]$ satisfy~\eqref{eq-Mrecursion} with $\mathfrak{s}_{\kappa}$, $\mathfrak{s}_{\sigma}$ resp. $\mathfrak{s}_{\omega}$ as only source term. Note that $\mathfrak{s}_{GUE}+\mathfrak{s}_\kappa+\mathfrak{s}_\sigma+\mathfrak{s}_\omega$ in~\eqref{eq-Mrecursion} is fully expressible as a function of $A_1,\dots,A_{k+\ell}$ and $m_1,\dots,m_{k+\ell}$ by~\eqref{eq-formulaM},~\eqref{eq-mcircgraphs} and~Lemma~\ref{lem-msharpgraphs}, eventually making $\m[\cdot|\cdot]$ a function of the same quantities.

\medskip
Recall that we set $X_{\alpha}=\langle T_1\dots T_k\rangle-\E\langle T_1\dots T_k\rangle$ with $\alpha=((z_1,A_1),\dots,(z_k,A_k))$. Before stating the CLT for $X_\alpha$, we note the following definition.
\begin{definition}
Consider two functions of the Wigner matrix $W$ in Assumption~\ref{as-Wigner2}, which we denote as $N$-dependent random variables $X^{(N)}$ and $Y^{(N)}$. We say that $X^{(N)}=Y^{(N)}+\cO(\varepsilon)$ \emph{in the sense of moments} if for any polynomial $\cP$ it holds that
\begin{displaymath}
\E \cP(X^{(N)})=\E \cP(Y^{(N)})+\cO(\eps),
\end{displaymath}
where the implicit constant in $\cO(\cdot)$ only depends on the polynomial $\cP$ and the constants in Assumption~\ref{as-Wigner2}.
\end{definition}

We now give a CLT for $X_\alpha$ in~\eqref{eq-defX}. As the main interest of the present paper is the deterministic approximation $\m[\cdot|\cdot]$, we restrict the discussion of the CLT to the macroscopic regime for technical simplicity. The proof of Theorem~\ref{thm-resolventCLT2} is analogous to that of~\cite[Thm.~3.6]{JRmain}. For the convenience of the reader, we include the necessary modifications for adapting the proof in~\cite{JRmain} to the generalized model in Assumption~\ref{as-Wigner2} in Appendix~\ref{app-resolventCLTproof}.

\begin{theorem}[Macroscopic CLT for resolvents]\label{thm-resolventCLT2}
Fix $p\in\N$, let $\alpha_1,\dots,\alpha_p$ be multi-indices, and let $W$ be a Wigner matrix satisfying Assumption~\ref{as-Wigner2}. For each $j=1,\dots,p$ pick a set of spectral parameters $\smash{z_1^{(j)},\dots,z_{k_j}^{(j)}}$ with $\smash{|\Im z_i^{(j)}|\gtrsim1}$  and $\max_j|z_j|\leq N^{100}$ as well as deterministic matrices $\smash{A_1^{(j)},\dots,A_{k_j}^{(j)}}$ with $\smash{\|A_i^{(j)}\|\lesssim1}$. Then,
\begin{align}
N^p\E\Big(\prod_{j=1}^p X_{\alpha_j}\Big)=\sum_{Q\in Pair([p])}\prod_{\{i,j\}\in Q}\m[\alpha_i|\alpha_j]+\cO\Big(\frac{N^\eps}{\sqrt{N}}\Big)\label{eq-resolventCLT2}
\end{align}
for any $\eps>0$. Here, $Pair(S)$ denotes the pairings of a set~$S$ and $\m[\cdot|\cdot]$ is a set function that satisfies Definition~\ref{def-M}. Equation~\eqref{eq-resolventCLT2} establishes an asymptotic version of Wick’s rule and hence identifies the joint limiting distribution of the random variables $(X_{\alpha_j})_j$ as asymptotically complex Gaussian in the sense of moments in the limit $N\rightarrow\infty$.
\end{theorem}

We remark that $\m[\cdot|\cdot]$ is cyclic in the sense that
\begin{displaymath}
\m[(T_j,j\in S_1)|T_1,\dots,T_k]=\m[(T_j,j\in S_1)|T_2,\dots,T_k,T_1]
\end{displaymath}
and that further
\begin{align*}
&\m[(T_j,j\in S_1)|T_1,\dots,T_{k-1},G_k]\NN\\
&=\frac{\m[(T_j,j\in S_1)|T_2,\dots,T_{k-1},G_kA_1]-\m[(T_j,j\in S_1)|T_1,\dots,T_{k-1}]}{z_k-z_1}.
\end{align*}
whenever  $z_1\neq z_k$, $A_k=\Id$, and $\sigma=0$. These identities can be obtained from the "meta argument" below~\cite[Lem.~4.1]{CES-optimalLL} (see also~\cite[Cor.~3.7]{JRmain}) using that the analogous formulas for the original resolvent chains are trivially true by resolvent identities. However, any additional information on $\m[\cdot|\cdot]$ has to be obtained from the recursion~\eqref{eq-Mrecursion} directly.

\subsection{Solution of the Recursion}\label{sect-formulas}
After identifying $\m[\alpha|\beta]$ as the deterministic approximation of $\E[X_\alpha X_\beta]$, we consider Definition~\ref{def-M} in detail. In this section, we derive a solution to the (deterministic) recursion~\eqref{eq-Mrecursion}. This characterizes the overall structure of the function $\m[\cdot|\cdot]$ and yields explicit combinatorial formulas to replace the recursive definition in applications. Making use of the linearity of the recursion and the decomposition~\eqref{eq-Mdecomp}, it is sufficient to consider the components $\m_{GUE}[\cdot|\cdot]$, $\m_\kappa[\cdot|\cdot]$,$\m_\sigma[\cdot|\cdot]$, and $\m_\omega[\cdot|\cdot]$ separately. We start by studying $\m_{GUE}[\cdot|\cdot]$. The proof consists of two steps that are carried out in Section~\ref{sect-mainproof}.

\begin{theorem}\label{thm-main}
Let $\alpha=((z_1,A_1),\dots,(z_k,A_k))$ and $\beta=((z_{k+1},A_{k+1}),\dots,(z_{k+\ell},A_{k+\ell}))$ for some $k,\ell\in\N$. Then,
\begin{align}
\m_{GUE}[\alpha|\beta]&=\sum_{\pi\in \NCA(k,\ell)}\Big(\prod_{B\in K(\pi)}\Big\langle \prod_{j\in B}A_j\Big\rangle\Big)\prod_{B\in\pi}m_{\circ}[B]\NN\\
&\quad+\sum_{\substack{\pi_1\times\pi_2\in NCP(k)\times NCP(\ell),\\ U_1\in\pi_1,U_2\in\pi_2\text{ marked}}}\Big(\prod_{\substack{B_1\in K(\pi_1),\\ B_2\in K(\pi_2)}}\Big\langle \prod_{j\in B_1} A_j\Big\rangle\Big\langle \prod_{j\in B_2} A_j\Big\rangle\Big)m_{\circ\circ}[U_1|U_2]\NN\\
&\quad\quad\times\prod_{\substack{B_1\in \pi_1\setminus U_1,\\ B_2\in\pi_2\setminus U_2}}m_{\circ}[B_1]m_{\circ}[B_2]\label{eq-fpformula}
\end{align}
with $m_{\circ}$ and $m_{\circ\circ}$ being the first and second-order free cumulant functions as defined in~\eqref{eq-mcrelation1} and~\eqref{eq-mcrelation2}, respectively.
\end{theorem}

Observe that  the right-hand side of~\eqref{eq-fpformula} reduces to the combinatorial expression in~\eqref{eq-mcrelation2} if $A_1=\dots=A_{k+\ell}=\Id$.

\begin{remark}
Note that the right-hand side of~\eqref{eq-fpformula} is symmetric with respect to interchanging of $((z_1,A_1),\dots,(z_k,A_k))$ and $((z_{k+1},A_{k+1}),\dots,(z_{k+\ell},A_{k+\ell}))$, which is consistent with the symmetry of $\m[\cdot|\cdot]$ in Definition~\ref{def-M}(i). We can check this directly from~\eqref{eq-fpformula} by observing that there is a one-to-one correspondence between non-crossing permutations of the $(k,\ell)$-annulus and those of the $(\ell,k)$-annulus. This follows from drawing the cycles of the permutation as curves on the respective annuli and observing that interchanging the inner and outer circle with a conformal map (e.g., by inversion to a concentric circle between the outer and inner circle) preserves the standardness and non-crossing property of the picture (cf.~Definition~\ref{def-m}).  Moreover, this symmetry of $m_{GUE}[\cdot|\cdot]$ implies that $m_{\circ\circ}[\cdot|\cdot]$ is also invariant under interchanging~$((z_1,A_1),\dots,(z_k,A_k))$ and $((z_{k+1},A_{k+1}),\dots,(z_{k+\ell},A_{k+\ell}))$ since $m_{GUE}[\cdot|\cdot]$ determines $m_{\circ\circ}[\cdot|\cdot]$ uniquely by Definition~\ref{def-circcirc}.
\end{remark}

\begin{example}[Asymptotics of covariances for GUE]\label{ex-minimalcase}
We consider a special case of Theorem~\ref{thm-resolventCLT2}. Let $p=2$, $k_1=k_2=1$, and assume that $W$ is a GUE matrix\footnote{Note that we are only using that $W$ satisfies Assumption~\ref{as-Wigner2} and $\kappa_4=\sigma=\widetilde{\omega}_2=0$.}. By decomposing $A_1$ and $A_2$ into a tracial and a traceless part, the deterministic approximation for the covariance follows directly from~\cite[Thm.~4.1]{CES-functCLT}, giving
\begin{align*}
&N^2\E(\langle T_1\rangle-\E\langle T_1\rangle)(\langle T_2\rangle-\E\langle T_2\rangle)\\
&\quad=\langle A_1A_2\rangle\frac{m_1^2m_2^2}{(1-m_1m_2)}+\langle A_1\rangle \langle A_2\rangle\Big(\frac{m_1'm_2'}{(1-m_1m_2)^2}-\frac{m_1^2m_2^2}{(1-m_1m_2)}\Big)+\cO\Big(\frac{N^\eps}{\sqrt{N}}\Big)\\
&\quad=\langle A_1A_2\rangle m_\circ[1,2] +\langle A_1\rangle \langle A_2\rangle m_{\circ\circ}[1|2]+\cO\Big(\frac{N^\eps}{\sqrt{N}}\Big),
\end{align*}
where the last equation follows from the formulas in Examples~\ref{ex-circ} and~\ref{ex-circcirc}. Note that the error bound $\Psi/\sqrt{L}$ in~(91) of~\cite{CES-functCLT} evaluates to $\cO(1/\sqrt{N})$ on macroscopic scales. We remark that the deterministic leading term matches the formula for $\m_{GUE}[T_1|T_2]$ obtained from applying~\eqref{eq-Mrecursion} to the initial condition $\m_{GUE}[T_1|\emptyset]=0$. 
\end{example}

Next, we consider the recursion for $\m_\kappa[\cdot|\cdot]$. We obtain a closed solution similar to Theorem~\ref{thm-main}, i.e., a sum of terms that factorizes into two parts depending only on the deterministic matrices $A_1,\dots,A_{k+\ell}$ and the spectral parameters $z_1,\dots,z_{k+\ell}$, respectively. The proof of Theorem~\ref{thm-structurekap} is given in Section~\ref{sect-kappa} below.

\begin{theorem}\label{thm-structurekap}
Let $\alpha=((z_1,A_1),\dots,(z_k,A_k))$ and $\beta=((z_{k+1},A_{k+1}),\dots,(z_{k+\ell},A_{k+\ell}))$ for some $k,\ell\in\N$. Then  there exist
\begin{itemize}
\item[(i)] a family $(\psi_{\pi,B})_{B\in\pi}$ of functions $\psi_{B,\pi}:\C^{|B|}\rightarrow\C$ for every $\pi\in\NCA(k,\ell)$ and
\item[(ii)] a family  $(\psi_{\pi,U_1,U_2})_{U_1\subset[k],U_2\subset[k+1,k+\ell]}$ of functions $\psi_{\pi,U_1,U_2}:\C^{|U_1|}\times\C^{|U_2|}\rightarrow\C$ that are invariant under interchanging of the two arguments as well as functions $(\psi_{\pi_1,B_1})_{B_1\in\pi_1\setminus U_1}$ and $(\psi_{\pi_2,B_2})_{B_2\in\pi_2\setminus U_2}$ with $\psi_{\pi_i,B_i}:\C^{|B_i|}\rightarrow\C$ for every $\pi=\pi_1\times\pi_2\in NCP(k)\times NCP(\ell)$ with marked blocks $ U_1\in\pi_1$ and $U_2\in\pi_2$
\end{itemize}
such that 
\begin{align}
&\m_{\kappa}[\alpha|\beta]=\sum_{\pi\in \NCA(k,\ell)}\prod_{B\in K(\pi)}\Big\langle\Big( \prod_{j\in B\cap[k]}A_j\Big)\odot\Big( \prod_{j\in B\cap[k+1,k+\ell]}A_j\Big)\Big\rangle\prod_{B\in\pi}\psi_{\pi,B}(z_j|j\in B)\NN\\
&\quad+\sum_{\substack{\pi=\pi_1\times\pi_2\in NCP(k)\times NCP(\ell),\\ U_1\in\pi_1,U_2\in\pi_2\text{ marked}}}\Big(\prod_{\substack{B_1\in K(\pi_1),\\ B_2\in K(\pi_2)}}\Big\langle \prod_{j\in B_1} A_j\Big\rangle\Big\langle \prod_{j\in B_2} A_j\Big\rangle\Big)\label{eq-fpformulakap}\\
&\quad\quad\times \psi_{\pi,U_1,U_2}(z_j|j\in U_1\cup U_2)\prod_{\substack{B_1\in \pi_1\setminus U_1,\\ B_2\in\pi_2\setminus U_2}}\psi_{\pi_1,B_1}(z_j|j\in B_1)\psi_{\pi_2,B_2}(z_j|j\in B_2),\NN
\end{align}
where $\odot$ denotes the Hadamard product.
\end{theorem}

Note that despite the obvious structural similarities between~\eqref{eq-fpformula} and~\eqref{eq-fpformulakap}, considering the minimal example
\begin{align*}
\m[T_1|T_2]&=\m_{GUE}[T_1|T_2]+\kappa_4\m_\kappa[T_1|T_2]\\
&=\langle A_1A_2\rangle\frac{m_1^2m_2^2}{(1-m_1m_2)}+\langle A_1\rangle \langle A_2\rangle\Big(\frac{m_1'm_2'}{(1-m_1m_2)^2}-\frac{m_1^2m_2^2}{(1-m_1m_2)}\Big)\\
&\quad+\kappa_4\Big(\langle \mathbf{a}_1\mathbf{a}_2\rangle m_1^3m_2^3+\langle A_1\rangle \langle A_2\rangle(2m_1m_1'm_2m_2'-m_1^3m_2^3)\Big)
\end{align*}
with $\sigma=\widetilde{\omega}_2=0$ already shows that the functions $\psi_i$ describing the dependence on the spectral parameters do not coincide with the free cumulant functions $m_\circ[\cdot]$ and $m_{\circ\circ}[\cdot|\cdot]$ in general. However,~\eqref{eq-Mrecursion} implies that the functions $\psi_i$ themselves satisfy a recursion which allows us to compute them inductively. 

\medskip
We continue by deriving an explicit formula for $\m_\sigma[\cdot|\cdot]$. As the source term $\mathfrak{s}_\sigma$ in the corresponding recursion is, up to transposes, identical to $\mathfrak{s}_{GUE}$, the solution of the recursion is analogous to Theorem~\ref{thm-main} but uses $m^{\#,\sigma}[\cdot]$ instead of the iterated divided differences $m[\cdot]$. We give the proof in Section~\ref{sect-restms}.

\begin{theorem}\label{thm-msigmaform}
Let $\alpha=((z_1,A_1),\dots,(z_k,A_k))$, $\beta=((z_{k+1},A_{k+1}),\dots,(z_{k+\ell},A_{k+\ell}))$ for some $k,\ell\in\N$ and abbreviate
\begin{displaymath}
m_{\sigma}[1,\dots,k|k+1,\dots,k+\ell]:=\m_{\sigma}[G_1,\dots G_k|G_{k+1},\dots,G_{k+\ell}]
\end{displaymath}
in the special case $A_1=\dots=A_{k+\ell}$. If $B\in\pi$ is a connecting cycle of $\pi\in\NCA(k,\ell)$ decomposed as $B=(i_1,\dots,i_r)\circ(j_1,\dots,j_s)$ with $i_1,\dots,i_r\subset[k]$ and $j_1,\dots,j_s\subset[k+1,k+\ell]$, we introduce the notation
\begin{displaymath}
B_{\sigma}:=(i_1,\dots,i_r)\circ(j_s,\dots,j_1).
\end{displaymath}
Then,
\begin{align}
\m_{\sigma}[\alpha|\beta]&=\sum_{\pi\in \NCA(k,\ell)}\Big(\prod_{B\in K(\pi)}\Big\langle \prod_{j\in B\cap[k]}A_j\Big(\prod_{j\in B\cap[k+1,k+j]}A_j\Big)^t\Big\rangle\Big)\prod_{B\in\pi}m^{\#,\sigma}_{\circ}[B_{\sigma}]\NN\\
&\quad+\sum_{\substack{\pi_1\times\pi_2\in NCP(k)\times NCP(\ell),\\ U_1\in\pi_1,U_2\in\pi_2\text{ marked}}}\Big(\prod_{\substack{B_1\in K(\pi_1),\\ B_2\in K(\pi_2)}}\Big\langle \prod_{j\in B_1} A_j\Big\rangle\Big\langle \prod_{j\in B_2} A_j\Big\rangle\Big)\label{eq-fpformulasigma}\\
&\quad\quad\times(m_\sigma)_{\circ\circ}[U_1|U_2]\prod_{\substack{B_1\in \pi_1\setminus U_1,\\ B_2\in\pi_2\setminus U_2}}m^{\#,\sigma}_{\circ}[B_1]m^{\#,\sigma}_{\circ}[B_2]\NN
\end{align}
where $\#=(0,\dots,0,1,\dots,1)$ with the number of zeros and ones matching the number of labels on the inner and outer circle involved in $B$, respectively. Moreover, $\smash{m^{\#,\sigma}_{\circ}[\cdot]}$ denotes the free cumulant function associated with $m^{\#,\sigma}[\cdot]$ via~\eqref{eq-mcrelation1} and $(m_\sigma)_{\circ\circ}[\cdot|\cdot]$ denotes the second-order free cumulant function asssociated to $m_{\sigma}[\cdot|\cdot]$ and $m^{\#,\sigma}[\cdot]$ via~\eqref{eq-mcrelation2}, respectively.
\end{theorem}

Note that the set function $\m_{\sigma}[\cdot|\cdot]$ satisfies the same factorization property as $\m_{GUE}[\cdot|\cdot]$ and~$\m_{\kappa}[\cdot|\cdot]$. In the case $k=\ell=1$, we obtain the formula
\begin{align*}
\m_{\sigma}[T_1|T_2]=\langle A_1A_2^t\rangle\frac{m_1^2m_2^2}{1-\sigma m_1m_2}+\langle A_1\rangle\langle A_2\rangle\Big(\frac{m_1'm_2'}{(1-\sigma m_1m_2)^2}-\frac{m_1^2m_2^2}{1-\sigma m_1m_2}\Big).
\end{align*}
It remains to consider $\m_\omega[\cdot|\cdot]$. The proof of Theorem~\ref{thm-structureom} is given in Section~\ref{sect-restms}.

\begin{theorem}\label{thm-structureom}
Let $\alpha=((z_1,A_1),\dots,(z_k,A_k))$ and $\beta=((z_{k+1},A_{k+1}),\dots,(z_{k+\ell},A_{k+\ell}))$ for some $k,\ell\in\N$. Then  there exist
\begin{itemize}
\item[(i)] a family $(\Psi_{\pi,B})_{B\in\pi}$ of functions $\Psi_{B,\pi}:\C^{|B|}\rightarrow\C$ for every $\pi\in\NCA(k,\ell)$ and
\item[(ii)] a family  $(\Psi_{\pi,U_1,U_2})_{U_1\subset[k],U_2\subset[k+1,k+\ell]}$ of functions $\Psi_{\pi,U_1,U_2}:\C^{|U_1|}\times\C^{|U_2|}\rightarrow\C$ that are invariant under interchanging of the two arguments as well as functions $(\Psi_{\pi_1,B_1})_{B_1\in\pi_1\setminus U_1}$ and $(\Psi_{\pi_2,B_2})_{B_2\in\pi_2\setminus U_2}$ with $\Psi_{\pi_i,B_i}:\C^{|B_i|}\rightarrow\C$ for every $\pi=\pi_1\times\pi_2\in NCP(k)\times NCP(\ell)$ with marked blocks $ U_1\in\pi_1$ and $U_2\in\pi_2$
\end{itemize}
such that 
\begin{align}
&\m_{\omega}[\alpha|\beta]=\sum_{\pi\in \NCA(k,\ell)}\prod_{B\in K(\pi)}\Big\langle\Big( \prod_{j\in B\cap[k]}A_j\Big)\odot\Big( \prod_{j\in B\cap[k+1,k+\ell]}A_j\Big)\Big\rangle\prod_{B\in\pi}\Psi_{B,\pi}(z_j|j\in B)\NN\\
&\quad+\sum_{\substack{\pi=\pi_1\times\pi_2\in NCP(k)\times NCP(\ell),\\ U_1\in\pi_1,U_2\in\pi_2\text{ marked}}}\Big(\prod_{\substack{B_1\in K(\pi_1),\\ B_2\in K(\pi_2)}}\Big\langle \prod_{j\in B_1} A_j\Big\rangle\Big\langle \prod_{j\in B_2} A_j\Big\rangle\Big)\label{eq-fpformulaom}\\
&\quad\quad\times \Psi_{\pi,U_1,U_2}(z_j|j\in U_1\cup U_2)\prod_{\substack{B_1\in \pi_1\setminus U_1,\\ B_2\in\pi_2\setminus U_2}}\Psi_{\pi_1,B_1}(z_j|j\in B_1)\Psi_{\pi_2,B_2}(z_j|j\in B_2).\NN
\end{align}
\end{theorem}

Note that the last contribution $\m_{\omega}[\cdot|\cdot]$ also satisfies the same factorization property as $\m_{GUE}[\cdot|\cdot]$, $\m_{\kappa}[\cdot|\cdot]$, and $\m_\sigma[\cdot|\cdot]$. In the case $k=\ell=1$, we have the formula
\begin{align*}
\m_{\omega}[T_1|T_2]=\langle \mathbf{a}_1\mathbf{a}_2\rangle m_1^2m_2^2+\langle A_1\rangle\langle A_2\rangle(m_1'm_2'-m_1^2m_2^2).
\end{align*}
It further follows from~\eqref{eq-Mrecursion} that the functions $\Psi_i$ themselves satisfy a recursion which allows us to compute them inductively.

\subsection{General Test Functions and Applications to Free Probability}\label{sect-CLT2}
We conclude the discussion by comparing the explicit formulas from Section~\ref{sect-formulas} to the free probability results in~\cite{MaleMingoPecheSpeicher2020}. To do so, we generalize the CLT for resolvents in Theorem~\ref{thm-resolventCLT2} to a full multi-point functional CLT for ($N$-independent) test functions $f_1,\dots,f_k$, i.e., a CLT for the statistics $Y_\alpha$ in~\eqref{eq-defY}. The proof is analogous to that of~\cite[Thm.~2.7]{JRmain} and hence omitted. Note that we restrict Theorem~\ref{thm-functCLT} to real-valued test functions only for simplicity. Extending the results in this section to complex-valued test functions only requires minor modifications to the argument.

\begin{theorem}[Macroscopic multi-point functional CLT]\label{thm-functCLT}
Let $k\in\N$ and pick deterministic matrices $A_1,\dots,A_k\in\C^{N\times N}$ with $\|A_j\|\lesssim1$. Let further $W$ be a Wigner matrix satisfying Assumption~\ref{as-Wigner2} and let $f_1,\dots,f_k\in H^{k+1}(\R)$ be real-valued compactly supported test functions with $\|f_j\|\lesssim1$. Then, for any $\eps>0$, the centered statistics~\eqref{eq-defY} are approximately distributed~(in the sense of moments) as
\begin{equation}
NY^{(k,a)}_{\alpha}=\xi(\alpha)+\cO\Big(\frac{N^\eps}{\sqrt{N}}\Big)
\end{equation}
with a centered ($N$-dependent) Gaussian process $\xi(\alpha)$ satisfying
\begin{align}
&\E\xi(\alpha)\xi(\beta)=\sum_{\pi\in \NCA(k,\ell)}\Big(\prod_{B\in K(\pi)}\Big\langle \prod_{j\in B}A_j\Big\rangle\Big)\Phi_{\pi}^{(GUE)}(f_1,\dots,f_{k+\ell})\NN\\
&\quad +\kappa_4 \sum_{\pi\in \NCA(k,\ell)}\prod_{B\in K(\pi)}\Big\langle\Big( \prod_{j\in B\cap[k]}A_j\Big)\odot\Big( \prod_{j\in B\cap[k+1,k+\ell]}A_j\Big)\Big\rangle\Phi_{\pi}^{(\kappa)}(f_1,\dots,f_{k+\ell})\NN\\
&\quad +\sigma\sum_{\pi\in \NCA(k,\ell)}\Big(\prod_{B\in K(\pi)}\Big\langle \prod_{j\in B\cap[k]}A_j\Big(\prod_{j\in B\cap[k+1,k+j]}A_j\Big)^t\Big\rangle\Big)\Phi_{\pi}^{(\sigma)}(f_1,\dots,f_{k+\ell})\label{eq-covfunctionsgeneral}\\
&\quad + \widetilde{\omega}_2\sum_{\pi\in \NCA(k,\ell)}\prod_{B\in K(\pi)}\Big\langle\Big( \prod_{j\in B\cap[k]}A_j\Big)\odot\Big( \prod_{j\in B\cap[k+1,k+\ell]}A_j\Big)\Big\rangle\Phi_{\pi}^{(\omega)}(f_1,\dots,f_{k+\ell})\NN\\
&\quad+\sum_{\substack{\pi_1\times\pi_2\in NCP(k)\times NCP(\ell),\\ U_1\in\pi_1,U_2\in\pi_2\text{ marked}}}\Big(\prod_{\substack{B_1\in K(\pi_1),\\ B_2\in K(\pi_2)}}\Big\langle \prod_{j\in B_1} A_j\Big\rangle\Big\langle \prod_{j\in B_2} A_j\Big\rangle\Big)\Phi_{\pi_1\times\pi_2,U_1\times U_2}(f_1,\dots,f_{k+\ell}).\NN
\end{align}
Here, $\beta$ denotes another multi-index of length $\ell$ containing the deterministic matrices $A_{k+1},\dots,A_{k+\ell}$ satisfying $\|A_j\|\lesssim1$ and the test functions $f_{k+1},\dots,f_{k+\ell}\in H^{\ell+1}(\R)$. The functions $\smash{\Phi_{\pi}^{(\cdot)}}$ and $\Phi_{\pi_1\times\pi_2,U_1\times U_2}$ in~\eqref{eq-covfunctionsgeneral} can be computed recursively and only depend on the underlying permutation resp. marked partition, the functions $f_1,\dots,f_{k+\ell}$ and the model parameters $\kappa_4$, $\sigma$, and $\widetilde{\omega}_2$.
\end{theorem}

For the later applications, we note the following formulas for the case $\kappa_4=\sigma=\widetilde{\omega}_2=0$. Corollary~\ref{cor-covarianceLL} below is proven in~\cite{JRmain} using the formulas in Section~\ref{sect-formulas} as input.

\begin{corollary}[Cor.~2.9 in~\cite{JRmain}]\label{cor-covarianceLL}
Consider Theorem~\ref{thm-functCLT} for a GUE matrix\footnote{We only use that $W$ satisfies Assumption~\ref{as-Wigner2} and $\kappa_4=\sigma=\widetilde{\omega}_2=0$.} $W$. In this case, we have
\begin{equation}\label{eq-phis1}
\Phi_\pi^{(GUE)}(f_1,\dots,f_{k+\ell}):=\prod_{B\in\pi}\mathrm{sc}_{\circ}[B],
\end{equation}
where $\mathrm{sc}_{\circ}[\cdot]$ denotes the free cumulant function associated with
\begin{equation}\label{eq-defsc}
\mathrm{sc}[i_1,\dots,i_n]:=\int_{-2}^2\Big[\prod_{j=1}^n f_{i_j}(x)\Big]\rho_{sc}(x)\dx x,
\end{equation}
with $\rho_{sc}$ as in~\eqref{eq-scdensity}, and
\begin{equation}\label{eq-phis2}
\Phi_{\pi_1\times\pi_2,U_1\times U_2}(f_1,\dots,f_{k+\ell}):=\mathrm{sc}_{\circ\circ}[U_1|U_2]\prod_{\substack{B_1\in \pi_1\setminus U_1,\\ B_2\in\pi_2\setminus U_2}}\mathrm{sc}_{\circ}[B_1]\mathrm{sc}_{\circ}[B_2],
\end{equation}
where $\mathrm{sc}_{\circ\circ}[\cdot|\cdot]$ denotes the second-order free cumulants associated with $\mathrm{sc}[\cdot]$ in~\eqref{eq-defsc} and
\begin{equation}\label{eq-defsc2}
\mathrm{sc}[i_1,\dots,i_n|i_{n+1},\dots,i_{n+m}]:=\frac12\int_{-2}^2\int_{-2}^2\Big(\prod_{j=1}^nf_{i_j}(x)\Big)'\Big(\prod_{j=1}^mf_{i_{n+j}}(y)\Big)'u(x,y)\dx x\dx y
\end{equation}
by Definition~\ref{def-circcirc}.The kernel $u:[-2,2]\times[-2,2]\rightarrow\R$ is given by
\begin{equation}\label{eq-kernel}
u(x,y):=\frac{1}{4\pi^2}\ln\Big[\frac{(\sqrt{4-x^2}+\sqrt{4-y^2})^2(xy+4-\sqrt{4-x^2}\sqrt{4-y^2})}{(\sqrt{4-x^2}-\sqrt{4-y^2})^2(xy+4+\sqrt{4-x^2}\sqrt{4-y^2})}\Big].
\end{equation}
\end{corollary}

We remark that the formula~\eqref{eq-kernel} also appears in~\cite{DiazJaramilloPardo2022} and~\cite{Vova2023} (see also~\cite[Cor.~3.8]{JRmain}).

\medskip
Whenever $f_j(x)=x$ for all $j=1,\dots,k+\ell$ or, more generally, $f_j$ is an ($N$-independent) polynomial\footnote{We implicitly assume $f_j$ to be compactly supported by setting $\widetilde{f}_j(x)=f_j(x)\chi(x)$, where $f_j$ is a polynomial supported on all of $\R$ and $\chi$ is a smooth cutoff function that is equal to one on $[-5/2,5/2]$ and equal to zero on $[-3,3]^c$. Since $f_j(W)=\widetilde{f}_j(W)$ with high probability by eigenvalue rigidity, we may use~$f_j$ and $\widetilde{f}_j$ interchangeably here.}, the $N\rightarrow\infty$ limit of~\eqref{eq-covfunctionsgeneral} describes the second-order limiting distribution of GUE and deterministic matrices in free probability. It is readily checked that Theorem~\ref{thm-functCLT} indeed coincides with the free probability literature in this case. The computations to obtain Corollary~\ref{cor-fpapplication} are included in Appendix~\ref{app-ncpairings}.

\begin{corollary}\label{cor-fpapplication}
Under the assumptions of Corollary~\ref{cor-covarianceLL} let $f_1(x)=\dots=f_{k+\ell}(x)=x$, i.e., $\smash{Y^{(k,a)}_{\alpha}}=\langle WA_1\dots WA_k\rangle-\E\langle\dots\rangle$ and $\smash{Y^{(\ell,b)}_{\beta}}=\langle WA_{k+1}\dots WA_{k+\ell}\rangle-\E\langle\dots\rangle$. Then,
\begin{equation}\label{eq-fpapplication}
\lim_{N\rightarrow\infty}N^2\E\Big(Y^{(k,a)}_{\alpha}Y^{(\ell,b)}_{\beta}\Big)=\sum_{\pi\in \NCA_2(k,\ell)}\Big(\prod_{B\in K(\pi)}\Big\langle \prod_{j\in B}A_j\Big\rangle\Big)
\end{equation}
where $\NCA_2(k,\ell)$ denotes the pairings in $\NCA(k,\ell)$.
\end{corollary}

The limit in~\eqref{eq-fpapplication} matches~\cite[Thm.~6]{MaleMingoPecheSpeicher2020} and thus reproduces the well-known result of second-order freeness of GUE and deterministic matrices from~\cite{MingoSpeicher2006}. Moreover, the deterministic approximation in Theorem~\ref{thm-functCLT} mirrors the overall structure of the joint second-order distribution of Wigner and deterministic matrices described in Equation~(3) of~\cite{MaleMingoPecheSpeicher2020}. We remark that resolvents and functions with an $N$-dependent mesoscopic scaling as considered in~\cite{JRmain} are usually not accessible in free probability theory as many of the standard techniques rely on explicit moment computations. Theorem~\ref{thm-functCLT} and~\cite[Thm.~2.7]{JRmain} thus show that the underlying combinatorics of non-crossing annular permutations and marked partitions are, in fact, more general. We further remark that the parallels between Theorem~\ref{thm-functCLT} and~\cite{MaleMingoPecheSpeicher2020} continue to hold if we consider multiple independent Wigner matrices instead of one matrix $W$. More precisely, for $n$ independent GUE matrices, the underlying combinatorial structure is given by the so-called \emph{non-mixing} annular non-crossing permutations resp. \emph{non-mixing} marked partitions for $n$ colors (see Remark below~\cite[Cor.~2.11]{JRmain}).

\section{Proof of Theorem~\ref{thm-main} (Formula for $\m_{GUE}[\cdot|\cdot]$)}\label{sect-mainproof}

\subsection{Part 1: Graphs}\label{sect-graphs}
As we only consider $m_{GUE}$ throughout this section, let $\kappa_4=\sigma=\widetilde{\omega}=0$ and thus $m_{GUE}[\cdot|\cdot]=m[\cdot|\cdot]$. Recall from~\eqref{eq-mgraphs} and~\eqref{eq-mcircgraphs} that both $m[\cdot]$ and $m_{\circ}[\cdot]$ are expressable in terms of disk non-crossing graphs. In this section, we give analogous combinatorial formulas for $m[\cdot|\cdot]$ and $m_{\circ\circ}[\cdot|\cdot]$. For this task, we define a new, albeit closely related, multi-set of graphs on the $(k,\ell)$-annulus. We start by introducing a transformation that translates between the disk and the annulus picture.

\begin{definition}\label{def-tau}
Fix $k,\ell\in\N$, $1\leq j\leq \ell$, and consider a disk with the $k+\ell+1$ labels $1,\dots,k,k+j,\dots,k+\ell,k+1,\dots,k+j$ equidistantly placed around its boundary in clockwise order. We define a map $\tau$, refered to as \emph{mediating map}, that takes this picture to the $(k,\ell)$-annulus as follows:
\begin{itemize}
\item[1.)] Use a homeomorphic continuous deformation, e.g., a conformal map, to map the disk and its labels to the $(k,\ell+1)$-annulus with a slit located between 1 and $k$ on the outer circle and the two copies of $k+j$ on the inner circle. 
\item[2.)] Remove the slit to obtain an annulus.
\item[3.)] Merge the two copies of the label $k+j$.
\end{itemize}
\end{definition}

We visualize $\tau$ for an example in Fig.~\ref{fig-deftau} below. The two labels $k+j$ are denoted as 6 and 6' to distinguish between them more easily.

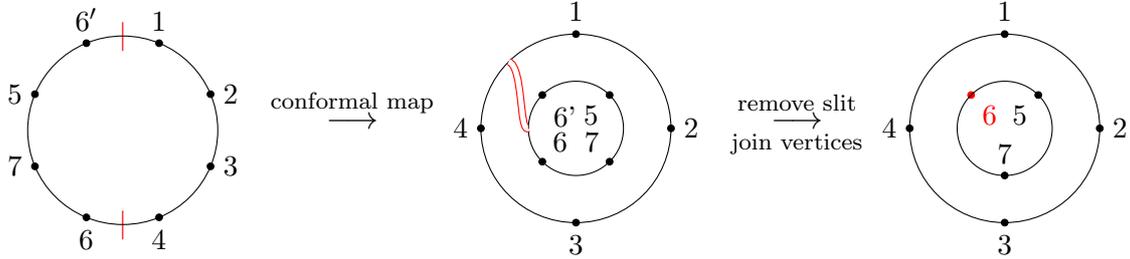
\begin{figure}[H]
\begin{center}
\begin{tikzpicture}[scale=1.25,baseline=(current bounding box.center)]
\draw (0.3827,0.9239) node[above=1pt] {$1$};
\filldraw [black] (0.3827,0.9239) circle (1pt);
\draw (0.9239,0.3827) node[right=1pt] {$2$};
\filldraw [black] (0.9239,0.3827) circle (1pt);
\draw (0.9239,-0.3827) node[right=1pt] {$3$};
\filldraw [black] (0.9239,-0.3827) circle (1pt);
\draw (0.3827,-0.9239) node[below=1pt] {$4$};
\filldraw [black] (0.3827,-0.9239) circle (1pt);

\draw (-0.3827,-0.9239) node[below=1pt] {$6$};
\filldraw [black] (-0.3827,-0.9239) circle (1pt);
\draw (-0.9239,-0.3827) node[left=1pt] {$7$};
\filldraw [black] (-0.9239,-0.3827) circle (1pt);
\draw (-0.9239,0.3827) node[left=1pt] {$5$};
\filldraw [black] (-0.9239,0.3827) circle (1pt);
\draw (-0.3827,0.9239) node[above=1pt] {$6'$};
\filldraw [black] (-0.3827,0.9239) circle (1pt);

\draw (0,-1) node {$\color{red}|\color{black}$};
\draw (0,1) node {$\color{red}|\color{black}$};

\draw (0,0) circle (1cm);
\end{tikzpicture}\hspace{0.3cm}\scalebox{1.1}{$\overset{\text{conformal map}}\longrightarrow$}
\begin{tikzpicture}[scale=1.25,baseline=(current bounding box.center)]
\draw (0,1) node[above=1pt] {1};
\filldraw [black] (0,1) circle (1pt);
\draw (1,0) node[right=1pt] {2};
\filldraw [black] (1,0) circle (1pt);
\draw (0,-1) node[below=1pt] {3};
\filldraw [black] (0,-1) circle (1pt);
\draw (-1,0) node[left=1pt] {4};
\filldraw [black] (-1,0) circle (1pt);

\draw (0.3536,0.3536) node[below left=0.3536pt] {5};
\filldraw [black] (0.3536,0.3536) circle (1pt);
\draw (-0.3536,0.3536) node[below right=0.3536pt] {6'};
\filldraw [black] (-0.3536,0.3536) circle (1pt);
\draw (-0.3536,-0.3536) node[above right=0.1pt] {6};
\filldraw [black] (-0.3536,-0.3536) circle (1pt);
\draw (0.3536,-0.3536) node[above left=0.1pt] {7};
\filldraw [black] (0.3536,-0.3536) circle (1pt);

\draw (0,0) circle (1cm);
\draw (0,0) circle (0.5cm);

\draw[red, style={double,double distance=2pt}] (-0.707,0.707) .. controls (-0.55,0.65) and (-0.6,-0.05) .. (-0.5,0);

\end{tikzpicture}\hspace{0.3cm}\scalebox{1.1}{$\overunderset{\text{remove slit}}{\text{join vertices}}{\longrightarrow}$}
\begin{tikzpicture}[scale=1.25,baseline=(current bounding box.center)]
\draw (0,1) node[above=1pt] {1};
\filldraw [black] (0,1) circle (1pt);
\draw (1,0) node[right=1pt] {2};
\filldraw [black] (1,0) circle (1pt);
\draw (0,-1) node[below=1pt] {3};
\filldraw [black] (0,-1) circle (1pt);
\draw (-1,0) node[left=1pt] {4};
\filldraw [black] (-1,0) circle (1pt);

\draw (0.3536,0.3536) node[below left=0.3536pt] {5};
\filldraw [black] (0.3536,0.3536) circle (1pt);
\draw (-0.3536,0.3536) node[below right=0.3536pt] {\color{red}6\color{black}};
\filldraw [red] (-0.3536,0.3536) circle (1pt);
\draw (0,-0.5) node[above=0.75pt] {7};
\filldraw [black] (0,-0.5) circle (1pt);

\draw (0,0) circle (1cm);
\draw (0,0) circle (0.5cm);
\end{tikzpicture}
\end{center}
\captionof{figure}{The geometry of the transformation $\tau$ for $k=4$, $\ell=3$, and $j=2$.}\label{fig-deftau}
\end{figure}

The map $\tau$ induces a transformation of any graph $\Gamma$ defined on the labeled disk to a graph defined on the $(k,\ell)$-annulus. We denote the resulting annulus graph as $\tau(\Gamma)$. By construction, $\tau(\Gamma)$ is planar whenever $\Gamma$ is a disk non-crossing graph. Recall that we use a slightly more general notion of planar graphs than the standard literature by allowing for loops and multi-edges. We give an example in Fig.~\ref{fig-taugraph} below. For better visibility, the loop arising in the last step is moved from between the two $(1,3)$ edges to the right.

\begin{figure}[H]
\begin{center}
\begin{tikzpicture}[scale=1.25,baseline=(current bounding box.center)]
\draw (-0.707,0.707) node[above left=1pt] {1};
\draw (0.707,0.707) node[above right=1pt] {2};
\draw (0.707,-0.707) node[below right=1pt] {3};
\draw (-0.707,-0.707) node[below left=1pt] {3'};

\draw (-0.707,0.707) -- (0.707,0.707);
\draw (-0.707,0.707) -- (0.707,-0.707);
\draw (-0.707,0.707) -- (-0.707,-0.707);
\draw (0.707,-0.707) -- (-0.707,-0.707);

\draw (0,0) circle (1cm);
\end{tikzpicture}\hspace{0.3cm}\scalebox{1.1}{$\overset{\text{map to annulus}}\longrightarrow$}
\begin{tikzpicture}[scale=1,baseline=(current bounding box.center)]
\draw (0,1.5) node[above=1pt] {1};
\draw (0,-1.5) node[below=1pt] {2};
\draw (0,0.5) node[below=1pt] {3};
\draw (0,-0.5) node[above=1pt] {3'};

\draw (0,1.5) -- (0,0.5);
\draw (0,1.5) .. controls (1.5,0.5) and (1.5,-0.5) .. (0,-1.5);
\draw (0,0.5) .. controls (0.9,0.5) and (0.9,-0.5) .. (0,-0.5);
\draw (0,-0.5) .. controls (1.15,-1.15) and (0.9,0.75) .. (0,1.5);

\draw (0,0) circle (1.5cm);
\draw (0,0) circle (0.5cm);
\end{tikzpicture}\hspace{0.3cm}\scalebox{1.1}{$\overset{\text{join vertices}}\longrightarrow$}
\begin{tikzpicture}[scale=1,baseline=(current bounding box.center)]
\draw (0,1.5) node[above=1pt] {1};
\draw (0,-1.5) node[below=1pt] {2};
\draw (0,0.5) node[below=1pt] {3};

\draw[black,style={double,double distance=2pt}] (0,1.5) -- (0,0.5);
\draw (0.04,0.5) .. controls (0.25,1.25) and (0.75,1) .. (0.04,0.5);
\draw (0,1.5) .. controls (1.5,0.5) and (1.5,-0.5) .. (0,-1.5);

\draw (0,0) circle (1.5cm);
\draw (0,0) circle (0.5cm);
\end{tikzpicture}
\end{center}
\captionof{figure}{Construction of $\tau(\Gamma)$ for a graph $\Gamma\in NCG(\{1,2,3,3'\})$.}\label{fig-taugraph}
\end{figure}
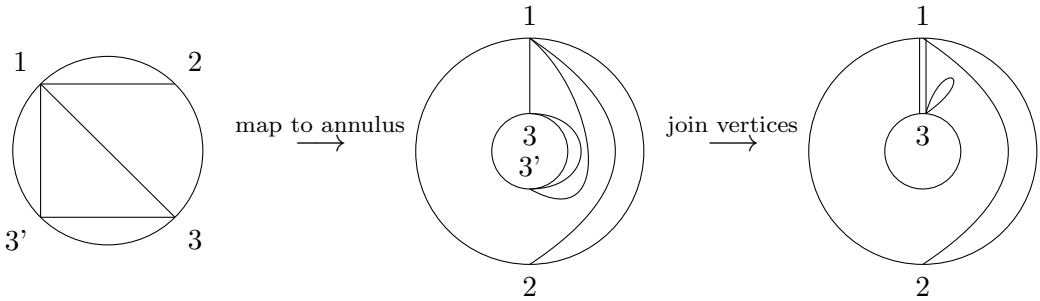

We can now introduce the family of graphs $\cG(k,\ell)$ that constitute the key tool in the proof of Theorem~\ref{thm-main}. In analogy to the disk non-crossing graphs in Definition~\ref{def-NCG}, we require the elements of $\cG(k,\ell)$ to be drawn on the $(k,\ell)$-annulus with the vertices placed around the boundary and the edges drawn in the interior of the annulus (see Fig.~\ref{fig-taugraph} and Example~\ref{ex-G11} below).

\begin{definition}\label{def-Ggraphs}
For $k,\ell\in\N$ we define $\cG([k],[k+1,k+\ell])$ to be the multi-set\footnote{The graph $\Gamma$ may be obtained in several different ways from the recursive construction. This is reflected by the multiplicity of $\Gamma$ in the multi-set $\cG([k],[k+1,k+\ell])$.} of undirected, planar graphs on the $(k,\ell)$-annulus with vertex set $\{1,\dots,k+\ell\}$ and possible loops or double edges that is obtained from the following recursive construction:
\begin{itemize}
\item[(i)] For any $S_1\subset[k]$ and $S_2\subset[k+1,k+\ell]$ we have
\begin{equation}\label{eq-Ginitial}
\cG(S_1,\emptyset)=\cG(\emptyset,S_2)=\emptyset.
\end{equation}
\item[(ii)] The multi-set $\cG([k],[k+1,k+\ell])$ can be constructed from the multi-sets $\cG(S_1,[k+1,k+\ell])$ with $S_1\subsetneq[k]$ as follows: We define
\begin{align}
\cG_{\neg (1,k)}:=\cG_1\cup\cG_2\cup \cG_3\cup\cG_4,\label{eq-Gdecomp2}
\end{align}
as the disjoint union of the sets
\begin{align*}
\cG_1&:=\big\{\Gamma\cup\{1\}\big|\Gamma\in \cG([2,k],[k+1,k+\ell])\big\}\\
\cG_2&:=\bigcup_{j=2}^k\big\{\Gamma=\Gamma_1\cup\Gamma_2\big| \Gamma_1\in NCG([1,j])\text{ with edge }(1,j), \Gamma_2\in\cG([j,k],[k+1,k+\ell])\big\},\\
\cG_3&:=\bigcup_{j=1}^{k-1}\big\{\Gamma=\Gamma_1\cup\Gamma_2\big|\Gamma_1\in\cG([1,j],[k+1,k+\ell])\text{ with edge } (1,j), \Gamma_2\in NCG([j,k])\big\},\\
\cG_4&:=\bigcup_{j=1}^{\ell} \tau\Big(\big\{\Gamma\in NCG(\{1,\dots,k,k+j,\dots,k+\ell,k+1,\dots,k+j-1,k+j\})\big|\\ &\quad\quad\quad\Gamma\text{ has edge }(1,k+j)\big\}\Big).
\end{align*}
Here, the union $\Gamma\cup\{1\}$ in $\cG_1$ is to be understood as adding a separated vertex $1$ to $\Gamma$ while the union $\Gamma_1\cup\Gamma_2$ in $\cG_2$ and $\cG_3$ refers to the graph with the vertex set $\{1,\dots,k+\ell\}$ and the edge set given by the union of the edge sets of $\Gamma_1$ and $\Gamma_2$, respectively. Further, recall $\tau$ from Definition~\ref{def-tau}. We remark that all elements of $\cG_1,\dots,\cG_4$ are planar graphs. Next, let
\begin{displaymath}
\cG_{(1,k)}:=\big\{\Gamma\cup\{(1,k)\}\big|\Gamma\in\cG_{\neg (1,k)}\big\},
\end{displaymath}
where the union is to be understood as adding an edge $(1,k)$ to each graph in $\cG_{\neg (1,k)}$. Observe that the resulting graphs are again planar. Finally, we define
\begin{equation}\label{eq-Gdecomp1}
\cG([k],[k+1,k+\ell]):=\cG_{(1,k)}\cup\cG_{\neg (1,k)},
\end{equation}
where any graphs that occur more than once are counted with multiplicity. In particular, whenever the same graph occurs in both $\cG_{(1,k)}$ and $\cG_{\neg(1,k)}$, the multiplicity of $\Gamma\in\cG([k],[k+1,k+\ell])$ is the total number of occurrences of $\Gamma$ in both subsets.

\item[(iii)] $\cG(S_1,S_2)=\cG(S_2,S_1)$ for any $S_1\subset[k]$ and $S_2\subset[k+1,k+\ell]$ (in the sense that there is a well-defined bijective mapping that takes each element of~$\cG(S_1,S_2)$ to its counterpart).
\end{itemize}
We abbreviate $\cG(k,\ell):=\cG([k],[k+1,k+\ell])$ and refer to its elements as \emph{good graphs}. The subset of connected good graphs is denoted by $\cG_c(k,\ell)$. Any edge $(i,j)$ with $i\in[k]$ and $j\in[k+1,k+\ell]$ is referred to as \emph{connecting edge}.
\end{definition}

\begin{remark}
The occurrence of multi-edges or loops is inherent to the construction of $\cG(k,\ell)$, which can be seen from a simple counting argument. First, note that an element of $\cG(k,\ell)$ can have at most $2(k+\ell)$ edges by construction. To see this, observe that the elements of $\cG_{\neg (1,k)}$ with the highest number of edges lie in $\cG_4$ and that $\Gamma\in\cG_4$ has the same number of edges as the underlying disk non-crossing graph. As any disk non-crossing graph on $n$ vertices has at most $2n-3$ edges (realized by a triangulation), the maximal number of edges for $\Gamma\in\cG_4$ is $2(k+\ell+1)-3=2(k+\ell)-1$. As~\eqref{eq-Gdecomp1} may add another edge to $\Gamma$, the maximum for $\cG(k,\ell)$ is $2(k+\ell)$ edges. On the other hand, the maximal number of edges in a planar graph on the $(k,\ell)$-annulus without multi-edges or loops is only $2(k+\ell)$ if $k,\ell\geq3$ (again realized by a triangulation). It is readily seen that such a graph has strictly less than $2(k+\ell)$ edges if either $k$ or $\ell$ is one or two. As the construction in Definition~\ref{def-Ggraphs} does not introduce crossings and begins with the cases $k,\ell\leq2$, the difference between the two maxima must be reflected as multi-edges or loops.
\end{remark}

We further remark that $\cG(k,\ell)$ is a genuine multi-set, i.e., some graphs appear with multiplicity larger than one, unless $k=\ell=1$ (cf. Lemma~\ref{lem-Gproperties1}(d) and Fig.~\ref{fig-multi}). The key property shared by all elements of $\cG(k,\ell)$ is that each graph arises from some disk non-crossing graph along the recursive construction. Therefore, we may interpret $\cG_4$ as a kind of source term. In view of Lemma~\ref{lem-graphformula1} below, the multi-set $\cG(k,\ell)$ gives an annulus analog to the disk non-crossing graphs, as it plays the same role in the combinatorial description of $m[\cdot|\cdot]$ as $NCG(S)$ does for $m[\cdot]$.

\begin{lemma}\label{lem-graphformula1}
For fixed $k,\ell\in\N$, we have
\begin{align}
m[1,\dots,k|k+1,\dots,k+\ell]=\Big(\prod_{s=1}^{k+\ell}m_s\Big)\sum_{\Gamma\in\cG(k,\ell)}\prod_{(i,j)\in E(\Gamma)}q_{i,j}.\label{eq-graphform}
\end{align}
\end{lemma}

\begin{proof}
As $q_{i,j}=q_{j,i}$ by~\eqref{eq-defq} and $\cG(k,\ell)=\cG(\ell,k)$ by Definition~\ref{def-Ggraphs}(iii), it readily follows that the right-hand side of~\eqref{eq-graphform} is symmetric under the interchanging $[k]$ and $[k+1,k+\ell]$. Similar to the proof of~\cite[Lem.~5.2]{CES-thermalization}, we use the combinatorial formula~\eqref{eq-graphform} as an ansatz to solve the recursion given in~\eqref{eq-minitial} and~\eqref{eq-mrecursion}. First, observe that
\begin{displaymath}
\Big(\prod_{s=1}^{k}m_s\Big)\sum_{\Gamma\in\cG(S_1,\emptyset)}\prod_{(i,j)\in E(\Gamma)}q_{i,j}=\Big(\prod_{s=k+1}^{k+\ell}m_s\Big)\sum_{\Gamma\in\cG(\emptyset,S_2)}\prod_{(i,j)\in E(\Gamma)}q_{i,j}=0
\end{displaymath}
for any $S_1\subset[k]$ and $S_2\subset[k+1,k+\ell]$ due to the sums being empty. Hence, the initial condition~\eqref{eq-minitial} is satisfied.

\medskip
It remains to check~\eqref{eq-mrecursion}. We introduce the notation
\begin{displaymath}
q_{\Gamma}:=\prod_{(i,j)\in E(\Gamma)}q_{i,j}
\end{displaymath}
and conclude from the decompositions~\eqref{eq-Gdecomp1} and~\eqref{eq-Gdecomp2} that
\begin{align}
\sum_{\Gamma\in\cG(k,\ell)}q_{\Gamma}&=(1+q_{1,k})\sum_{\Gamma\in\cG_{\neg(1,k)}}q_{\Gamma}\NN\\
&=(1+q_{1,k})\Big(\sum_{\Gamma\in\cG_1}q_{\Gamma}+\sum_{\Gamma\in\cG_2}q_{\Gamma}+\sum_{\Gamma\in\cG_3}q_{\Gamma}+\sum_{\Gamma\in\cG_4}q_{\Gamma}\Big).\label{eq-decomp1}
\end{align}
Noting that the vertex $1$ in $\Gamma\in\cG_1$ has no adjacent edges, we may write
\begin{displaymath}
\sum_{\Gamma\in\cG_1}q_{\Gamma}=\sum_{\Gamma\in\cG([2,k],[k+1,k+\ell])}q_{\Gamma}.
\end{displaymath}
Further, the transformation $\tau$ only changes the geometry underlying a graph $\Gamma$, but does not influence its edge set. By the definition of $\cG_4$, any $\Gamma\in NCG([1,\dots,k+\ell,k+j])$ used in the construction must have at least one edge $(1,k+j)$. As a consequence, the product $q_{\Gamma}$ always includes the factor $q_{1,k+j}=m_1m_{k+j}(1+q_{1,k+j})$. This yields
\begin{displaymath}
\sum_{\Gamma\in\cG_4}q_{\Gamma}=\sum_{j=1}^{\ell}\Big(m_1m_{k+j}\sum_{\Gamma\in NCG([1,\dots,k+\ell,k+j])}q_{\Gamma}\Big).
\end{displaymath}
Note that the identity $q_{1,k+j}=m_1m_j(1+q_{1,k+j})$ allows writing summations restricted to graphs with an edge $(1,k+j)$ on the left-hand side into an unrestricted sum over all graphs on the right-hand side. We use the same trick for $\cG_2$ and $\cG_3$ as $q_{\Gamma_1\cup\Gamma_2}=q_{\Gamma_1}q_{\Gamma_2}$ for the union of graphs introduced in Definition~\ref{def-Ggraphs}. With these replacements,~\eqref{eq-decomp1} can be written as
\begin{align*}
&\frac{1}{1+q_{1,k}}\sum_{\Gamma\in\cG(k,\ell)}q_{\Gamma}\\
&=\sum_{\Gamma\in\cG([2,k],[k+1,k+\ell])}q_{\Gamma}+\sum_{j=2}^{k}m_1m_j\Big(\sum_{\Gamma\in NCG([1,j])}q_{\Gamma}\Big)\Big(\sum_{\Gamma\in \cG([j,k],[k+1,k+\ell])}q_{\Gamma}\Big)\\
&\quad+\sum_{j=1}^{k-1}m_1m_j\Big(\sum_{\Gamma\in\cG([1,j],[k+1,k+\ell])}q_{\Gamma}\Big)\Big(\sum_{\Gamma\in NCG([j,k])}q_{\Gamma}\Big)+\sum_{j=1}^{\ell}m_1m_{k+j}\Big(\sum_{\Gamma\in NCG([1,\dots,k+\ell,k+j])}q_{\Gamma}\Big).
\end{align*}
Multiplying both sides with $\prod_{s=1}^{k+\ell}m_s$ and noting that $1+q_{1,j}=(1-m_1m_j)^{-1}$, we see that the right-hand side of~\eqref{eq-graphform} satisfies~\eqref{eq-mrecursion} as claimed.
\end{proof}

\begin{example}\label{ex-G11} We have $|\cG(1,1)|=8$. The graphs are visualized in Fig.~\ref{fig-G11} below.
\begin{figure}[H]
\begin{center}
\begin{tikzpicture}[scale=1]
\draw (0,1.5) node[above=1pt] {1};
\draw (0,0.5) node[below=1pt] {2};

\draw (0,1.5) -- (0,0.5);

\draw (0.5,0) arc[start angle=0, end angle=180, radius=0.5];
\draw (1.5,0) arc[start angle=0, end angle=180, radius=1.5];
\end{tikzpicture}\hspace{0.5cm}
\begin{tikzpicture}[scale=1]
\draw (0,1.5) node[above=1pt] {1};
\draw (0,0.5) node[below=1pt] {2};

\draw (0,1.5) -- (0,0.5);
\draw (0,1.5) .. controls (0.25,0.75) and (0.75,1) .. (0,1.5);

\draw (0.5,0) arc[start angle=0, end angle=180, radius=0.5];
\draw (1.5,0) arc[start angle=0, end angle=180, radius=1.5];
\end{tikzpicture}\hspace{0.5cm}
\begin{tikzpicture}[scale=1]
\draw (0,1.5) node[above=1pt] {1};
\draw (0,0.5) node[below=1pt] {2};

\draw (0,1.5) -- (0,0.5);
\draw (0,0.5) .. controls (0.25,1.25) and (0.75,1) .. (0,0.5);

\draw (0.5,0) arc[start angle=0, end angle=180, radius=0.5];
\draw (1.5,0) arc[start angle=0, end angle=180, radius=1.5];
\end{tikzpicture}\hspace{0.5cm}
\begin{tikzpicture}[scale=1]
\draw (0,1.5) node[above=1pt] {1};
\draw (0,0.5) node[below=1pt] {2};

\draw (0,1.5) -- (0,0.5);
\draw (0,1.5) .. controls (0.25,0.75) and (0.75,1) .. (0,1.5);
\draw (0,0.5) .. controls (0.25,1.25) and (0.75,1) .. (0,0.5);

\draw (0.5,0) arc[start angle=0, end angle=180, radius=0.5];
\draw (1.5,0) arc[start angle=0, end angle=180, radius=1.5];
\end{tikzpicture}\\
\begin{tikzpicture}[scale=1]
\draw (0,1.5) node[above=1pt] {1};
\draw (0,0.5) node[below=1pt] {2};

\draw[black,style={double,double distance=2pt}] (0,1.5) -- (0,0.5);

\draw (0.5,0) arc[start angle=0, end angle=180, radius=0.5];
\draw (1.5,0) arc[start angle=0, end angle=180, radius=1.5];
\end{tikzpicture}\hspace{0.5cm}
\begin{tikzpicture}[scale=1]
\draw (0,1.5) node[above=1pt] {1};
\draw (0,0.5) node[below=1pt] {2};

\draw[black,style={double,double distance=2pt}] (0,1.5) -- (0,0.5);
\draw (0.04,1.5) .. controls (0.25,0.75) and (0.75,1) .. (0.04,1.5);

\draw (0.5,0) arc[start angle=0, end angle=180, radius=0.5];
\draw (1.5,0) arc[start angle=0, end angle=180, radius=1.5];
\end{tikzpicture}\hspace{0.5cm}
\begin{tikzpicture}[scale=1]
\draw (0,1.5) node[above=1pt] {1};
\draw (0,0.5) node[below=1pt] {2};

\draw[black,style={double,double distance=2pt}] (0,1.5) -- (0,0.5);
\draw (0.04,0.5) .. controls (0.25,1.25) and (0.75,1) .. (0.04,0.5);

\draw (0.5,0) arc[start angle=0, end angle=180, radius=0.5];
\draw (1.5,0) arc[start angle=0, end angle=180, radius=1.5];
\end{tikzpicture}\hspace{0.5cm}
\begin{tikzpicture}[scale=1]
\draw (0,1.5) node[above=1pt] {1};
\draw (0,0.5) node[below=1pt] {2};

\draw[black,style={double,double distance=2pt}] (0,1.5) -- (0,0.5);
\draw (0.04,1.5) .. controls (0.25,0.75) and (0.75,1) .. (0.04,1.5);
\draw (0.04,0.5) .. controls (0.25,1.25) and (0.75,1) .. (0.04,0.5);

\draw (0.5,0) arc[start angle=0, end angle=180, radius=0.5];
\draw (1.5,0) arc[start angle=0, end angle=180, radius=1.5];
\end{tikzpicture}
\end{center}
\captionof{figure}{The elements of $\cG(1,1)$.}\label{fig-G11}
\end{figure}
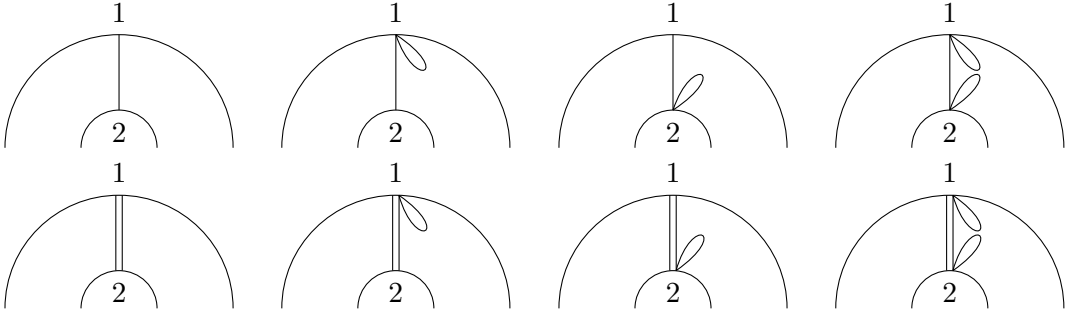

In particular, we readily reobtain~\eqref{eq-msmallest} by evaluating $q_\Gamma$ for every graph in the above list.
Note that decomposing~\eqref{eq-msmallest} into the form $m_1m_2\sum_{\Gamma\in\cG(1,1)} q_\Gamma$ is possible in multiple ways. However, picking graphs that contain a connecting edge yields the set $\cG(1,1)$ in Fig.~\ref{fig-G11} as the smallest possible set.
\end{example}

We state a few properties of the elements of~$\cG(k,\ell)$. In Lemma~\ref{lem-Gproperties1}, we focus on general characteristics  of $\Gamma\in\cG(k,\ell)$ as a planar graph drawn on the $(k,\ell)$-annulus. The properties of $\cG(k,\ell)$ that are needed for the proof of the combinatorial formula for $m_{\circ\circ}$ are given in the separate Lemma~\ref{lem-Gproperties2}.

\begin{lemma}\label{lem-Gproperties1}
Let $k,\ell\in\N$.
\begin{itemize}
\item[(a)] The connected components of any $\Gamma\in\cG(k,\ell)$ give rise to a non-crossing partition of the $(k,\ell)$-annulus. In particular, $\Gamma$ contains a connecting edge if $k,\ell\geq1$.
\item[(b)] $\Gamma\in\cG(k,\ell)$ may have at most two loops. Loops only occur at vertices that are adjacent to a connecting edge. Two loops on vertices on the same circle do not occur.
\item[(c)] $\Gamma\in\cG(k,\ell)$ may have up to $k+\ell-1$ double edges. Double edges are either connecting edges or adjacent to a connecting edge. Edges with a multiplicity higher than two do not occur.
\item[(d)] $\cG(k,\ell)$ is a genuine multi-set unless $k=\ell=1$.
\end{itemize}
\end{lemma}

\begin{proof}
(a) We use proof by induction. First, note that the elements of $\cG(S_1,\emptyset)=\emptyset$ and $\cG(\emptyset,S_2)=\emptyset$ with $S_1\subseteq[k]$ and $S_2\subseteq [k+1,k+\ell]$ clearly give rise to an annular non-crossing partition. Moreover, Example~\ref{ex-G11} establishes the claim in the case $k=\ell=1$ and shows that any $\Gamma\in\cG(1,1)$ contains at least one connecting edge. 

\medskip
Assume next that the elements of $\cG(S_1,[k+1,k+\ell])$ give rise to an annular non-crossing partition for any $S_1\subset[k]$ with $|S_1|\leq k-1$ and a fixed $\ell\geq1$. We aim to show that the connected components of any $\Gamma\in\cG(k,l)$ also correspond to the blocks of some $\pi\in NCP(k,\ell)$. Due to the symmetry induced by Definition~\ref{def-Ggraphs}(iii), this is enough to establish the induction step.

\medskip
By definition, the vertices 1 and $k$ lie next to each other on the outer circle. Hence, adding an edge~$(1,k)$ may connect two connected components, but cannot introduce a crossing in the partition obtained from them. It is, therefore, sufficient to check the claim for elements of $\cG_{\neg (1,k)}$ or, equivalently, for the sets $\cG_1$, $\cG_2$, $\cG_3$, and $\cG_4$ in~\eqref{eq-Gdecomp2}. First, note that the transformation $\tau$ indeed takes the disk partition induced by the disk non-crossing graph to an annular non-crossing partition. This is due to the continuous homeomorphism used in the definition of $\tau$. Moreover, the edge $(1,k+j)$ prescribed for $\Gamma\in NCG(\{1,\dots,k,k+j,\dots,k+j-1,k+j\})$ by the definition is mapped to a connecting edge, ensuring that the resulting partition has a connecting block.

\medskip
By construction, the graphs in $\cG_1$, $\cG_2$, and $\cG_3$ contain an element of $\cG(S_1,[k+1,k+\ell])$ with some $S_1\subset[k]$ as a subgraph. Applying the induction hypothesis for $\cG([2,k],k+1,k+\ell])$ and noting that a separate vertex 1 only adds a singleton set to the underlying partition, we can conclude that $\cG_1$, too, behaves as claimed. The argument for $\cG_2$ and $\cG_3$ is similar. Here, the key observation is that the connected components of the added disk non-crossing graph induce a non-crossing partition of an interval placed along the outer circle. Recalling that all elements of $NCP(k,\ell)$ have at least one connecting block, the corresponding connected component of $\Gamma\in\cG(k,\ell)$ must contain a connecting edge.

\medskip
(b) It is readily seen that the elements of $\cG(S_1,\emptyset)=\emptyset$ and $\cG(\emptyset,S_2)=\emptyset$ with $S_1\subseteq[k]$ and $S_2\subseteq [k+1,k+\ell]$, as well as all $\Gamma\in\cG(1,1)$ have the claimed structure. Moreover, adding an edge $(1,k)$ to a graph cannot introduce a loop unless $k=1$, so it is again sufficient to establish the induction step for $\cG_{\neg (1,k)}$.

\medskip
By Definition~\ref{def-NCG}, a disk non-crossing graph does not contain any loops. Hence, any loops of $\Gamma\in\cG_1\cup\cG_2\cup\cG_3$ must occur in the subgraph in $\cG(S_1,[k+1,k+\ell])$ with $S_1\subset[k]$ used to construct $\Gamma$. As the induction hypothesis applies to this subgraph, we conclude that $\Gamma$ has at most two loops that satisfy the claimed placement rules, respectively. For $\Gamma\in\cG_4$, note that applying $\tau$ to an element of $NCG(\{1,\dots,k,k+j,\dots,k+j-1,k+j\})$ yields an annulus graph with a loop if and only if the disk non-crossing graph contains an edge $(k+j,k+j)$. As the latter must also contain an edge $(1,k+j)$ by definition, a loop in $\Gamma$ is always adjacent to a connecting edge.

\medskip
We remark that the loops in $\Gamma\in\cG(k,\ell)$ are a consequence of the definition of $\tau$ rather than an artifact of the recursive construction of $\cG(k,\ell)$, i.e., they can be created only once. In particular, $\Gamma$ can only contain a loop if it arises from $\cG_4$ or its analog in a previous iteration. Hence, the construction cannot yield a graph with more than two loops or more than one loop per circle, respectively. In particular, any graph in $\cG(k,\ell)$ containing more than one loop is necessarily obtained from an element in $\cG(1,1)$ with two loops.

\medskip
(c) Again, the elements of $\cG(S_1,\emptyset)=\emptyset$ and $\cG(\emptyset,S_2)=\emptyset$ with $S_1\subseteq[k]$ and $S_2\subseteq [k+1,k+\ell]$, as well as all $\Gamma\in\cG(1,1)$ have the claimed structure. We further note that any double edges of $\Gamma\in\cG_1\cup\cG_2\cup\cG_3$ must occur in the subgraph in $\cG(S_1,[k+1,k+\ell])$ with $S_1\subset[k]$ used to construct $\Gamma$ and that the induction hypothesis applies to the latter. For $\Gamma\in\cG_4$, a double edge occurs if and only if the corresponding element of $NCG(\{1,\dots,k,k+j,\dots,k+j-1,k+j\})$ has a vertex $i$ that shares an edge with both copies of $k+j$. However, there is at most one such vertex in $[k]$ and $[k+1,k+\ell]$, respectively, as having two vertices $i,i'$ connected to both copies of $k+j$ in either set induces a crossing. In particular, any double edge shares a vertex with the edge $(1,k+j)\in\Gamma$ prescribed by the definition. Hence, all graphs in $\cG_{\neg(1,k)}$ have the claimed structure.

\medskip
It remains to consider~$\cG_{(1,k)}$, i.e., to add an edge $(1,k)$ to the graphs considered previously. In the case $j=k$ of $\cG_2$, this doubles an existing $(1,k)$ edge. By part (a), any such graph must also have an edge connecting $k$ with a vertex in $[k+1,k+\ell]$. This shows that the placement rule for double edges is satisfied. Further, at most one doubled edge can be added with each application of the recursion, i.e., there are at most $k+\ell-1$ double edges in total. It is readily checked that two is indeed the highest edge multiplicity possible.

\medskip
(d) As the case $k=\ell=1$ has already been discussed in Example~\ref{ex-G11}, let $k=2$, $\ell=1$, and consider $\tau(NCG(\{1,2,3,3\}))$. We relabel the vertices as $1,2,3,3'$ to distinguish between the two copies of the doubled vertex 3 more easily.  Observe that the graphs $\Gamma_1$ and $\Gamma_2$ with edge sets $\{(1,3'), (2,3')\}$ and $\{(1,3'),(2,3)\}$, respectively, give rise to the same element of $\cG(2,1)$, namely the annulus graph with edge set $\{(1,3),(2,3)\}$ (see Fig.~\ref{fig-multi} below).

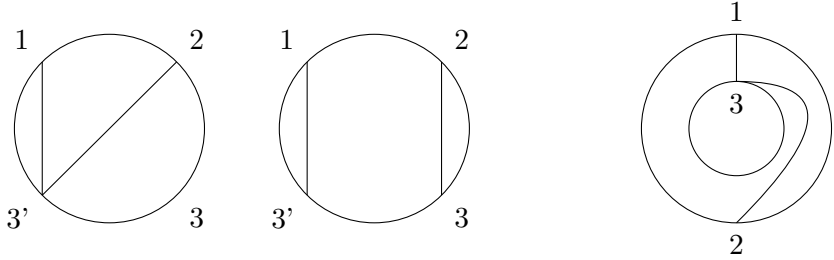
\begin{figure}[H]
\begin{center}
\begin{tikzpicture}[scale=1.25]
\draw (-0.707,0.707) node[above left=1pt] {1};
\draw (0.707,0.707) node[above right=1pt] {2};
\draw (0.707,-0.707) node[below right=1pt] {3};
\draw (-0.707,-0.707) node[below left=1pt] {3'};

\draw (0,1) node[above=1pt] {\color{white}0\color{black}};
\draw (0,-1) node[below=1pt] {\color{white}0\color{black}};

\draw (0,0) circle (1cm);

\draw (-0.707,0.707) -- (-0.707,-0.707);
\draw (0.707,0.707) -- (-0.707,-0.707);
\end{tikzpicture}\hspace{0.5cm}
\begin{tikzpicture}[scale=1.25]
\draw (-0.707,0.707) node[above left=1pt] {1};
\draw (0.707,0.707) node[above right=1pt] {2};
\draw (0.707,-0.707) node[below right=1pt] {3};
\draw (-0.707,-0.707) node[below left=1pt] {3'};

\draw (0,1) node[above=1pt] {\color{white}0\color{black}};
\draw (0,-1) node[below=1pt] {\color{white}0\color{black}};

\draw (0,0) circle (1cm);

\draw (-0.707,0.707) -- (-0.707,-0.707);
\draw (0.707,0.707) -- (0.707,-0.707);
\end{tikzpicture}\hspace{2cm}
\begin{tikzpicture}[scale=1.25]
\draw (0,1) node[above=1pt] {1};
\draw (0,-1) node[below=1pt] {2};

\draw (0,0.5) node[below=1pt] {3};

\draw (0,0) circle (1cm);
\draw (0,0) circle (0.5cm);

\draw (0,1) -- (0,0.5);
\draw (0,-1) .. controls (1,0) and (1,0.5).. (0,0.5);
\end{tikzpicture}
\end{center}
\captionof{figure}{Two disk non-crossing graphs that give rise to the same annulus graph.}\label{fig-multi}
\end{figure}

Hence, $\cG(1,2)$ is indeed a multi-set. By Definition~\ref{def-Ggraphs}, either $\cG(1,2)$ or $\cG(2,1)$ is used in the construction of $\cG(k,\ell)$ for $k,\ell>2$, i.e., the construction yields again a multi-set.
\end{proof}

For the following discussion, we introduce the disjoint decomposition
\begin{equation}\label{eq-Gdecomp3}
\cG(k,\ell)=\cG_{dec}(k,\ell)\cup\cG_{\neg dec}(k,\ell),
\end{equation}
where $\cG_{dec}(k,\ell)$ denotes the graphs in $\cG(k,\ell)$ that have double edges or loops and $\cG_{\neg dec}(k,\ell)$ denotes the graphs that do not have either. We refer to the elements of $\cG_{dec}(k,\ell)$ as \emph{decorated} graphs.

\begin{lemma}\label{lem-Gproperties2}
Let $k,\ell\in\N$.
\begin{itemize}
\item[(a)] Whenever $\Gamma\in\cG(k,\ell)$ has more than one connected component that contains a connecting edge, every connected component of $\Gamma$ can be uniquely identified with a disk non-crossing graph. 
\item[(b)] Whenever $\Gamma\in\cG(k,\ell)$ has exactly one connected component with a connecting edge and $\Gamma$ is not decorated, the same identification as in (a) holds, but it is no longer unique. If $\Gamma_1$ denotes the connected component of $\Gamma$ that contains a connecting edge and $U_1\cup U_2$ with $U_1\subset [k]$ and $U_2\subset [k+1,k+\ell]$ is the vertex set of $\Gamma_1$, there are $|U_1|\cdot |U_2|$ different ways to identify the connected components of $\Gamma$ with a disk non-crossing graph.
\end{itemize}
\end{lemma}

Lemma~\ref{lem-Gproperties2} translates between a graph $\Gamma\in\cG(k,\ell)$, the partition induced by its connected components and the cycle structure arising in~\eqref{eq-fpformula}. We give a schematic of the construction of the disk graphs in Fig.~\ref{fig-trafo1} below. Recall from Definition~\ref{def-NCA} that any cycle of an annular non-crossing permutation encloses a region homeomorphic to the unit disk with the boundary oriented clockwise.

\begin{figure}[H]
\begin{center}
\begin{tikzpicture}[scale=1,baseline=(current bounding box.center)]
\draw (0,2) node[above=1pt] {1};
\draw (1.414,1.414) node[above right=1pt] {2};
\draw (2,0) node[right] {3};

\draw (0.3536,0.3536) node[below left=0.1pt] {9};

\draw(0,2) -- (2,0);
\draw[black,style={double,double distance=2pt}](2,0) -- (0.3536,0.3536);
\draw (0.3536,0.3536) .. controls (0,1) and (1,1) .. (0.3536,0.3536);

\draw (2,0) arc[start angle=0, end angle=100, radius=2];
\draw (2,0) arc[start angle=0, end angle=-10, radius=2];
\draw (0.5,0) arc[start angle=0, end angle=90, radius=0.5];
\end{tikzpicture}$\rightarrow$
\begin{tikzpicture}[scale=1,baseline=(current bounding box.center)]
\draw (0,2) node[above=1pt] {1};
\draw (1.414,1.414) node[above right=1pt] {2};
\draw (2,0) node[right] {3};

\draw (0.3536,0.3536) node[below left=0.1pt] {9};

\draw (2,0) arc[start angle=0, end angle=100, radius=2];
\draw (2,0) arc[start angle=0, end angle=-10, radius=2];
\draw (0.5,0) arc[start angle=0, end angle=90, radius=0.5];

\draw[black] (0,2) -- (2,0);
\draw[black] (2,0) -- (0.3536,0.3536);
\draw[black] (0.3536,0.3536) -- (0,2);
\filldraw [black] (1.414,1.414) circle (1pt);
\end{tikzpicture}$\rightarrow$
\begin{tikzpicture}[scale=1,baseline=(current bounding box.center)]
\draw (0,2) node[above=1pt] {1};
\draw (1.414,1.414) node[above right=1pt] {2};
\draw (2,0) node[right] {3};

\draw (0.3536,0.3536) node[below left=0.1pt] {9};

\draw (2,0) arc[start angle=0, end angle=100, radius=2];
\draw (2,0) arc[start angle=0, end angle=-10, radius=2];
\draw (0.5,0) arc[start angle=0, end angle=90, radius=0.5];

\draw[black,postaction={on each segment={mid arrow=black}}] (0,2) -- (2,0);
\draw[black,postaction={on each segment={mid arrow=black}}] (2,0) -- (0.3536,0.3536);
\draw[black,postaction={on each segment={mid arrow=black}}] (0.3536,0.3536) -- (0,2);
\draw[black,postaction={on each segment={mid arrow=black}}] (1.414,1.414) .. controls (1.6,0.7) and (0.7,1.3) .. (1.414,1.414);
\end{tikzpicture}\hspace{2cm}
\begin{tikzpicture}[scale=1,baseline=(current bounding box.center)]
\draw (0,1) node[above=1pt] {1};
\draw (-0.707,-0.707) node[below left=1pt] {9};
\draw (0.707,-0.707) node[below right=1pt] {3};

\draw[black,style={double,double distance=2pt}](-0.707,-0.707) -- (0.707,-0.707);
\draw (0,1) -- (0.707,-0.707);
\draw (-0.707,-0.707) .. controls (-1,-0.1) and (-0.3,-0.3) .. (-0.707,-0.707);

\draw (0,0) circle (1cm);
\end{tikzpicture}
\end{center}
\captionof{figure}{A subgraph of $\Gamma\in\cG(k,\ell)$ with the induced partition $\pi_\Gamma$ and a possible permutation $\pi_\Gamma'$ (left) as well as the disk non-crossing graph obtained for a connected component of $\Gamma$ (right).}\label{fig-trafo1}
\end{figure}
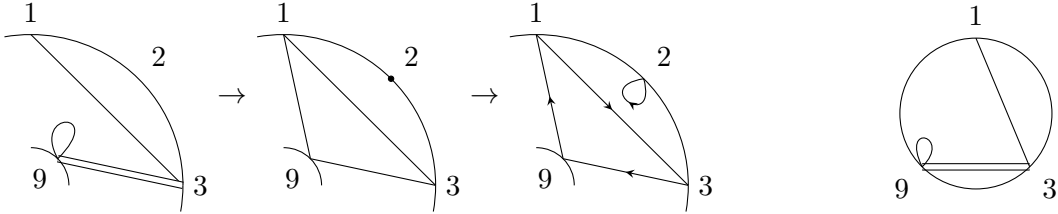

Note that the assignment of the orientation in the second step on the left of Fig.~\ref{fig-trafo1} is not unique if $\pi_\Gamma$ has only one connecting block (cf.~\cite[Prop.~4.6]{MingoNica2004}, see Fig.~\ref{fig-NCAvsNC} for an example). We remark that~(a) and (b) are almost complementary cases and that only decorated graphs with exactly one connected component containing a connecting edge are not covered by Lemma~\ref{lem-Gproperties2}. The proof of (a) further shows that any $\Gamma\in\cG(k,\ell)$ with at least two connected components containing a connecting edge cannot be decorated.

\medskip
We give the full construction behind the schematics in Fig.~\ref{fig-trafo1} below.

\begin{proof}[Proof of Lemma~\ref{lem-Gproperties2}]
(b) It follows from Lemma~\ref{lem-Gproperties1}(a) that the connected components of $\Gamma$ give rise to a partition $\pi_{\Gamma}\in NCP(k,\ell)$. By Propositions 4.4. and 4.6 of~\cite{MingoNica2004}, any element of $NCP(k,\ell)$ may be identified with an annular non-crossing permutation, i.e., its blocks can be given an orientation (see the left of Fig.~\ref{fig-trafo1}). This orientation of the individual cycles is naturally induced by the orientation of the inner and outer circle. However, the identification between $NCP(k,\ell)$ and $\NCA(k,\ell)$ is only unique whenever the underlying partition has more than one connecting block. If there is only one connecting block $U_1\cup U_2$ with $U_1\subset [k]$, $U_2\subset [k+1,k+\ell]$, there are $|U_1|\cdot |U_2|$ possibilities to identify $\pi_{\Gamma}$ with an annular non-crossing permutation (cf. Fig.~\ref{fig-NCAvsNC}). We fix one $\pi_\Gamma'\in \NCA(k,\ell)$ that is associated with $\pi_\Gamma$ in this way.

\medskip
By definition, the cycles of $\pi_\Gamma'$ can be drawn on the $(k,\ell)$-annulus such that each cycle encloses a region between the circles homeomorphic to the disk with boundary oriented clockwise. Adding the elements of the cycle around the boundary yields a labeled disk as in Definition~\ref{def-NCG}. We may use the same transformation to map a connected component of $\Gamma$ to a disk non-crossing graph (see Fig.~\ref{fig-trafo2} below). Note that this transformation cannot induce any crossings of the edges of $\Gamma$, however, any loops or double edges of the original graph are kept. Hence, we only obtain a disk non-crossing graph in the sense of Definition~\ref{def-NCG} if $\Gamma$ does not contain any loops or double edges to begin with.

\begin{figure}[H]
\begin{center}
\begin{tikzpicture}[scale=1,baseline=(current bounding box.center)]
\draw (0,2) node[above=1pt] {1};
\draw (1.414,1.414) node[above right=1pt] {2};
\draw (2,0) node[right] {3};

\draw (0.3536,0.3536) node[below left=0.1pt] {9};

\draw (2,0) arc[start angle=0, end angle=100, radius=2];
\draw (2,0) arc[start angle=0, end angle=-10, radius=2];
\draw (0.5,0) arc[start angle=0, end angle=90, radius=0.5];

\draw[black,postaction={on each segment={mid arrow=black}}] (0,2) -- (2,0);
\draw[black,postaction={on each segment={mid arrow=black}}] (2,0) -- (0.3536,0.3536);
\draw[black,postaction={on each segment={mid arrow=black}}] (0.3536,0.3536) -- (0,2);
\draw[black,postaction={on each segment={mid arrow=black}}] (1.414,1.414) .. controls (1.6,0.7) and (0.7,1.3) .. (1.414,1.414);
\end{tikzpicture}$\rightarrow$
\begin{tikzpicture}[scale=1,baseline=(current bounding box.center)]
\draw (0,1) node[above=1pt] {1};
\filldraw [black] (0,1) circle (1.5pt);
\draw (-0.707,-0.707) node[below left=1pt] {9};
\filldraw [black] (-0.707,-0.707) circle (1.5pt);
\draw (0.707,-0.707) node[below right=1pt] {3};
\filldraw [black] (0.707,-0.707) circle (1.5pt);
\draw (0,0) circle (1cm);
\end{tikzpicture}\hspace{2cm}
\begin{tikzpicture}[scale=1,baseline=(current bounding box.center)]
\draw (0,2) node[above=1pt] {1};
\draw (1.414,1.414) node[above right=1pt] {2};
\draw (2,0) node[right] {3};

\draw (0.3536,0.3536) node[below left=0.1pt] {9};

\draw(0,2) -- (2,0);
\draw[black,style={double,double distance=2pt}](2,0) -- (0.3536,0.3536);
\draw (0.3536,0.3536) .. controls (0,1) and (1,1) .. (0.3536,0.3536);

\draw (2,0) arc[start angle=0, end angle=100, radius=2];
\draw (2,0) arc[start angle=0, end angle=-10, radius=2];
\draw (0.5,0) arc[start angle=0, end angle=90, radius=0.5];
\end{tikzpicture}$\rightarrow$
\begin{tikzpicture}[scale=1,baseline=(current bounding box.center)]
\draw (0,1) node[above=1pt] {1};
\draw (-0.707,-0.707) node[below left=1pt] {9};
\draw (0.707,-0.707) node[below right=1pt] {3};

\draw[black,style={double,double distance=2pt}](-0.707,-0.707) -- (0.707,-0.707);
\draw (0,1) -- (0.707,-0.707);
\draw (-0.707,-0.707) .. controls (-1,-0.1) and (-0.3,-0.3) .. (-0.707,-0.707);

\draw (0,0) circle (1cm);
\end{tikzpicture}
\end{center}
\captionof{figure}{Transformation of a cycle of $\pi_{\Gamma}'\in\protect\NCA(k,\ell)$ to a circle (left) and the induced transformation of a connected component of $\Gamma\in\cG(k,\ell)$ into a disk graph (right).}\label{fig-trafo2}
\end{figure}
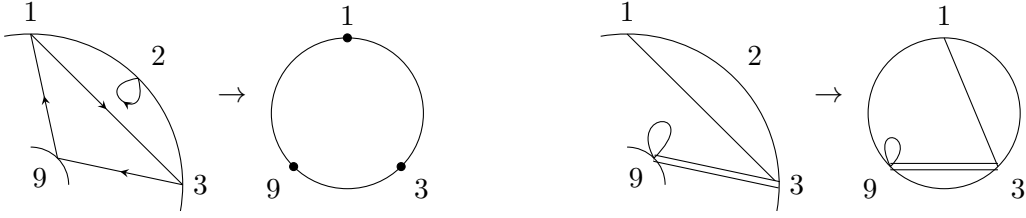

(a) Assume next that $\Gamma\in\cG(k,\ell)$ has at least two connected components that contain a connecting edge. It follows from Definition~\ref{def-Ggraphs} that this structure can only arise from $\cG_4$ in \eqref{eq-Gdecomp2} or its analog in a previous iteration of the recursive definition. Since adding subgraphs that only live on one circle of the $(k,\ell)$-annulus does not interfere with the following argument, assume w.l.o.g. that $\Gamma$ arises from $\cG_4$ directly. To have two connecting blocks, $\Gamma$ must have at least two connecting edges $(i_1,i_2)$ and $(i_1',i_2')$ where $i_1,i_1'\in[k]$, $i_2,i_2'\in [k+1,k+\ell]$ and $i_1\neq i_1'$, $i_2\neq i_2'$. Note that either $(i_1,i_2)$ or $(i_1',i_2')$ may coincide with the edge $(1,k+j)$ prescribed by the definition.

\medskip
 Let $\widetilde{\Gamma}\in NCG(\{1,\dots,k,k+j,\dots,k+j-1,k+j\})$ denote a disk non-crossing graph such that $\tau(\widetilde{\Gamma})=\Gamma$. Here, $\tau$ denotes the map introduced in Definition~\ref{def-tau}. By construction, $\widetilde{\Gamma}$ also has two edges edges $(i_1,i_2)$ and $(i_1',i_2')$ with $i_1,i_1'\in[k]$, $i_2,i_2'\in [k+1,k+\ell]$ and $i_1\neq i_1'$, $i_2\neq i_2'$. This structure imposes several restrictions on $\widetilde{\Gamma}$, as can be seen from Fig.~\ref{fig-trafo3} below.

\begin{figure}[H]
\begin{center}
\begin{tikzpicture}[scale=1.5,baseline=(current bounding box.center)]
\draw (0.3827,0.9239) node[above=1pt] {$1$};
\draw (0.9239,0.3827) node[right=1pt] {$i_1$};
\draw (0.9239,-0.3827) node[right=1pt] {$i_1'$};
\draw (0.3827,-0.9239) node[below=1pt] {$k$};
\draw (-0.3827,-0.9239) node[below=1pt] {$k+j$};
\draw (-0.9239,-0.3827) node[left=1pt] {$i_2'$};
\draw (-0.9239,0.3827) node[left=1pt] {$i_2$};
\draw (-0.3827,0.9239) node[above=1pt] {$k+j$};
\draw (-0.707,0.707) node[above left=1pt] {$k+j-1$};

\draw[red] (0,0) circle (1cm);
\draw (-0.9239,0.3827) -- (0.9239,0.3827);
\draw (-0.9239,-0.3827) -- (0.9239,-0.3827);
\draw (-0.3827,0.9239) -- (0.3827,0.9239);
\end{tikzpicture}\hspace{0.3cm}$\rightarrow$\hspace{0.25cm}
\begin{tikzpicture}[scale=1.5,baseline=(current bounding box.center)]
\draw (0,1) node[above=1pt] {1};
\draw (-0.4484,0.88) node[above left=1pt] {$k$};
\draw (0.707,-0.707) node[below right=1pt] {$i_1$};
\draw (-0.707,-0.707) node[below left=1pt] {$i_1'$};

\draw (0,0.5) node[below=0.3536pt] {$k\nsp+\nsp j$};
\draw (-0.3536,-0.3536) node[above right=0pt] {$i_2'$};
\draw (0.3536,-0.3536) node[above left=0.pt] {$i_2$};

\draw[red] (0,1) arc[start angle=90, end angle=-250, radius=1];
\draw[red] (0,1) arc[start angle=90, end angle=100, radius=1];
\draw[red] (0,0.5) arc[start angle=90, end angle=-250, radius=0.5];
\draw[red] (0,0.5) arc[start angle=90, end angle=100, radius=0.5];

\draw[red] (-0.1736,0.9848) -- (-0.0868,0.4924);
\draw[red] (-0.342,0.9397) -- (-0.171,0.4698);

\draw[black] (0,0.5) -- (0,1);
\draw[black] (-0.707,-0.707) -- (-0.3536,-0.3536);
\draw[black] (0.707,-0.707) -- (0.3536,-0.3536);
\end{tikzpicture}\hspace{0.25cm}$\rightarrow$\hspace{0.25cm}
\begin{tikzpicture}[scale=1.5,baseline=(current bounding box.center)]
\draw (0,1) node[above=1pt] {1};
\draw (-0.309,0.951) node[above left=1pt] {$k$};
\draw (0.707,-0.707) node[below right=1pt] {$i_1$};
\draw (-0.707,-0.707) node[below left=1pt] {$i_1'$};

\draw (0,0.5) node[below=0.3536pt] {$k\nsp+\nsp j$};
\draw (-0.3536,-0.3536) node[above right=0pt] {$i_2'$};
\draw (0.3536,-0.3536) node[above left=0.pt] {$i_2$};

\draw (0,0) circle (1cm);
\draw (0,0) circle (0.5cm);

\draw[black] (0,0.5) -- (0,1);
\draw[black] (-0.707,-0.707) -- (-0.3536,-0.3536);
\draw[black] (0.707,-0.707) -- (0.3536,-0.3536);
\end{tikzpicture}

\end{center}
\captionof{figure}{A schematic visualization of $\widetilde{\Gamma}$ (left) and its mapping to the original graph $\Gamma=\tau(\widetilde{\Gamma})$ (right).}\label{fig-trafo3}
\end{figure}
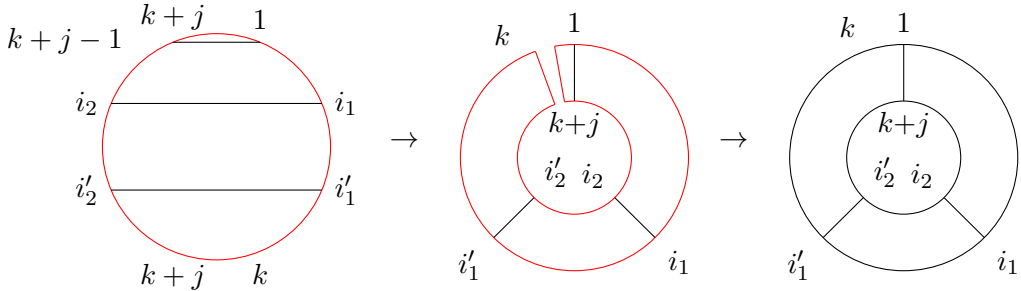

First, there cannot be an edge $(1,k)\in\widetilde{\Gamma}$ without violating the non-crossing condition. Together with~\eqref{eq-Gdecomp1}, this implies that $\Gamma=\tau(\widetilde{\Gamma})$ has at most a single edge $(1,k)$. Further, $\widetilde{\Gamma}$ cannot contain a vertex that connects to both copies of $k+j$. This implies that $\Gamma=\tau(\widetilde{\Gamma})$ cannot have any double edges. Lastly, $\Gamma$ cannot have any loops, as $\widetilde{\Gamma}$ containing an edge $(k+j,k+j)$ would induce a crossing, too. Hence, any $\Gamma\in\cG(k,\ell)$ with more than one connected component containing a connecting edge has only single edges and no loops.
\end{proof}

So far, we have only considered the non-crossing annular partition induced by an element of $\cG(k,\ell)$. The following lemma allows us to partially reverse this relation and explicitly construct a graph that is associated with a given $\pi\in NCP(k,\ell)$.

\begin{lemma}\label{lem-Gproperties3}
For every $\pi\in NCP(k,\ell)$ there is at least one $\Gamma\in\cG(k,\ell)$ for which the vertex sets of the connected components coincide with the blocks of $\pi$. If $\pi$ has exactly one connecting block, then there is at least one such graph in $\cG_{dec}(k,\ell)$ and one in $\cG_{\neg dec}(k,\ell)$, respectively.
\end{lemma}

We briefly sketch the construction of an element on $\cG(k,\ell)$ from a given annular non-crossing partition. For the example, we assume that the sketched connecting block is the only one in the partition. After completing the steps sketched in Fig.~\ref{fig-construct1}, the remaining connected components of the graph are readily added using steps corresponding to $\cG_1$, $\cG_2$, and $\cG_3$ in Definition~\ref{def-Ggraphs} and the symmetry under interchanging of the inner and outer circle.

\begin{figure}[H]
\begin{center}
\begin{tikzpicture}[scale=1,baseline=(current bounding box.center)]
\draw (0,2) node[above=1pt] {1};
\draw (2,0) node[right] {2};
\draw (0.3536,0.3536) node[below left=0.1pt] {7};

\draw (2,0) arc[start angle=0, end angle=100, radius=2];
\draw (2,0) arc[start angle=0, end angle=-10, radius=2];
\draw (0.5,0) arc[start angle=0, end angle=90, radius=0.5];

\draw[black] (0,2) -- (2,0);
\draw[black] (2,0) -- (0.3536,0.3536);
\draw[black] (0.3536,0.3536) -- (0,2);
\end{tikzpicture}\hspace{0.1cm}$\rightarrow$\hspace{0.1cm}
\begin{tikzpicture}[scale=1,baseline=(current bounding box.center)]
\draw (0,2) node[above=1pt] {1};
\draw (2,0) node[right] {2};

\draw (0.3536,0.3536) node[below left=0.1pt] {7};

\draw (2,0) arc[start angle=0, end angle=100, radius=2];
\draw (2,0) arc[start angle=0, end angle=-10, radius=2];
\draw (0.5,0) arc[start angle=0, end angle=90, radius=0.5];

\draw[black,postaction={on each segment={mid arrow=black}}] (0,2) -- (2,0);
\draw[black,postaction={on each segment={mid arrow=black}}] (2,0) -- (0.3536,0.3536);
\draw[black,postaction={on each segment={mid arrow=black}}] (0.3536,0.3536) -- (0,2);
\end{tikzpicture}$\rightarrow$
\begin{tikzpicture}[scale=1,baseline=(current bounding box.center)]
\draw (0.707,0.707) node[above right=1pt] {1};
\draw (-0.707,0.707) node[above left=1pt] {7'};
\draw (-0.707,-0.707) node[below left=1pt] {7};
\draw (0.707,-0.707) node[below right=1pt] {2};

\draw (0,0) circle (1cm);

\draw[black] (0.707,0.707) -- (0.707,-0.707);
\draw[black] (0.707,0.707) -- (-0.707,-0.707);
\end{tikzpicture}\hspace{0.2cm}$\overset{\tau}{\rightarrow}$\hspace{0.2cm}
\begin{tikzpicture}[scale=1,baseline=(current bounding box.center)]
\draw (0,1) node[above=1pt] {1};
\draw (0,-1) node[below=1pt] {2};

\draw (0,0.5) node[below=1pt] {7};

\draw (0,0) circle (1cm);
\draw (0,0) circle (0.5cm);

\draw (0,1) -- (0,0.5);
\draw (0,-1) .. controls (1,0) and (1,0.25).. (0,1);
\end{tikzpicture}
\end{center}
\captionof{figure}{The construction of an element of $\cG(2,1)$ (right) from the connecting block of an annular non-crossing partition (left).}\label{fig-construct1}
\end{figure}
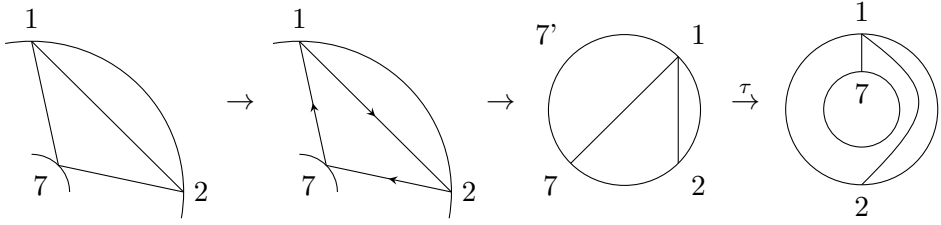

\begin{proof}[Proof of Lemma~\ref{lem-Gproperties3}]
Fix $\pi\in NCP(k,\ell)$ and assume first that $\pi$ has at least two connecting blocks. W.l.o.g. let 1 and $k$ be assigned to different blocks of $\pi$. Indeed, if 1 and $k$ occurred in the same block, we may split it into two disjoint parts that contain 1 and~$k$, respectively, and later add an edge $(1,k)$ to the graph obtained from this modified partition. Further, assume w.l.o.g. that 1 is contained in a connecting block. If $i>1$ were the smallest element that occurs in a connecting block, we may further remove any blocks containing $\{1,\dots,i-1\}$ from the partition and later add a suitable subgraph. Note that the latter is possible by only using steps corresponding to $\cG_1$ and $\cG_2$ in Definition~\ref{def-Ggraphs}.

\medskip
Under these assumptions, there exists $\widetilde{\pi}\in NCP(\{1,\dots,k,k+j,\dots,k+j-1,k+j\})$ such that the transformation $\tau$ in Definition~\ref{def-tau} maps the blocks of $\widetilde{\pi}$ to the blocks of $\pi$. This is readily seen from Figure~\ref{fig-deftau}, as 1 and $k$ being in different blocks implies that none of the blocks of $\pi$ intersects the slit between 1 and $k$. We pick $j$ such that $1$ and $k+j$ are in the same block of $\pi$ and may further choose $\widetilde{\pi}$ such that one copy of the doubled label $k+1$ occurs as a singleton set. Finally, define $\widetilde{\Gamma}\in NCG(\{1,\dots,k,k+1,\dots,k+\ell,k+1\})$ by considering each block $B=\{i_1,\dots,i_n\}\in\widetilde{\pi}$ separately and adding the edges $(i_1,i_2),\dots,(i_1,i_n)$ to the graph. By construction, the vertex sets of the connected components of $\widetilde{\Gamma}$ coincide with the blocks of $\widetilde{\pi}$. Considering $\Gamma_\pi=\tau(\widetilde{\Gamma})$ now yields the element of $\cG(k,\ell)$ with the claimed properties (cf. Fig.~\ref{fig-construct1}, where the procedure is sketched for a single block).

\medskip
Next, consider $\pi\in NCP(k,\ell)$ that has only one connecting block $U=U_1\cup U_2$ with $U_1\subseteq[k]$, $U_2\subseteq[k+1,k+\ell]$. Similar to the first case, we can construct a disk non-crossing graph $\Gamma_U$ for which $\tau(\Gamma_U)\in\cG_c(|U_1|,|U_2|)$. As $\tau(\Gamma_U)$ is only required to have one connecting edge, there is also a choice for $\Gamma_U$ containing a $(k+j,k+j)$ edge (cf. Fig.~\ref{fig-trafo2}). In particular, we obtain at least one graph in $\cG_{c,\neg dec}(|U_1|,|U_2|)$ and one graph in $\cG_{c,dec}(|U_1|,|U_2|)$, respectively. The remaining connected subgraphs of $\Gamma_\pi$ are then added recursively by alternating between adding an isolated vertex (cf. $\cG_1$) and a disk non-crossing graph~(cf.~$\cG_2$ and $\cG_3$) to $\tau(\Gamma_U)$. Note that starting from $\tau(\Gamma_U)\in\cG_{c,\neg dec}(|U_1|,|U_2|)$ yields a graph that is not decorated while starting from $\tau(\Gamma_U)\in\cG_{c,dec}(|U_1|,|U_2|)$ yields a decorated graph.
\end{proof}

Using the properties of $\Gamma\in\cG(k,\ell)$ from Lemmas~\ref{lem-Gproperties2} and~\ref{lem-Gproperties3}, we obtain an explicit non-recursive combinatorial formula for $m_{\circ\circ}$.

\begin{lemma}\label{lem-graphformula2}
Let $k,\ell\in\N$. Then,
\begin{equation}\label{eq-2ndordercumus}
m_{\circ\circ}[1,\dots,k|k+1,\dots,k+\ell]=\Big(\prod_{s=1}^{k+\ell}m_s\Big)\sum_{\Gamma\in\cG_{c}(k,\ell)}c_{\Gamma}q_{\Gamma},
\end{equation}
with suitable constants $c_\Gamma\in\Z$. In particular, $c_{\Gamma}=1$ for $\Gamma\in\cG_{dec}(k,\ell)$.
\end{lemma}
Note that the constants $c_{\Gamma}$ for $\Gamma\in\cG_{\neg dec}(k,\ell)$ are readily obtained from the multiplicity of the corresponding graph in the multi-set, however, their exact values are not needed for the proof of Theorem~\ref{thm-main}.

\begin{proof}
We start by observing that any connected component of $\Gamma\in\cG(k,\ell)$ with a connecting edge is itself an annular non-crossing graph. Further, a connected component of~$\Gamma$ that only involves vertices from either $[k]$ or $[k+\ell]$ cannot contain loops or double edges~(cf.~(b) and (c) of Lemma~\ref{lem-Gproperties1}) and we may identify it with a disk non-crossing graph.

\medskip
To simplify notation, decompose any $B\subseteq S_1\cup S_2$ as a union $B_1\cup B_2$ with $B_1\subseteq [k]$, $B_2\subseteq [k+1,k+\ell]$ and associate it with a tuple $(B_1,B_2)$ if neither of the subsets is empty. This allows us to use the common notation $NCG(B)$ for both disk and annular graphs by setting $NCG(B)=\cG(B_1,B_2)$ whenever $B$ contains elements from both $[k]$ and $[k+1,k+\ell]$, and $NCG(B)=NCG(B_1)$ resp. $NCG(B)=NCG(B_2)$ if it does not. Splitting the sum in~\eqref{eq-graphform} according to the underlying partition, we can write
\begin{align}
m[1,\dots,k|k+1,\dots,k+\ell]&=\Big(\prod_{s=1}^{k+\ell}m_s\Big)\sum_{\Gamma\in\cG(k,\ell)}q_{\Gamma}\NN\\
&=\sum_{\pi\in NCP(k,\ell)}\prod_{B\in\pi}\Big[\Big(\prod_{s\in B}m_s\Big)\sum_{\Gamma\in NCG_c(B)}q_{\Gamma}\Big]\label{eq-graphtocc}
\end{align}
where $NCG_c(B)$ denotes the connected graphs in $NCG(B)$. By Lemma~\ref{lem-Gproperties3}, none of the sums on the right-hand side are empty. However, there may be multiple permutations in $\NCA(k,\ell)$ as well as a marked element of $NCP(k)\times NCP(\ell)$ associated with a given annular non-crossing partition $\pi\in NCP(k,\ell)$. To obtain the same structure as in~\eqref{eq-mcrelation2}, we thus need to decompose the sum over $\pi\in NCP(k,\ell)$ on the right-hand side of \eqref{eq-graphtocc} further.

\medskip
Distinguishing by the number of connecting blocks of $\pi$ yields
\begin{align}
&\sum_{\pi\in NCP(k,\ell)}\prod_{B\in\pi}\Big[\Big(\prod_{s\in B}m_s\Big)\sum_{\Gamma\in NCG_c(B)}q_{\Gamma}\Big]\label{eq-ccdecomp}\\
&=\sum_{\substack{\pi\in NCP(k,\ell),\\ 1\ conn.\ block}}\prod_{B\in\pi}\Big[\Big(\prod_{s\in B}m_s\Big)\sum_{\Gamma\in NCG_c(B)}q_{\Gamma}\Big]+\sum_{\substack{\pi\in NCP(k,\ell),\\\geq2\ conn.\ blocks}}\prod_{B\in\pi}\Big[\Big(\prod_{s\in B}m_s\Big)\sum_{\Gamma\in NCG_c(B)}q_{\Gamma}\Big].\NN
\end{align}
As an annular non-crossing partition with at least two connecting blocks uniquely corresponds to an annular non-crossing permutation~(cf.~\cite[Prop.~4.4]{MingoNica2004}), we may replace $NCP(k,\ell)$ by $\NCA(k,\ell)$ in the summation if we interpret a cycle as an ordered set. Together with Lemma~\ref{lem-Gproperties2} and~\eqref{eq-mcircgraphs}, it follows that
\begin{align}
\sum_{\substack{\pi\in NCP(k,\ell),\\\geq2\ conn.\ blocks}}\prod_{B\in\pi}\Big[\Big(\prod_{s\in B}m_s\Big)\sum_{\Gamma\in NCG_c(B)}q_{\Gamma}\Big]&=\sum_{\substack{\pi\in \NCA(k,\ell),\\ \geq2\ conn.\ cycles}}\prod_{B\in\pi}\Big[\Big(\prod_{s\in B_1\cup B_2}m_s\Big)\sum_{\Gamma\in NCG_c(B_1\cup B_2)}q_{\Gamma}\Big]\NN\\
&=\sum_{\substack{\pi\in \NCA(k,\ell),\\ \geq2\ conn.\ cycles}}\prod_{B\in\pi}m_{\circ}[B_1\cup B_2].\label{eq-mpart1}
\end{align}
where we used that $m_{\circ}$ is invariant under permutation of the spectral parameters $(z_j)_j$.

\medskip
Applying a similar argument for the partitions with one connecting block $U$, interpreted as $(U_1,U_2)$ with $U_1=U\cap[k], U_2=U\cap[k+1,k+\ell]$, yields
\begin{align}
&\sum_{\substack{\pi\in NCP(k,\ell),\\ 1\ conn.\ block}}\prod_{B\in\pi}\Big[\Big(\prod_{s\in B}m_s\Big)\sum_{\Gamma\in NCG_c(B)}q_{\Gamma}\Big]\NN\\
&=\sum_{\substack{\pi\in NCP(k,\ell),\\1\ conn.\ block}}\Big[\Big(\prod_{s\in U}m_s\Big)\sum_{\Gamma\in\cG_c(U)}q_{\Gamma}\Big]\prod_{B\in\pi\setminus \{U\}}\Big[\Big(\prod_{s\in B}m_s\Big)\sum_{\Gamma\in NCG_c(B)}q_{\Gamma}\Big]\NN\\
&=\sum_{\substack{\pi\in NCP(k,\ell),\\1\ conn.\ block}}\Big[\Big(\prod_{s\in U}m_s\Big)\sum_{\Gamma\in\cG_c(U)}q_{\Gamma}\Big]\prod_{B\in\pi\setminus \{U\}}m_{\circ}[B]\label{eq-1ccterm}
\end{align}
by~\eqref{eq-mcircgraphs}. Note that $B\neq U$ cannot involve both sets $[k]$ and $[k+1,k+\ell]$, as $U$ is the only connecting block. Using~\eqref{eq-Gdecomp3}, we decompose
\begin{displaymath}
\sum_{\Gamma\in\cG_c(U)}q_{\Gamma}=\sum_{\Gamma\in\cG_{c,dec}(U)}q_{\Gamma}+\sum_{\Gamma\in\cG_{c,\neg dec}
(U)}q_{\Gamma}
\end{displaymath}
and recall from Lemma~\ref{lem-Gproperties3} that neither sum on the right-hand side is empty. Note that the split induced by~\eqref{eq-Gdecomp3} also decomposes the right-hand side of~\eqref{eq-1ccterm} into two terms.

\medskip
Next, consider the term corresponding to $\cG_{c,\neg dec}(U)$ and recall that any $\Gamma\in\cG_{c,\neg dec}(U)$ can be identified with a disk non-crossing graph by Lemma~\ref{lem-Gproperties2}(b). However, this identification is not unique. Decompose $U$ into $U_1=U\cap[k],U_2=U\cap[k+1,k+\ell]$. Then there are $|U_1|\cdot |U_2|$ different disk graphs that can be obtained from a given $\Gamma\in\cG_{c,\neg dec}(U)$. As the resulting graphs only differ in the labeling of the vertices and $m_{\circ}$ is invariant under permutation of its arguments, the contribution of each graph to the sum is the same. Recall that the number of annular non-crossing permutations arising from $\pi\in NCP(k,\ell)$ is equal to $|U_1|\cdot |U_2|$ by~\cite[Prop.~4.6]{MingoNica2004}. We thus write
\begin{displaymath}
\sum_{\Gamma\in \cG_{c,\neg dec}(U)}q_{\Gamma}=|U_1|\cdot|U_2|\sum_{\Gamma\in NCG_c(U_1\cup U_2)}q_{\Gamma}+\sum_{\Gamma\in \cG_{c,\neg dec}(U)}c_{\Gamma}q_{\Gamma}.
\end{displaymath}
with suitable constants $c_{\Gamma}\in\Z$. In particular, we do not necessarily have $c_{\Gamma}=1$. This can be seen from considering the element of $\cG(2,1)$ that has the edge set $\{(1,2),(1,3),(2,3)\}$. Hence,
\begin{equation}\label{eq-multiplicity}
\sum_{\Gamma\in \cG_{c}(U)}q_{\Gamma}=|U_1|\cdot|U_2|\sum_{\Gamma\in NCG_c(U_1\cup U_2)}q_{\Gamma}+\sum_{\Gamma\in \cG_{c}(U)}c_{\Gamma}q_{\Gamma}.
\end{equation}
where we set $c_{\Gamma}=1$ for $\Gamma\in\cG_{c,dec}(U)$. Note that the contribution of the first term of~\eqref{eq-multiplicity} to~\eqref{eq-1ccterm} is
\begin{align}
&\sum_{\substack{\pi\in NCP(k,\ell),\\1\ conn.\ block\ U}}\Big[\Big(\prod_{s\in U}m_s\Big)\cdot|U_1|\cdot|U_2|\sum_{\Gamma\in NCG_c(U_1\cup U_2)}q_{\Gamma}\Big]\prod_{B\in\pi\setminus \{U\}}m_{\circ}[B]\NN\\
&=\sum_{\substack{\pi\in \NCA(k,\ell),\\1\ conn.\ block}}\prod_{B\in\pi}m_{\circ}[B_1\cup B_2].\label{eq-mpart2}
\end{align}
Putting everything together, the right-hand side of~\eqref{eq-ccdecomp} reads
\begin{align}
&m[1,\dots,k|k+1,\dots,k+\ell]=\sum_{\pi\in \NCA(k,\ell)}\prod_{B\in\pi}m_{\circ}[B]\label{eq-mparts}\\
&\quad\quad\quad\quad\quad\quad\quad\quad\quad+\sum_{\substack{\pi\in NCP(k)\times NCP(\ell),\\ U_1,U_2\text{ marked}}} \prod_{B\in\pi\setminus \{U_1,U_2\}}m_{\circ}[B_1\cup B_2]\Big[\Big(\prod_{s\in U}m_s\Big)\sum_{\Gamma\in\cG_{c}(U)}c_{\Gamma}q_\Gamma\Big].\NN
\end{align}
with $c_{\Gamma}\in\Z$ as in~\eqref{eq-multiplicity} and $U=U_1\cup U_2$. The first term on the right-hand side of~\eqref{eq-mparts} is the sum of~\eqref{eq-mpart1} and~\eqref{eq-mpart2}. The second term is obtained from the second term in~\eqref{eq-multiplicity} by noting that $U$ is the only connecting block of the partition, i.e., any block of $\pi\setminus\{U\}$ only lives on one of the circles. Observing that this matches the structure in~\eqref{eq-mcrelation2}, we obtain ~\eqref{eq-2ndordercumus} by comparing term-by-term.
\end{proof}

We finally have all the necessary tools to prove Theorem~\ref{thm-main}.

\subsection{Part 2: Conclusion}\label{sect-proof-fpformula}

\begin{proof}[Proof of Theorem~\ref{thm-main}]
Let $\mathfrak{f}$ denote the right-hand side of~\eqref{eq-fpformula}. We recall that the initial condition~\eqref{eq-Minitial} is immediate from Definition~\ref{def-Ggraphs}(ii) and that the symmetry in Definition~\ref{def-M}(i) follows from the remark on Theorem~\ref{thm-main} above. Hence, it only remains to check that $\mathfrak{f}$ satisfies the recursion~\eqref{eq-Mrecursion}. To simplify notation, we interpret $\mathfrak{f}$ as a function of the multi-indices $\alpha$ and $\beta$.

\medskip
Similar to the proof of~\cite[Lem.~4.4]{CES-thermalization}, we use~\eqref{eq-mcircgraphs} and~\eqref{eq-2ndordercumus} to write out the first and second-order free cumulant functions in terms of suitable graphs. This yields
\begin{align}
\frac{\mathfrak{f}[\alpha|\beta]}{m_1\dots m_{k+\ell}}&=\sum_{\pi\in \NCA(k,\ell)}\Big(\prod_{B\in K(\pi)}\Big\langle \prod_{j\in B}A_j\Big\rangle\Big)\prod_{B\in\pi}\Big(\sum_{\Gamma\in NCG_c(B_1\cup B_2)}q_\Gamma\Big)\label{eq-longF}\\
&\quad+\sum_{\substack{\pi\in NCP(k)\times NCP(\ell),\\ U_1,U_2 \text{ marked}}}\Big(\prod_{B\in K(\pi)}\Big\langle \prod_{j\in B} A_j\Big\rangle\Big)\Big(\sum_{\Gamma\in NCG_c(U_1,U_2)}c_{\Gamma}q_{\Gamma}\Big)\NN\\
&\quad\quad\times\prod_{B\in\pi\setminus\{U_1,U_2\}}\Big(\sum_{\Gamma\in NCG_c(B_1\cup B_2)}q_{\Gamma}\Big).\NN
\end{align}
To collect the terms involving the deterministic matrices $A_1,\dots,A_k$, we define
\begin{equation}\label{eq-defF1}
\cF(\pi):=\prod_{B\in K(\pi')}\Big\langle \prod_{j\in B}A_j\Big\rangle
\end{equation}
for any $\pi\in NCP(k,\ell)$ that has more than one connecting block. Here, $\pi'\in \NCA(k,\ell)$ denotes the unique permutation for which the blocks of $\pi'$ coincide with $\pi$. Further, we define
\begin{equation}\label{eq-defF2}
\cF(\pi):=\sum_{\substack{\pi'\in \NCA(k,\ell)\\ blocks(\pi')=\pi}}\Big(\prod_{B\in K(\pi')}\Big\langle \prod_{j\in B}A_j\Big\rangle\Big)+\prod_{B\in K(\pi'')}\Big\langle \prod_{j\in B}A_j\Big\rangle
\end{equation}
for any $\pi\in NCP(k,\ell)$ that has exactly one connecting block $U$. Here, $\pi''$ is the marked partition obtained from $\pi$ by splitting the connecting block into $U_1=U\cap [k]$ and $U_2=U\cap[k+1,k+\ell]$, and marking $U_1,U_2$ on the respective circles.

\medskip
Next, we decompose the sum over $\NCA(k,\ell)$ in~\eqref{eq-longF} according to the number of connecting cycles in the permutation and rewrite the right-hand side as
\begin{align*}
\frac{\mathfrak{f}[\alpha|\beta]}{m_1\dots m_{k+\ell}}&=\sum_{\pi\in NCP(k,\ell)}\cF(\pi)\prod_{B\in\pi}\Big(\sum_{\Gamma\in NCG_c(B)}q_\Gamma\Big)
\end{align*}
using~\eqref{eq-defF1} and~\eqref{eq-defF2}. Recall that any connected component of $\Gamma\in\cG(k,\ell)$ can either be identified with a disk non-crossing graph or itself satisfies Definition~\ref{def-Ggraphs}. We can thus interpret $NCG_c(B)$ as a connected component of a bigger graph and rewrite
\begin{equation}\label{eq-shortF}
\sum_{\pi\in NCP(k,\ell)}\cF(\pi)\prod_{B\in\pi}\Big(\sum_{\Gamma\in NCG_c(B)}q_\Gamma\Big)=\sum_{\Gamma\in\cG(k,\ell)}\cF(\pi_{\Gamma})q_{\Gamma}.
\end{equation}
Here, $\pi_\Gamma$ denotes the partition arising from the vertex sets of the connected components of $\Gamma$.

\medskip
Using~\eqref{eq-Gdecomp1} from Definition~\ref{def-Ggraphs}, decompose the right-hand side of~\eqref{eq-shortF} as
\begin{align}
\sum_{\Gamma\in\cG(k,\ell)}\cF(\pi_{\Gamma})q_{\Gamma}=\sum_{\Gamma\in\cG_{\neg(1,k)}}q_{\Gamma}\Big(\cF(\pi_{\Gamma})+q_{1,k}\cF(\pi_{\Gamma\cup\{(1,k)\}})\Big).\label{eq-Fdecomp1}
\end{align}
Recall that $\Gamma\cup\{(1,k)\}$ denotes the graph obtained from adding an edge $(1,k)$ to $\Gamma$ and that, therefore, $q_{\Gamma\cup\{(1,k)\}}=q_{1,k}q_{\Gamma}$. Next, apply~\eqref{eq-Gdecomp2} to split the right-hand side of~\eqref{eq-Fdecomp1} into contributions corresponding to $\cG_1$, $\cG_2$, $\cG_3$, and $\cG_4$. This  yields
\begin{align}
\sum_{\Gamma\in\cG(k,\ell)}\cF(\pi_{\Gamma})q_{\Gamma}&=\sum_{\Gamma\in\cG([2,k],[k+1,k+\ell])}q_{\Gamma}\Big(\cF(\pi_{\Gamma})+q_{1,k}\cF(\pi_{\Gamma\cup\{(1,k)\}})\Big)\label{eq-Fdecomp2}\\
&+\sum_{j=2}^k\sum_{\Gamma\in NCG_{(1,j)}([1,j])\times \cG([j,k],[k+1,k+\ell])}q_{\Gamma}\Big(\cF(\pi_{\Gamma})+q_{1,k}\cF(\pi_{\Gamma\cup\{(1,k)\}})\Big)\NN\\
&+\sum_{j=1}^{k-1}\sum_{\Gamma\in\cG_{(1,j)}([1,j],[k+1,k+\ell])\times NCG([j,k)]}q_{\Gamma}\Big(\cF(\pi_{\Gamma})+q_{1,k}\cF(\pi_{\Gamma\cup\{(1,k)\}})\Big)\NN\\
&+\sum_{j=1}^\ell\sum_{\Gamma\in NCG_{(1,k+j)}(\{1,k,k+j,\dots,k+\ell,\dots,k+j\})}q_{\Gamma}\Big(\cF(\pi_{\Gamma})+q_{1,k}\cF(\pi_{\Gamma\cup\{(1,k)\}})\Big),\NN
\end{align}
where the edges prescribed in the definitions of $\cG_2$, $\cG_3$, and $\cG_4$ are added as a subscript to $NCG$ and $\cG$, respectively. It remains to compare~\eqref{eq-Fdecomp2} with~\eqref{eq-Mrecursion}. 

\medskip
\underline{\smash{Computation of the first line of~\eqref{eq-Fdecomp2}}}: Recall that $\sum_{\Gamma\in\cG_1}q_{\Gamma}=\sum_{\Gamma\in\cG([2,k],[k+1,k+\ell])}q_{\Gamma}$. As any $\Gamma\in\cG_1$ has an isolated vertex 1, the set $\{1\}$ appears as a singleton block of the underlying partition $\pi_{\Gamma}$. Let $\pi_{\Gamma}'$ denote an annular non-crossing permutation with blocks given by $\pi_{\Gamma}$. Then $(\dots k1\dots)\in K(\pi_{\Gamma}')$, i.e., 1 and $k$ appear in the same cycle of the Kreweras complement. This gives
\begin{equation}\label{eq-Fsimplified1}
\prod_{B\in K(\pi_{\Gamma}')}\Big\langle \prod_{j\in B}A_j\Big\rangle=\prod_{B\in K(\pi_{\Gamma}')|_{(1,k]\cup [k+1,k+\ell]}}\Big\langle \prod_{j\in B}A_j'\Big\rangle
\end{equation}
where $A_j'=A_j$ for $j=2,\dots k-1$ and $j\in [k+1,k+\ell]$, but $A_k'=A_kA_1$. Considering $\cF(\pi_{\Gamma\cup(1,k)})$, note that adding an edge $(1,k)$ to the graph $\Gamma$ implies that $(k)\in K(\pi_{\Gamma\cup\{(1,k)\}})$, i.e., $\langle A_k\rangle$ always occurs as a separate factor in $\cF$. Hence, whenever $\pi_\Gamma$ has more than one connecting block we can use~\eqref{eq-defF1} to evaluate
\begin{align}
&\cF(\pi_{\Gamma})+q_{1,k}\cF(\pi_{\Gamma\cup\{(1,k)\}})\NN\\
&=\prod_{B\in K(\pi_{\Gamma}')|_{\langle1,k]\cup [k+1,k+\ell]}}\Big\langle \prod_{j\in B}A_j'\Big\rangle+q_{1,k}\langle A_k\rangle\prod_{B\in K(\pi_{\Gamma}')|_{(1,k)\cup [k+1,k+\ell]}}\Big\langle \prod_{j\in B}A_j\Big\rangle\label{eq-Fsimplified2}
\end{align}
with $A_j'$ as in~\eqref{eq-Fsimplified1}. In the remaining cases, $\pi_{\Gamma}$ has only one connecting block $U$, and $\cF$ is a sum of two terms. As the first term of~\eqref{eq-defF2} can be evaluated similarly to~\eqref{eq-Fsimplified2}, we only consider the second term. Let $\pi_{\Gamma}''$ denote the element of $NCP(k)\times NCP(\ell)$ in which the blocks $U_1=U\cap [k]$ and $U_2=U\cap [k+1,k+\ell]$ are marked. Recall that the marking does not influence the Kreweras complement, which is taken for both circles separately here. As 1 and $k$ lie on the same circle, we can argue as in the permutation case. It follows~that
\begin{align}
&\Big(\prod_{s=1}^{k+\ell}m_s\Big)\sum_{\Gamma\in\cG([2,k],[k+1,k+\ell])}q_{\Gamma}\Big(\cF(\pi_{\Gamma})+q_{1,k}\cF(\pi_{\Gamma\cup\{(1,k)\}})\Big)\NN\\ 
&=m_1\big(F[T_2,\dots,T_{k-1},G_kA_kA_1|T_{k+1},\dots,T_{k+\ell}]\NN\\
&\quad+q_{1,k}F[T_2,\dots,T_{k-1}G_kA_1|T_{k+1},\dots,T_{k+\ell}]\langle A_k\rangle\big).\label{eq-rewrittenG1}
\end{align}

\underline{\smash{Computation of the second line of~\eqref{eq-Fdecomp2}}}: For the contribution arising from $\cG_2$, recall that $\Gamma=\Gamma_1\cup\Gamma_2$ with $\Gamma_1\in NCG_{(1,j)}([1,j])$ and $\Gamma_2\in\cG([j,k],[k+1,k+\ell])$ implies $q_{\Gamma}=q_{\Gamma_1}q_{\Gamma_2}$, i.e., the weights factorize. Further, the term involving $\Gamma_1$ evaluates to
\begin{equation}\label{eq-goodweights}
\sum_{\Gamma_1\in NCG_{(1,j)}([1,j])}q_{\Gamma_1}=\frac{q_{1,j}}{1+q_{1,j}}\sum_{\Gamma_1\in NCG[1,j])}q_{\Gamma_1}=m_1m_j\sum_{\Gamma_1\in NCG[1,j])}q_{\Gamma_1}
\end{equation}
as $q_{\Gamma_1}$ always includes a factor $q_{1,j}$ if $\Gamma_1\in NCG_{(1,j)}([1,j])$.

\medskip
Let $\pi_{\Gamma}\in NCP(j)\times NCP(k-j+1,\ell)$ denote the partition associated with a graph $\Gamma\in NCG_{(1,j)}([1,j])\times \cG([j,k],[k+1,k+\ell])$ in the second term of~\eqref{eq-Fdecomp2} and let $\pi_{\Gamma}'$ is be an annular non-crossing permutation with blocks given by $\pi_{\Gamma}$. Since the edge $(1,j)$ must occur in $\Gamma$, the vertices 1 and $j$ must be associated with same cycle of $\pi_{\Gamma}'$. Hence, the elements in $[1,j\rangle$ and $[j,k]$ are in different cycles of $K(\pi_{\Gamma}')$. Moreover, none of $1,\dots,j-1$ can be part of a connecting cycle in $K(\pi_{\Gamma}')$. This implies the decomposition
\begin{displaymath}
\prod_{B\in K(\pi_{\Gamma}')}\Big\langle \prod_{j\in B}A_j\Big\rangle=\Big(\prod_{B\in K(\pi_{\Gamma}'|_{[1,j]})}\Big\langle \prod_{i\in B\setminus\{j\}}A_i\Big\rangle\Big)\Big(\prod_{B\in K(\pi_{\Gamma}'|_{[j,k]\cup [k+1,k+\ell]})}\Big\langle \prod_{i\in B}A_i\Big\rangle\Big).
\end{displaymath}
Again, the only difference between $\cF(\pi_{\Gamma})$ and $\cF(\pi_{\Gamma\cup\{(1,k)\}})$ is the fact that $\langle A_k\rangle$ must occur as a separate factor in the second case. Thus, we can argue as in~\eqref{eq-Fsimplified2} and evaluate
\begin{align}
\cF(\pi_{\Gamma})+q_{1,k}\cF(\pi_{\Gamma\cup\{(1,k)\}})&=\Big(\prod_{B\in K(\pi_{\Gamma}|_{[1,j]})}\Big\langle \prod_{i\in B\setminus\{j\}}A_i\Big\rangle\Big)\Big(\prod_{B\in K(\pi_{\Gamma}'|_{[j,k]\cup [k+1,k+\ell]})}\Big\langle \prod_{i\in B}A_i\Big\rangle\Big)\NN\\
&\quad+q_{1,k}\langle A_k\rangle\Big(\prod_{B\in K(\pi_{\Gamma}|_{[1,j]})}\Big\langle \prod_{i\in B\setminus\{j\}}A_j\Big\rangle\Big)\NN\\
&\quad\quad\times\Big(\prod_{B\in K(\pi_{\Gamma}'|_{[j,k]\cup [k+1,k+\ell]})}\Big\langle \prod_{i\in B\setminus\{k\}}A_i\Big\rangle\Big)\label{eq-Fsimplified3}
\end{align}
whenever $\pi_{\Gamma}$ has more than one connecting block and~\eqref{eq-defF1} applies.

\medskip
In the remaining cases, $\cF(\pi_{\Gamma})$ is evaluated using~\eqref{eq-defF2}. Here, the first term can be treated similarly to~\eqref{eq-Fsimplified3}, leaving only the contribution of the marked partition $\pi_\Gamma''$. Recalling that $[1,j\rangle$ and $[j,k]$ lie on the same circle and that $K(\pi_\Gamma'')$ is evaluated circle-wise, we can argue as in the permutation case. In particular, 
\begin{align*}
\cF(\pi_{\Gamma})&=\sum_{\substack{\pi_{\Gamma}'\in \NCA(k,\ell)\\ blocks(\pi_{\Gamma}')=\pi_{\Gamma}}}\Big(\prod_{B\in K(\pi_{\Gamma}'|_{[1,j]})}\Big\langle \prod_{i\in B\setminus\{j\}}A_i\Big\rangle\Big)\Big(\prod_{B\in K(\pi_{\Gamma}'|_{[j,k]\cup [k+1,k+\ell]})}\Big\langle \prod_{i\in B}A_i\Big\rangle\Big)\\
&\quad+\Big(\prod_{B\in K(\pi_{\Gamma}''|_{[1,j]})}\Big\langle \prod_{i\in B\setminus\{j\}}A_i\Big\rangle\Big)\Big(\prod_{B\in K(\pi_{\Gamma}''|_{[j,k]\cup [k+1,k+\ell]})}\Big\langle \prod_{i\in B}A_i\Big\rangle\Big)\\
&=\Big(\prod_{B\in K(\pi_{\Gamma}|_{[1,j]})}\Big\langle \prod_{i\in B\setminus\{j\}}A_i\Big\rangle\Big)\cF(\pi_{\Gamma_2})
\end{align*}
i.e., $\cF$ factorizes similar to~\eqref{eq-Fsimplified3}. Here, $\Gamma_2$ denotes the subgraph of $\Gamma\in\cG_2$ that lies in $\cG([j,k],[k+1,k+\ell])$. Using~\eqref{eq-formulaM} and~\eqref{eq-mcircgraphs}, we evaluate
\begin{align}
&\Big(\prod_{s=1}^{k+\ell}m_s\Big)\sum_{\Gamma\in NCG_{(1,j)}([1,j])\times \cG([j,k],[k+1,k+\ell])}q_{\Gamma}\Big(\cF(\pi_{\Gamma})+q_{1,k}\cF(\pi_{\Gamma\cup\{(1,k)\}})\Big)\NN\\
&=m_1\m[T_1,\dots,T_{j-1},G_j]\big(F[T_j,\dots,T_k|T_{k+1},\dots,T_{k+\ell}]\NN\\
&\quad\quad+q_{1,k}F[T_j,\dots,T_{k-1},G_k|T_{k+1},\dots,T_{k+\ell}]\langle A_k\rangle\big).\label{eq-rewrittenG2}
\end{align}

\underline{\smash{Computation of the third line of~\eqref{eq-Fdecomp2}}}: The contribution from $\cG_3$ can be treated similarly to the second line of~\eqref{eq-Fdecomp2}.

\medskip
\underline{\smash{Computation of the fourth line of~\eqref{eq-Fdecomp2}}}: For the term that arises from $\cG_4$, recall that~$\tau$ only influences the geometry of the graph, but not its edge set. Hence, the summation reduces to the underlying disk non-crossing graphs and we can evaluate it using~\eqref{eq-formulaM} and~\eqref{eq-mcircgraphs}. As the only difference between $\cF(\pi_{\Gamma})$ and $\cF(\pi_{\Gamma\cup\{(1,k)\}})$ is again the fact that $\langle A_k\rangle$ must occur as a separate factor in the second case, we obtain
\begin{align}
&\Big(\prod_{s=1}^{k+\ell}m_s\Big)\sum_{\Gamma\in NCG_{(1,k+j)}(\{1,k,k+j,\dots,k+\ell,\dots,k+j\})}q_{\Gamma}\Big(\cF(\pi_{\Gamma})+q_{1,k}\cF(\pi_{\Gamma\cup\{(1,k)\}})\Big)\NN\\
&=m_1\big(\m[T_1,\dots,T_k,T_{k+j},\dots,T_{k+j-1},G_{k+j}]\NN\\
&\quad\quad +q_{1,k}\m[T_1,\dots,T_{k-1},G_k,T_{k+j},\dots,T_{k+j-1},G_{k+j}]\langle A_k\rangle\big).\label{eq-rewrittenG4}
\end{align}
Similar to~\eqref{eq-goodweights}, $\Gamma$ containing an edge $(1,k+j)$ ensures that the contribution has the prefactor $m_1$. 

\medskip
Putting~\eqref{eq-rewrittenG1},~\eqref{eq-rewrittenG2}, and~\eqref{eq-rewrittenG4} together, we see that~\eqref{eq-Fdecomp2} is equivalent to~\eqref{eq-Mrecursion} with $\mathfrak{f}[\cdot|\cdot]$ in place of $\m[\cdot|\cdot]$. We conclude that $\mathfrak{f}[\cdot|\cdot]$ satisfies the same symmetry and initial condition as $\m[\cdot|\cdot]$, as well as the same recursion. Recall that these three properties uniquely identify $\mathfrak{f}[\alpha|\beta]$ and $\m[\alpha|\beta]$ for any multi-indices $\alpha,\beta$. It now readily follows by induction that $\mathfrak{f}[\alpha|\beta]$ and $\m[\alpha|\beta]$ indeed coincide for all $\alpha,\beta$, i.e., that $\mathfrak{f}[\cdot|\cdot]=\m[\cdot|\cdot]$, as claimed.
\end{proof}

\section{Proof of the Formulas for $\m_\kappa[\cdot|\cdot]$, $\m_\sigma[\cdot|\cdot]$, and $\m_\omega[\cdot|\cdot]$}\label{sect-restproof}
\subsection{Proof of Theorem~\ref{thm-structurekap}}\label{sect-kappa}
We use proof by induction over the length of the multi-indices and the recursion for $\m_{\kappa}[\cdot|\cdot]$ from~\eqref{eq-Mrecursion}. First, recall that by definition of $NCP(k)$, $NCP(\ell)$, and $\NCA(k,\ell)$ both sums on the right-hand side of~\eqref{eq-fpformulakap} are empty whenever either $\alpha$ or $\beta$ is empty. Observing that Definition~\ref{def-M}(i) together with Theorem~\ref{thm-main} implies that $\m_{\kappa}[S_1|\emptyset]=\m_{\kappa}[\emptyset|S_2]=0$, the base case is established. As the formula on the right-hand side of~\eqref{eq-fpformulakap} is further symmetric under the interchanging $\alpha$ and~$\beta$ (cf. remark below Theorem~\ref{thm-main}), it is sufficient to only carry out the induction step for one of the arguments in $\m_{\kappa}[\cdot|\cdot]$.

\medskip
Fix $k,\ell\in\N$ and assume that~\eqref{eq-fpformulakap} holds for multi-indices $\alpha$ of length $1,\dots,k-1$ and $\beta$ of length $\ell$. Recalling that $\kappa_4\m_{\kappa}[\cdot|\cdot]$ satisfies~\eqref{eq-Mrecursion} with $\mathfrak{s}_{\kappa}$ in~\eqref{eq-sourcekappa} as the only source term, we can use the recursion to express $\m_{\kappa}[\alpha|\beta]$ in terms of $\m[\cdot]$ and $\m_\kappa[\alpha'|\beta]$ with multi-indices $\alpha'$ of length $1,\dots,k-1$. The claim thus follows by applying the induction hypothesis and showing that the expression obtained from the recursion can be rewritten to match the structure on the right-hand side of~\eqref{eq-fpformulakap}. We start by considering the summands on the right-hand side of~\eqref{eq-Mrecursion} separately and then check that their sum, i.e., $\m_{\kappa}[\cdot|\cdot]$, is of the same form. To facilitate keeping track of the individual contributions, we start with the terms on the right-hand side of~\eqref{eq-Mrecursion} that do not contain the prefactor~$q_{1,k}$ and abbreviate
\begin{align*}
K_1&:=m_1\m_{\kappa}[T_2,\dots,T_{k-1},G_kA_kA_1|T_{k+1},\dots,T_{k+\ell}]\\
K_2^{(j)}&:=m_1\m_{\kappa}[T_1,\dots,T_{j-1},G_j|T_{k+1},\dots,T_{k+\ell}]\m[T_j,\dots,T_k],\ j\in[k-1]\\
K_3^{(j)}&:=m_1\m[T_1,\dots,T_{j-1},G_j]\m_{\kappa}[T_j,\dots,T_k|T_{k+1},\dots,T_{k+\ell}],\ j\in[2,k]\\
K_4^{(r,s,t)}&:=m_1\langle M_{[r]}\odot M_{[s,t]}\rangle\langle (M_{[r,k]}A_k)\odot M_{(t,\dots,k+\ell,k+1,\dots,s)}\rangle,\ r\in[k],\ k+1\leq s\leq t\leq k+\ell.
\end{align*}
As the argument is similar, we fix $j$ resp. $r,s,t$ and omit the superscripts for the following discussion.

\medskip
\underline{\smash{Structure of $K_1$:}} By the induction hypothesis, $\m_{\kappa}[T_2,\dots,T_{k-1},G_kA_kA_1|T_{k+1},\dots,T_{k+\ell}]$ factorizes into expressions involving only deterministic matrices or spectral parameters, respectively, and further has the structure specified on the right-hand side of~\eqref{eq-fpformulakap} in terms of $z_2,\dots,z_{k+\ell}$ and the matrices $A_2,\dots,A_{k-1},A_kA_1,A_{k+1},\dots,A_{k+\ell}$. As a consequence, the matrices $A_k$ and $A_1$ always occur together in the matrix products. On the level of the indices of the deterministic matrices, we may reinterpret this as an element $\pi\in\NCA(k,\ell)$ with a cycle $(\dots k1\dots)$ or an element of $\pi\in NCP(k)\times NCP(\ell)$ with a block $\{\dots,k,1,\dots\}$ depending on the rest of the underlying structure. The treament of the two cases is identical and we consider them in parallel. As the indices $1$ and $k$ always occur together in $K(\pi)$, the index 1 must occur separated in $K^{-1}(\pi)$, either as a fixed point $(1)$ or a singleton set $\{1\}$, to match the structure in~\eqref{eq-fpformulakap}. Note that the spectral parameter $z_1$ only appears in the prefactor $m_1$. Hence, setting the functions $\psi_{K^{-1}(\pi),(1)}$ resp. $\psi_{K^{-1}(\pi),\{1\}}$ equal to $m(z_1)$ yields the missing contribution. It follows that $K_1$ matches the structure on the right-hand side of~\eqref{eq-fpformulakap}. Note that all $\psi_i$ associated with permutations without a fixed point $(1)$ and partitions without a singleton block $\{1\}$, respectively, are equal to zero for $K_1$.

\medskip
\underline{\smash{Structure of $K_2$:}} We apply the induction hypothesis for $\m_{\kappa}[T_1,\dots,T_{j-1},G_j|T_{k+1},\dots,T_{k+\ell}]$ as well as~\eqref{eq-formulaM} for $\m[T_j,\dots,T_k]$ to rewrite $K_2$ as a sum of terms that naturally factorize into expressions involving only $z_1,\dots,z_{k+\ell}$ or $A_1,\dots,A_{k+\ell}$, respectively. It remains to check for the structure on the right-hand side of~\eqref{eq-fpformulakap}, i.e., that each summand can be associated with an annular non-crossing permutation $\pi$ or a marked element $\pi\in NCP(k)\times NCP(\ell)$ such that the terms involving spectral parameters factorize according to cycles resp. blocks of $\pi$ and the contribution of the deterministic matrices factorizes according to the cycles resp. blocks in the Kreweras complement $K(\pi)$. As treatment of the cases $\pi\in\NCA(k,\ell)$ and $\pi\in NCP(k)\times NCP(\ell)$ is identical, we consider them in parallel.

\medskip
Observe that the induction hypothesis and~\eqref{eq-formulaM} already prescribe the desired complement structure for the indices of the  spectral parameters and deterministic matrices occurring in $m_1\m_{\kappa}[T_1,\dots,T_{j-1},G_j|T_{k+1},\dots,T_{k+\ell}]$ and $\m[T_j,\dots,T_k]$, respectively. As we may visualize the elements of $NCP(j)$ with their Kreweras complement on an interval by cutting the boundary of the labeled disk (cf. Fig.~\ref{fig-disknc}), we can draw both factors of $K_2$ onto the same annulus. The result is visualized on the left of Fig.~\ref{fig-gluednca} below. For simplicity, we omit most of the intermediate labels and only add some of the matrices associated with the vertices on the midpoints of the arcs between the labels in red.

\medskip
Note that the interval is placed such that the orientation inherited from the disk aligns with the orientation of the underlying annulus. Moreover, Fig.~\ref{fig-gluednca} matches the picture of a $(k,\ell)$-annulus up to the label $j$ occurring twice and the label on the midpoint of the arch connecting the two copies of $j$ that is associated with the identity matrix. Since this identity matrix does not influence the value of $K_2$, we may remove the label corresponding to $\Id$ from the picture in Fig.~\ref{fig-gluednca}. This leaves $k+\ell$ labels at the midpoints of arches along the annulus, one associated with each matrix $A_1,\dots,A_{k+\ell}$. On the level of the indices of the deterministic matrices, each term in in $K_2$ can thus be identified with an element $\pi\in\NCA(k,\ell)$ resp. an element of $\pi\in NCP(k)\times NCP(\ell)$. As the two labels~$j$ now occur next to each other, we merge them to obtain the $(k,\ell)$-annulus as the structure underlying the indices of the spectral parameters. The result is visualized on the right of Fig.~\ref{fig-gluednca} below. Note that the labels now match Definition~\ref{def-Kreweras2} exactly.

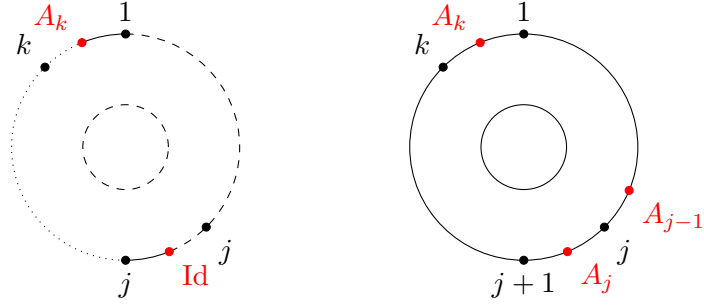
\begin{figure}[H]
\begin{center}
\begin{tikzpicture}[scale=0.75]
\draw[dashed] (0,0) circle (0.75cm);
\draw[dotted] (0,-2) arc[start angle=270, end angle=112.5, radius=2];
\draw[dashed] (0,2) arc[start angle=90, end angle=-67.5, radius=2];
\draw (0,2) arc[start angle=90, end angle=112.5, radius=2];
\draw (0,-2) arc[start angle=270, end angle=292.5, radius=2];

\draw (0,2) node[above=1pt] {1};
\filldraw [black] (0,2) circle (2pt);
\draw (1.414,-1.414) node [below right=1pt] {$j$};
\filldraw [black] (1.414,-1.414)  circle (2pt);
\draw (0,-2) node[below=1pt] {$j$};
\filldraw [black] (0,-2) circle (2pt);
\draw (-1.414,1.414) node [above left=1pt] {$k$};
\filldraw [black] (-1.414,1.414) circle (2pt);

\draw (0.765, -1.848) node[below right=1pt] {\color{red}$\Id$\color{black}};
\filldraw [red] (0.765, -1.848) circle (2pt);
\draw (-0.765, 1.848) node[above left=1pt] {\color{red}$A_k$\color{black}};
\filldraw [red] (-0.765, 1.848) circle (2pt);
\end{tikzpicture}\hspace{2cm}
\begin{tikzpicture}[scale=0.75]
\draw (0,0) circle (0.75cm);
\draw (0,0) circle (2cm);

\draw (0,2) node[above=1pt] {1};
\filldraw [black] (0,2) circle (2pt);
\draw (0,-2) node[below=1pt] {$j+1$};
\filldraw [black] (0,-2) circle (2pt);
\draw (1.414,-1.414) node [below right=1pt] {$j$};
\filldraw [black] (1.414,-1.414)  circle (2pt);
\draw (-1.414,1.414) node [above left=1pt] {$k$};
\filldraw [black] (-1.414,1.414) circle (2pt);

\draw (1.848,-0.765) node[below right=1pt] {\color{red}$A_{j-1}$\color{black}};
\filldraw [red] (1.848,-0.765) circle (2pt);
\draw (0.765, -1.848) node[below right=1pt] {\color{red}$A_j$\color{black}};
\filldraw [red] (0.765, -1.848) circle (2pt);
\draw (-0.765, 1.848) node[above left=1pt] {\color{red}$A_k$\color{black}};
\filldraw [red] (-0.765, 1.848) circle (2pt);
\end{tikzpicture}
\end{center}
\captionof{figure}{Visualization of the indices in $m_1\m_{\kappa}[T_1,\dots,T_{j-1},G_j|T_{k+1},\dots,T_{k+\ell}]$ (dashed) and in $\m[T_j,\dots,T_k]$ (dotted) before the rewriting (left) and the indices after merging the two labels $j$ (right).}\label{fig-gluednca}
\end{figure}

It is readily checked that the cycle resp. block structure obtained from merging the two labels~$j$ matches $K^{-1}(\pi)$ and that any cycles resp. blocks that were connected in this step can be interpreted as a cycle of an element in $\NCA(k,\ell)$ resp. a block of an element of $NCP(k)\times NCP(\ell)$. In particular, the contribution of the spectral parameters factorizes as claimed on the right-hand side of~\eqref{eq-fpformulakap}. Comparing the permutations resp. partitions contributing to each $\smash{K_2^{(j)}}$ and noting that the right-hand side of~\eqref{eq-fpformulakap} is linear in each $\psi_i$, the sum $\smash{\sum_jK_2^{(j)}}$ also has the desired structure.

\medskip
\underline{\smash{Structure of $K_3$:}} The treatment of $K_3$ is analogous to that of $K_2$.

\medskip
\underline{\smash{Structure of $K_4$:}} From the formula for $M_{(\dots)}$ in~\eqref{eq-LLformula}, it follows that $K_4$ can be written as a sum of terms that naturally factorize into expressions involving only $z_1,\dots,z_{k+\ell}$ or $A_1,\dots,A_{k+\ell}$, respectively. In particular, the part that involves deterministic matrices always consists of two factors of the form $\langle (\prod_{j\in I_1}A_j)\odot(\prod_{j\in I_2}A_j)\rangle$ with index sets $I_1\subseteq[k]$ and $I_2\subseteq[k+1,\dots,k+\ell]$ due to the Hadamard product in $K_4$. It remains to check for the structure on the right-hand side of~\eqref{eq-fpformulakap}.

\medskip
Using the same trick as for $K_2$ and $K_3$, we can visualize the terms involved in $K_4$ on the same annulus. The result is sketched in Fig.~\ref{fig-gluedncp} below. Again, we omit most of the labels and add the matrices associated with the vertices on the midpoints of the arcs in red. To avoid overcrowding the labels in the interior of the inner circle, we further visualize the two circles separately.

\begin{figure}[H]
\begin{center}
\begin{tikzpicture}[scale=0.75]
\draw[dotted] (0,-2) arc[start angle=270, end angle=112.5, radius=2];
\draw[dashed] (0,2) arc[start angle=90, end angle=-67.5, radius=2];
\draw (0,2) arc[start angle=90, end angle=112.5, radius=2];
\draw (0,-2) arc[start angle=270, end angle=292.5, radius=2];

\draw (0,2) node[above=1pt] {1};
\filldraw [black] (0,2) circle (2pt);
\draw (1.414,-1.414) node [below right=1pt] {$r$};
\filldraw [black] (1.414,-1.414) circle (2pt);
\draw (0,-2) node[below=1pt] {$r$};
\filldraw [black] (0,-2) circle (2pt);
\draw (-1.414,1.414) node [above left=1pt] {$k$};
\filldraw [black] (-1.414,1.414)  circle (2pt);

\draw (0.765, -1.848) node[below right=1pt] {\color{red}$\Id$\color{black}};
\filldraw [red] (0.765, -1.848) circle (2pt);
\draw (-0.765, 1.848) node[above left=1pt] {\color{red}$A_k$\color{black}};
\filldraw [red] (-0.765, 1.848) circle (2pt);
\end{tikzpicture}\hspace{2cm}
\begin{tikzpicture}[scale=0.75]
\draw[dashed] (2,0) arc[start angle=0, end angle=-135, radius=2];
\draw[dashed] (2,0) arc[start angle=0, end angle=22.5, radius=2];
\draw[dotted] (-2,0) arc[start angle=180, end angle=202.5, radius=2];
\draw[dotted] (-2,0) arc[start angle=180, end angle=45, radius=2];
\draw (1.84776, 0.765367) arc[start angle=22.5, end angle=45, radius=2];
\draw (-1.84776,-0.765367) arc[start angle=202.5, end angle=225, radius=2];

\draw (-0.765, 1.848) node[above left=1pt] {$k+1$};
\filldraw [black] (-0.765, 1.848) circle (2pt);
\draw (0,2) node[above=1pt] {$k+\ell$};
\filldraw [black] (0,2) circle (2pt);
\draw (-1.414,-1.414) node [below left=1pt] {$s$};
\filldraw [black] (-1.414,-1.414) circle (2pt);
\draw (-2,0) node[left=1pt] {$s$};
\filldraw [black] (-2,0) circle (2pt);
\draw (2,0) node[right=1pt] {$t$};
\filldraw [black] (2,0) circle (2pt);
\draw (1.414,1.414) node [above right=1pt] {$t$};
\filldraw [black] (1.414,1.414) circle (2pt);
\draw (0,-2) node[below=1pt] {\color{white}r\color{black}};

\draw (-1.848, -0.765) node[left=4pt] {\color{red}$\Id$\color{black}};
\filldraw [red] (-1.848, -0.765) circle (2pt);
\draw (1.848, 0.765) node[right=5pt] {\color{red}$\Id$\color{black}};
\filldraw [red] (1.848, 0.765) circle (2pt);
\end{tikzpicture}
\end{center}
\captionof{figure}{Visualization of the indices in $m_1\langle M_{[r]}\odot M_{[s,t]}\rangle$ (dashed) and $\langle (M_{[r,k]}A_k)\odot M_{(t,\dots,k+\ell,k+1,\dots,s)}\rangle$ (dotted) on the outer (left) and inner circle (right) of the $(k,\ell)$-annulus~(pictured separately).}\label{fig-gluedncp}
\end{figure}
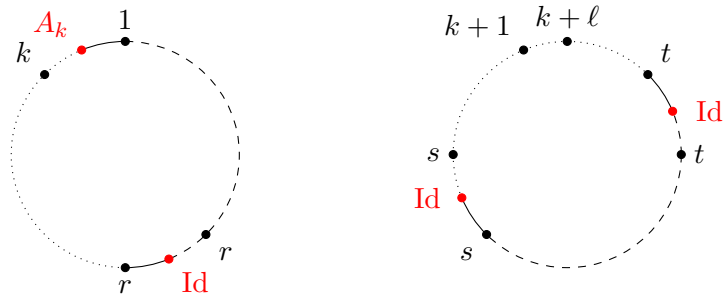

Observe that whenever a label occurs twice, both copies are placed next to one another and the label at the midpoint of the arch connecting them is always associated with the identity matrix. As the identity matrices do not contribute to $K_4$, we may argue as before and remove the corresponding labels from the picture in Fig.~\ref{fig-gluedncp}. Note that this leaves $k+\ell$ labels at the midpoints of arches along the annulus, one associated with each matrix $A_1,\dots,A_{k+\ell}$, such that we can reinterpret the  structure on the level of the indices of the deterministic matrices as an element $\pi\in\NCA(k,\ell)$ resp. $\pi\in NCP(k)\times NCP(\ell)$. Since the two copies of $r$, $s$, and $t$ are now placed right next to their counterpart, we merge them to obtain a $(k,\ell)$-annulus as the structure underlying the indices of the spectral parameters. It is readily checked that the cycle resp. block structure obtained from this step matches $K^{-1}(\pi)$ and that any cycles resp. blocks that were connected by merging two identical labels can be interpreted as a cycle of an element in $\NCA(k,\ell)$ resp. a block of an element of $NCP(k)\times NCP(\ell)$. In particular, whenever the merging of the labels creates two new cycles, the cycles can be drawn onto the $(k,\ell)$-annulus without crossing. It follows that the contribution of the spectral parameters factorizes as claimed on the right-hand side of~\eqref{eq-fpformulakap}. Comparing the permutations resp. partitions contributing to each $\smash{K_4^{(r,s,t)}}$ and noting that the right-hand side of~\eqref{eq-fpformulakap} is linear in each $\psi_i$, summing over $r,s,t$ preserves the structure of the term. 

\medskip
The remaining terms with prefactor $q_{1,k}\langle A_k\rangle$ can be rewritten similarly. As the factor $\langle A_k\rangle$ always occurs separately, the structure underlying the indices of the deterministic matrices is an element $\pi\in\NCA(k,\ell)$ resp. $\pi\in NCP(k)\times NCP(\ell)$ with a fixed point $(k)$ resp. a singleton set $\{k\}$. Hence, $K^{-1}$ must contain a cycle $(\dots k1\dots)$ resp. a block $\{\dots,k,1,\dots\}$ in which $1$ and $k$ occur together. Recalling that $q_{1,k}=\frac{m_1m_k}{1-m_1m_k}$, it is ensured that the part of the term depending on $z_1,\dots,z_k$ contains a factor $\psi_i(\dots,z_k,z_1,\dots)$. With this modification, we can use the same argument as for $K_1,\dots,K_4$ to conclude that all terms contributing to the recursion match the structure on the right-hand side of~\eqref{eq-fpformulakap}. Comparing the permutations resp. partitions contributing to each term and recalling that the right-hand side of~\eqref{eq-fpformulakap} is linear in each $\psi_i$, it follows that the same holds for their sum, i.e., for $\m_\kappa[\alpha|\beta]$. This concludes the induction step.
\qed

\subsection{Proof of Theorems~\ref{thm-msigmaform} and~\ref{thm-structureom}}\label{sect-restms}

The proofs of Theorems~\ref{thm-msigmaform} and~\ref{thm-structureom} are similar to those of Theorems~\ref{thm-main} and~\ref{thm-structurekap}, respectively. We, therefore, mainly focus on the necessary modifications below.

\begin{proof}[Proof of Theorem~\ref{thm-msigmaform}]
The overall argument is analogous to the proof of Theorem~\ref{thm-main} in Sections~\ref{sect-graphs} and~\ref{sect-proof-fpformula} up to the definition of the graphs appearing in the combinatorial formula for~$m_{\sigma}[\cdot|\cdot]$. Recalling that the structure of the recursion for $m_{\sigma}[\cdot|\cdot]$ is similar to the one for $m_{GUE}[\cdot|\cdot]$, the multi-set of graphs can be constructed as described in Definition~\ref{def-Ggraphs}. However, since the source term involves $m^{\#,\sigma}[\cdot]$, the resulting multi-set of graphs will carry the same kind of vertex coloring. By replacing $\cG_4$ in Definition~\ref{def-Ggraphs} by
\begin{align*}
\cG^\sigma_4&:=\bigcup_{j=1}^{\ell} \tau\Big(\big\{\Gamma\in NCG^\#(\{1,\dots,k,k+j,k+j-1,\dots,k+1,k+\ell,k+j\})\big|\\ &\quad\quad\quad\Gamma\text{ has edge }(1,k+j)\big\}\Big).
\end{align*}
where $\#=(0,\dots,0,1,\dots,1)$ with $k$ zeros and $k+\ell+1$ ones, we obtain a family $\cG^\sigma(k,\ell)$ of graphs with the desired properties. Recall that map $\tau$ was introduced in Definition~\ref{def-tau}. An example in the current setting is given in Fig.~\ref{fig-deftau2} below. Note the different arrangement of the indices on the left compared to Fig.~\ref{fig-deftau}, which results from the structure of the source term in the recursion for $m_\sigma[\cdot|\cdot]$. Further, $\tau$ does not influence the coloring of the vertices.

\begin{figure}[H]
\begin{center}
\begin{tikzpicture}[scale=1.25,baseline=(current bounding box.center)]
\draw (0.3827,0.9239) node[above=1pt] {$1$};
\filldraw [black] (0.3827,0.9239) circle (1pt);
\draw (0.9239,0.3827) node[right=1pt] {$2$};
\filldraw [black] (0.9239,0.3827) circle (1pt);
\draw (0.9239,-0.3827) node[right=1pt] {$3$};
\filldraw [black] (0.9239,-0.3827) circle (1pt);
\draw (0.3827,-0.9239) node[below=1pt] {$4$};
\filldraw [black] (0.3827,-0.9239) circle (1pt);

\draw (-0.3827,-0.9239) node[below=1pt] {\color{red}$6'$\color{black}};
\filldraw [red] (-0.3827,-0.9239) circle (1pt);
\draw (-0.9239,-0.3827) node[left=1pt] {\color{red}$5$\color{black}};
\filldraw [red] (-0.9239,-0.3827) circle (1pt);
\draw (-0.9239,0.3827) node[left=1pt] {\color{red}$7$\color{black}};
\filldraw [red] (-0.9239,0.3827) circle (1pt);
\draw (-0.3827,0.9239) node[above=1pt] {\color{red}$6$\color{black}};
\filldraw [red] (-0.3827,0.9239) circle (1pt);

\draw (0,-1) node {$\color{black}|\color{black}$};
\draw (0,1) node {$\color{black}|\color{black}$};

\draw (0,0) circle (1cm);
\end{tikzpicture}\hspace{0.3cm}\scalebox{1.1}{$\overset{\text{conformal map}}\longrightarrow$}
\begin{tikzpicture}[scale=1.25,baseline=(current bounding box.center)]
\draw (0,1) node[above=1pt] {1};
\filldraw [black] (0,1) circle (1pt);
\draw (1,0) node[right=1pt] {2};
\filldraw [black] (1,0) circle (1pt);
\draw (0,-1) node[below=1pt] {3};
\filldraw [black] (0,-1) circle (1pt);
\draw (-1,0) node[left=1pt] {4};
\filldraw [black] (-1,0) circle (1pt);

\draw (0.3536,0.3536) node[below left=0.3536pt] {\color{red}7\color{black}};
\filldraw [red] (0.3536,0.3536) circle (1pt);
\draw (-0.3536,0.3536) node[below right=0.3536pt] {\color{red}6\color{black}};
\filldraw [red] (-0.3536,0.3536) circle (1pt);
\draw (-0.3536,-0.3536) node[above right=0.1pt] {\color{red}6'\color{black}};
\filldraw [red] (-0.3536,-0.3536) circle (1pt);
\draw (0.3536,-0.3536) node[above left=0.1pt] {\color{red}5\color{black}};
\filldraw [red] (0.3536,-0.3536) circle (1pt);

\draw (0,0) circle (1cm);
\draw (0,0) circle (0.5cm);

\draw[black, style={double,double distance=2pt}] (-0.707,0.707) .. controls (-0.55,0.65) and (-0.6,-0.05) .. (-0.5,0);

\end{tikzpicture}\hspace{0.3cm}\scalebox{1.1}{$\overunderset{\text{remove slit}}{\text{join vertices}}{\longrightarrow}$}
\begin{tikzpicture}[scale=1.25,baseline=(current bounding box.center)]
\draw (0,1) node[above=1pt] {1};
\filldraw [black] (0,1) circle (1pt);
\draw (1,0) node[right=1pt] {2};
\filldraw [black] (1,0) circle (1pt);
\draw (0,-1) node[below=1pt] {3};
\filldraw [black] (0,-1) circle (1pt);
\draw (-1,0) node[left=1pt] {4};
\filldraw [black] (-1,0) circle (1pt);

\draw (0.3536,0.3536) node[below left=0.3536pt] {\color{red}7\color{black}};
\filldraw [red] (0.3536,0.3536) circle (1pt);
\draw (-0.3536,0.3536) node[below right=0.3536pt] {\color{red}6\color{black}};
\filldraw [red] (-0.3536,0.3536) circle (1pt);
\draw (0,-0.5) node[above=0.75pt] {\color{red}5\color{black}};
\filldraw [red] (0,-0.5) circle (1pt);

\draw (0,0) circle (1cm);
\draw (0,0) circle (0.5cm);
\end{tikzpicture}
\end{center}
\captionof{figure}{The geometry of the transformation $\tau$ for $k=4$, $\ell=3$, and $j=2$.}\label{fig-deftau2}
\end{figure}
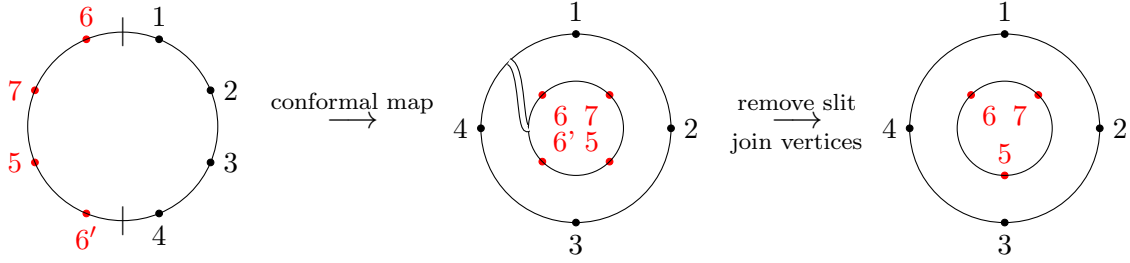

Similar to the proof of Lemma~\ref{lem-graphformula1}, we obtain
\begin{displaymath}
m_{\sigma}[1,\dots,k|k+1,\dots,k+\ell]=\Big(\prod_{s=1}^{k+\ell}m_s\Big)\sum_{\Gamma\in\cG^\sigma(k,\ell)}\prod_{(i,j)\in E(\Gamma)}q_{i,j}^\sharp
\end{displaymath}
where $q_{i,j}^\sharp$ is as in Lemma~\ref{lem-msharpgraphs}, i.e., $\smash{q_{i,j}^\sharp=q_{i,j}=\frac{m_im_j}{1-m_im_j}}$ whenever the edge $(i,j)$ connects two vertices of the same color and $\smash{q_{i,j}^\sharp=\frac{\sigma m_im_j}{1-\sigma m_im_j}}$ otherwise. From here, the remaining steps are carried out as in the proof of Theorem~\ref{thm-main}.
\end{proof}

\begin{proof}[Proof of Theorem~\ref{thm-structureom}]
We use again proof by induction. As the recursions for $\m_{\omega}[\cdot|\cdot]$ and $\m_{\kappa}[\cdot|\cdot]$ are the same up to the source term, it only remains to show that
\begin{displaymath}
K_5^{(j)}:=\langle (M_{[k]}A_k)\odot M_{(k+j,\dots,k+\ell,k+1,\dots k+j)}\rangle,\ j\in[\ell]
\end{displaymath}
is of the form~\eqref{eq-fpformulaom} and that summing up the contributions on the right-hand side of the recursion~\eqref{eq-Mrecursion} does not break the structure. We start by noting that each $\smash{K_5^{(j)}}$ contains the matrices $A_1,\dots,A_{k+\ell}$ exactly once and that the indices involved in the first and second factor of the Hadamard product match the indices on the outer and inner circle of the $(k,\ell)$-annulus, respectively. The desired structure now follows from~\eqref{eq-formulaM} and the fact that any annular non-crossing permutation of the $(k,\ell)$-annulus can be decomposed uniquely into a composition of two permutations that only act on the inner and outer circle, respectively. Note that the index $k+j$ occurring twice in $M_{(k+j,\dots,k+\ell,k+1,\dots k+j)}$ does not influence the structure of the Kreweras complement. When visualizing the term on a labeled disk, the two vertices labeled $k+j$ lie next to each other and the vertex on the midpoint of the arch connecting the two copies is associated with the identity matrix, i.e., its effect is not visible when the matrix product in $M_{(k+j,\dots,k+\ell,k+1,\dots k+j)}$ is evaluated.
\end{proof}

\appendix
\section{Proofs for the CLT in the Resolvent Case}

\subsection{Proof of Theorem~\ref{thm-skewedLL} (Global Law with Transposes)}\label{app-skewedLLproof}
The proof of Theorem~\ref{thm-skewedLL} is, modulo careful bookkeeping of the transposes, similar to the proof of the averaged local law in~\cite[Thm.~3.4]{CES-thermalization}. In particular, an explicit formula for $m^{\#,\sigma}[\cdot]$ and the associated free cumulants is obtained along the way. For the convenience of the reader, we give a brief overview of the necessary changes. As there is nothing to prove if $\sigma=1$ and the case $\sigma=-1$ needs an additional argument, consider $\sigma\in(-1,1)$ first. We start by introducing a renormalization that captures the general second moment structure of $W$ more precisely (cf.~\cite{CES-ETH} and~\cite{CES-functCLT}). Let
\begin{equation}\label{eq-underline2}
\underline{Wf(W)}:= Wf(W)-\widetilde{\E}\widetilde{W}(\partial_{\widetilde{W}}f)(W)
\end{equation}
with $\widetilde{W}$ an independent copy of $W$. This yields, e.g.,
\begin{align}
\underline{WG_1}&=WG_1+\langle G_1\rangle G_1+\frac{\sigma}{N}G_1^tG_1+\frac{\widetilde{\omega_2}}{N}\diag{G_1}G_1,\label{eq-underlinesingle2}\\
\underline{WT_1\dots T_k}&=\underline{WG_1}A_1T_{[2,k]}+\sum_{j=2}^k\Big(\langle T_{[1,j\rangle}G_j\rangle T_{[j,k]}+\frac{\sigma}{N}(T_{[1,j\rangle}G_j)^tT_{[j,k]}\NN\\
&\quad+\frac{\widetilde{\omega_2}}{N}\diag{ T_{[1,j\rangle}G_j}T_{[j,k]}\Big).\label{eq-underlinechain2}
\end{align}
Note that a simpler renormalization can be obtained by choosing $\widetilde{W}$ to be an independent GUE matrix instead of an independent copy of $W$. For the formulas~\eqref{eq-underlinesingle2} and~\eqref{eq-underlinechain2} above, the difference between the two renormalizations is negligible. However, it becomes significant if the matrix product involves at least one transpose of a resolvent~(cf.~\cite[Rem.~4.3]{CES-ETH}). As an example, consider the normalized trace~of
\begin{displaymath}
\underline{WG_1^t}=WG_1^t+\frac{1}{N} G_1G_1^t+\sigma \langle G_1\rangle G_1^t+\frac{\widetilde{\omega_2}}{N}\diag{G_1}G_1^t
\end{displaymath}
where $\sigma\langle G_1\rangle^2\sim1$.

\medskip
To obtain~\eqref{eq-defMsharp}, we rewrite $G_1^\sharp A_1\dots G_k^\sharp A_k$ in terms of $\underline{\smash{WG_1^\sharp A_1\dots G_k^\sharp A_k}}$ and estimate the resulting terms. We start by considering the case $A_1=\dots=A_k=\Id$. For $k=2$, we obtain, e.g.,
\begin{align*}
\underline{WG_1G_2^t}=\underline{WG_1}G_2^t+\frac{1}{N}(G_1G_2^t)^tG_2^t+\sigma\langle G_1G_2^t\rangle G_2^t+\frac{\widetilde{\omega_2}}{N}\diag{G_1G_2^t}G_2^t
\end{align*}
which yields
\begin{align*}
\Big(1-\sigma m_1m_2+\cO_{\prec}\Big(\frac{1}{N}\Big)\Big)\langle G_1G_2^t\rangle &= m_1m_2-m_1\underline{WG_1G_2^t}+\frac{m_1}{N}\Big(\sigma\langle G_1^tG_1G_2^t\rangle\\
&\quad+\widetilde{\omega_2}\langle \diag{G_1}G_1G_2^t+\diag{G_1G_2^t}G_2^t\rangle\Big)+\cO_\prec\Big(\frac{1}{N}\Big)
\end{align*}
by~\eqref{eq-underlinesingle2}, the resolvent identity $WG-zG=\Id$,~\eqref{eq-mselfcon}, and the local law for $\langle G_j\rangle$. Noting that $|\langle\underline{\smash{ WG_1^tG_2}}\rangle|=\cO_{\prec}(N^{-1})$ following the bounds in~\cite[Thm.~4.1]{CES-ETH} and that the term with prefactor $N^{-1}$ is also of lower order (cf. (69) and (70) in~\cite{CES-functCLT}), we obtain~\eqref{eq-defMsharp} with
\begin{displaymath}
\m^{\#,\sigma}[G_1,G_2^t]=\frac{m_1m_2}{1-\sigma m_1m_2}.
\end{displaymath}
Note that the only effect of the inclusion of transposes lies in the emergence of the stability factor $1-\sigma m_1m_2$. However, as $|1-\sigma m_1m_2|>1-|\sigma|$ for $|\sigma|<1$, the term is bounded from below and does not change the outcome of the estimates. The analog for general $k\geq2$ follows by induction over the number of resolvents. Similar to~\cite[Thm.~3.4]{CES-thermalization}, the above argument allows us to extract a recursion that is satisfied by the deterministic approximation of $\langle G_1^\sharp\dots G_k^\sharp\rangle$ up to an $\cO_{\prec}(N^{-1})$ error. By defining $m^{\#,\sigma}[\cdot]$ to satisfy this recursion exactly, i.e., defining it by~\eqref{eq-mhashrecursion}, yields the desired statement. The recursion for $\m[G_1^\sharp A_1,\dots, G_k^\sharp A_k]$ is obtained analogously.

\medskip
The explicit formula~\eqref{eq-defMsharp} now follows by solving this recursion. As the treatment of the deterministic matrices is completely analogous to~\cite[Thm.~3.4]{CES-thermalization}, we restrict the discussion to the case $A_1=\dots=A_k=\Id$, i.e., to~\eqref{eq-mhashrecursion}.  Here, we obtain that $m^{\#,\sigma}[\cdot]$ and the associated free cumulants have a representation in terms of non-crossing graphs that mirror~\eqref{eq-mgraphs} and~\eqref{eq-mcircgraphs}.

\begin{definition}[Bicolored NCG]
Let $k\in\N$ and denote by $\#$ a binary vector of length~$k$. For every $\Gamma\in NCG(k)$, color the vertex $j\in\{1,\dots,k\}$ red if the $j$-th entry of $\#$ is 1, otherwise color it black. We call the resulting set of graphs \emph{bicolored (disk) non-crossing graphs} on $\{1,\dots,k\}$ and denote it by $NCG^\#(k)$.
\end{definition}

\begin{lemma}\label{lem-msharpgraphs}
Let $k\in\N$ and fix a binary vector $\#$ of length $k$. Then
\begin{align}
m^{\#,\sigma}[1,\dots,k]&=\Big(\prod_{s=1}^k m_s\Big)\sum_{\Gamma\in NCG^\#(k)}q_\Gamma^\#\\
m_\circ^{\#,\sigma}[1,\dots,k]&=\Big(\prod_{s=1}^k m_s\Big)\sum_{\Gamma\in NCG_c^\#(k)}q_\Gamma^\#
\end{align}
where $NCG^\#_c(k)$ denotes the connected graphs in $NCG^\#(k)$ and $q_\Gamma:=\prod_{(i,j)\in\E(\Gamma)}q_{i,j}^\sharp$ with $E(\Gamma)$ denoting the edge set of $\Gamma$. The edge weights $q_{i,j}^\sharp$ are such that $q_{i,j}^\sharp:=q_{i,j}=\frac{m_im_j}{1-m_im_j}$ whenever the edge $(i,j)$ connects two vertices of the same color and $q_{i,j}^\sharp:=\frac{\sigma m_im_j}{1-\sigma m_im_j}$ whenever $(i,j)$ connects a red vertex to a black one.
\end{lemma}
The proof is analogous to~\cite[Lem.~5.2]{CES-thermalization} using the recursive structure of $NCG(k)$. Note that the modified edge weights as well as the factors $\smash{q_{1,k}^\sharp}$ and $\smash{c_{1,j}}$ in the recursion account for the stability factor $1-\sigma m_im_j$ arising whenever the product $G_1^\sharp\dots G_k^\sharp$ involves both resolvents and their transposes. We illustrate Lemma~\ref{lem-msharpgraphs} with an example.

\begin{example}
Let $k=3$ and pick spectral parameters $z_1,z_2,z_3\in\C$ with $|\Im z_j|\gtrsim1$ as well as $A_1=A_2=A_3=\Id$. By Theorem~\ref{thm-skewedLL}, the deterministic approximation of $\langle G_1G_2^tG_3^t\rangle$ is given by $m^{\#,\sigma}[1,2,3]$ with $\#=(0,1,1)$. We visualize the elements of the set $NCG^\#(3)$ in Fig.~\ref{fig-NCGcolored} below. For a better overview, the edges are drawn as solid or dashed according to their contribution to $\smash{q_\Gamma^\#}$. We use this sketch to compute
\begin{displaymath}
m^{\#,\sigma}[1,2,3]=m_1m_2m_3\sum_{\Gamma\in NCG^\#}=\frac{m_1m_2m_3}{(1-\sigma m_1 m_2)(1-\sigma m_1m_3)(1-m_2m_3)}.
\end{displaymath}
Note that we thus reobtain (49) of~\cite{CES-functCLT} on macroscopic scales if $z_2=z_3$.
\end{example}

\begin{figure}[H]
\begin{center}
\begin{tikzpicture}[scale=1,baseline=(current bounding box.center)]
\draw (0,1) node[above=1pt] {1};
\draw (-0.707,-0.707) node[below left=1pt] {\color{red}3\color{black}};
\draw (0.707,-0.707) node[below right=1pt] {\color{red}2\color{black}};

\draw (0,0) circle (1cm);
\end{tikzpicture}
\hspace{0.5cm}
\begin{tikzpicture}[scale=1,baseline=(current bounding box.center)]
\draw (0,1) node[above=1pt] {1};
\draw (-0.707,-0.707) node[below left=1pt] {\color{red}3\color{black}};
\draw (0.707,-0.707) node[below right=1pt] {\color{red}2\color{black}};

\draw[dashed] (0,1) -- (0.707,-0.707);

\draw (0,0) circle (1cm);
\end{tikzpicture}
\hspace{0.5cm}
\begin{tikzpicture}[scale=1,baseline=(current bounding box.center)]
\draw (0,1) node[above=1pt] {1};
\draw (-0.707,-0.707) node[below left=1pt] {\color{red}3\color{black}};
\draw (0.707,-0.707) node[below right=1pt] {\color{red}2\color{black}};

\draw (0.707,-0.707) -- (-0.707,-0.707);

\draw (0,0) circle (1cm);
\end{tikzpicture}
\hspace{0.5cm}
\begin{tikzpicture}[scale=1,baseline=(current bounding box.center)]
\draw (0,1) node[above=1pt] {1};
\draw (-0.707,-0.707) node[below left=1pt] {\color{red}3\color{black}};
\draw (0.707,-0.707) node[below right=1pt] {\color{red}2\color{black}};

\draw[dashed] (0,1) -- (-0.707,-0.707);

\draw (0,0) circle (1cm);
\end{tikzpicture}\\
\begin{tikzpicture}[scale=1,baseline=(current bounding box.center)]
\draw (0,1) node[above=1pt] {1};
\draw (-0.707,-0.707) node[below left=1pt] {\color{red}3\color{black}};
\draw (0.707,-0.707) node[below right=1pt] {\color{red}2\color{black}};

\draw[dashed] (0,1) -- (0.707,-0.707);
\draw (0.707,-0.707) -- (-0.707,-0.707);

\draw (0,0) circle (1cm);
\end{tikzpicture}
\hspace{0.5cm}
\begin{tikzpicture}[scale=1,baseline=(current bounding box.center)]
\draw (0,1) node[above=1pt] {1};
\draw (-0.707,-0.707) node[below left=1pt] {\color{red}3\color{black}};
\draw (0.707,-0.707) node[below right=1pt] {\color{red}2\color{black}};

\draw (0.707,-0.707) -- (-0.707,-0.707);
\draw[dashed] (0,1) -- (-0.707,-0.707);

\draw (0,0) circle (1cm);
\end{tikzpicture}
\hspace{0.5cm}
\begin{tikzpicture}[scale=1,baseline=(current bounding box.center)]
\draw (0,1) node[above=1pt] {1};
\draw (-0.707,-0.707) node[below left=1pt] {\color{red}3\color{black}};
\draw (0.707,-0.707) node[below right=1pt] {\color{red}2\color{black}};

\draw[dashed] (0,1) -- (0.707,-0.707);
\draw[dashed] (0,1) -- (-0.707,-0.707);

\draw (0,0) circle (1cm);
\end{tikzpicture}
\hspace{0.5cm}
\begin{tikzpicture}[scale=1,baseline=(current bounding box.center)]
\draw (0,1) node[above=1pt] {1};
\draw (-0.707,-0.707) node[below left=1pt] {\color{red}3\color{black}};
\draw (0.707,-0.707) node[below right=1pt] {\color{red}2\color{black}};

\draw[dashed] (0,1) -- (0.707,-0.707);
\draw (0.707,-0.707) -- (-0.707,-0.707);
\draw[dashed] (0,1) -- (-0.707,-0.707);

\draw (0,0) circle (1cm);
\end{tikzpicture}
\end{center}
\captionof{figure}{The elements of $NCG^\#(3)$ for $\#=(0,1,1)$. The solid edges contribute a factor of $\frac{m_im_j}{1-m_im_j}$ to $q_\Gamma^\#$ while dashed edges contribute a factor $\frac{\sigma m_im_j}{1-\sigma m_im_j}$.}\label{fig-NCGcolored}
\end{figure}
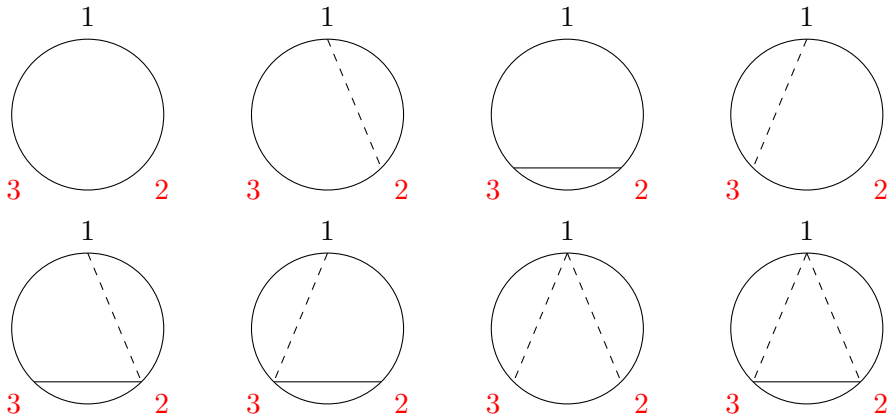

It remains to consider the case $\sigma=-1$. Here, the Wigner matrix is of the form $W=D+\ri S$ with a diagonal matrix $D$ and a skew-symmetric matrix $S$. Note that whenever the diagonal part is equal to zero, the resolvent $R(z_j)=(\ri S-z_j)^{-1}$ satisfies $R(z_j)^t=-R(-z_j)$. This allows treating the proof of~\eqref{eq-defMsharp} for $\langle R(z_1)^\sharp \dots R(z_k)^\sharp\rangle$ with $R(z_j)^\sharp\in\{R(z_j),R(z_j)^t\}$ analogous to the case $\langle G_1\dots G_k\rangle$. Recall that this requires in particular a local law for $\langle R(z_j)\rangle$, which holds even if the Wigner matrix has zero diagonal. The general case follows from the bound
\begin{equation}\label{eq-dropdiag}
\langle G_1^\sharp\dots G_k^\sharp\rangle=\langle R(z_1)^\sharp\dots R(z_k)^\sharp\rangle+\cO_\prec\Big(\frac{1}{N}\Big)
\end{equation}
which reduces the proof of~\eqref{eq-defMsharp} to the previously considered case $D=0$. The proof of~\eqref{eq-dropdiag} is given for $\langle G_1G_2^t\rangle$ in~\cite[App.~B]{CES-functCLT}. By expanding $G_3^\sharp$,\dots, $G_k^\sharp$ analogously, the argument readily extends to $k\geq3$ resolvents.
\qed

\subsection{Proof of Therorem~\ref{thm-resolventCLT2} (CLT for Resolvents)}\label{app-resolventCLTproof}
We follow the general outline of the proof of~\cite[Thm.~3.6]{JRmain} to identify the necessary modifications for the general model described by Assumption~\ref{as-Wigner2}. As the formulas in the general case are derived analogously, we assume w.l.o.g. $\langle A_k\rangle=0$ throughout to keep the following equations short. The first step is the analog of~\cite[Lem.~2.3]{JRmain} that characterizes the difference between $\langle T_{[1,k]}\rangle-\m[T_1,\dots,T_k]$ and $\langle T_{[1,k]}\rangle-\E\langle T_{[1,k]}\rangle$, thus connecting the statistics $X_\alpha$ back to the local law~\eqref{eq-multiGaveraged}.

\begin{lemma}\label{lem-mEexchange2}
Let $k\in\N$ and fix spectral parameters $z_1,\dots,z_n$ with $|\Im z_j|\gtrsim1$ and $\max_j|z_j|\leq N^{100}$ as well as bounded deterministic matrices $A_1,\dots,A_k$ such that $\langle A_k\rangle=0$. Then,
\begin{equation}
\E\langle T_1\dots T_k\rangle=\m[T_1,\dots,T_k]+\frac{1}{N}\cE[T_1,\dots,T_k]+\cO\Big(\frac{N^\eps}{N^{3/2}}\Big)
\end{equation}
with $\m[\cdot]$ as in~\eqref{eq-defM} and a set function $\cE[\cdot]$ that satisfies $\cE[\emptyset]=0$ as well as the recursion
\begin{align*}
\cE[T_1,\dots,T_k]&=m_1\Big(\cE[T_2,\dots,T_{j-1},G_kA_kA_1]+\sum_{j=1}^{k-1}\cE[T_1,\dots,T_{j-1},G_j]\m[T_j,\dots,T_k]\NN\\
&\quad+\sum_{j=2}^k\m[T_1,\dots,T_{j-1},G_j]\cE[T_j,\dots,T_k]+\frac{\widetilde{\omega_2}}{N}\sum_{j=1}^k\langle M_{[j]}\odot(M_{[j,k]}A_k)\rangle\\
&\quad+\frac{\sigma}{N}\sum_{j=1}^k\m[G_j^tA_{j-1}^t,\dots,G_2^tA_1^t,G_1^t,T_j,\dots, T_k]\\
&\quad+\kappa_4\sum_{1\leq r\leq s\leq t\leq k}\langle M_{[r]}\odot M_{[s,t]}\rangle\langle M_{[r,s]}\odot(M_{[t,k]}A_k)\rangle\Big),\NN
\end{align*}
where $\odot$ denotes the Hadamard product and $\m[\cdot|\cdot]$ is interpreted as in~\eqref{eq-mtransposes}.
\end{lemma}

\begin{proof}
We note the following modifications to the proof of~\cite[Lem.~2.3]{JRmain}. Using the renormalization in~\eqref{eq-underline2} with $\widetilde{W}$ being an independent copy of $W$ (compared to the independent GUE matrix $\widetilde{W}$ used in~\cite{JRmain}), we obtain the relation
\begin{align}
\langle T_{[1,k]}\rangle-\m[T_1,\dots,T_k]\NN&=m_1\Big(-\langle \underline{WT_{[1,k]}}\rangle+(\langle T_{[2,k\rangle}G_kA_kA_1\rangle-\m[T_2,\dots,T_{k-1},G_kA_kA_1])\NN\\
&\quad+\sum_{j=1}^{k-1}(\langle T_{[1,j\rangle}G_j\rangle-\m[T_1,\dots,T_{j-1},G_j])\m[T_j,\dots,T_k]\NN\\
&\quad+\sum_{j=2}^k\m[T_1,\dots,T_{j-1},G_j](\langle T_{[j,k]}\rangle-\m[T_j,\dots,T_k])\label{eq-newdifference}\\
&\quad +\frac{\sigma}{N}\sum_{j=1}^k\langle(T_{[1,j\rangle}G_j)^tT_{[j,k]}\rangle+\frac{\widetilde{\omega_2}}{N}\sum_{j=1}^k\langle\diag{T_{[1,j\rangle}G_j}T_{[j,k]}\rangle\Big)+\cO_{\prec}\Big(\frac{1}{N^2}\Big).\NN
\end{align}
Recall that the deterministic approximation $\m[\cdot]$ is independent of the value of $\sigma$ and $\widetilde{\omega}_2$ (cf.~\cite[Thm.~3.4]{CES-thermalization}). The terms in the last line of~\eqref{eq-newdifference} involve an additional factor of $N^{-1}$. Applying Theorem~\ref{thm-skewedLL} as well as the isotropic local law~\eqref{eq-multiGisotropic} thus yields
\begin{align*}
\langle(T_{[1,j\rangle}G_j)^tT_{[j,k]}\rangle&=\m[G_j^tA_{j-1}^t,\dots,G_2^tA_1^t,G_1^t,T_j,\dots, T_k]+\cO_{\prec}\Big(\frac{1}{N}\Big)\\
\langle\diag{T_{[1,j\rangle}G_j}T_{[j,k]}\rangle&=\frac{1}{N}\sum_{r=1}^N(T_{[1,j\rangle}G_j)_{rr}(T_{[j,k]})_{rr}=\frac{1}{N}\sum_{r=1}^N(M_{[j]})_{rr}(M_{[j,k]}A_k)_{rr}+\cO_{\prec}\Big(\frac{1}{\sqrt{N}}\Big)
\end{align*}
and we can replace the respective terms in~\eqref{eq-newdifference} by their deterministic approximation. Note that this also changes the error term to $\cO_{\prec}(N^{-3/2})$. Computing the expectation of~\eqref{eq-newdifference} is now analogous to~\cite[Lem.~2.3]{JRmain}. In particular, the cumulant expansion
\begin{displaymath}
N\E\langle \underline{WT_{[1,k]}}\rangle=\sum_{n\geq2}\sum_{x,y\in[N]}\sum_{\nu\in\{xy,yx\}^n}\frac{\kappa(xy,\nu)}{n!}\E \partial_{\nu}(T_{[1,k]})_{yx}
\end{displaymath}
yields the same result, since the term does not involve any second moments of the entries of $W$.
\end{proof}

Next, we identify the limiting covariance of two modes $X_\alpha$ and $X_\beta$. This yields an analog of~\cite[Lem.~2.5]{JRmain} for the model in Assumption~\ref{as-Wigner2} on macroscopic scales and thus constitutes the base case for the induction argument in the proof of Theorem~\ref{thm-resolventCLT2}.

\begin{lemma}\label{lem-covapprox2}
Fix $k,\ell\in\N$ and let $\alpha,\beta$ be two multi-indices of length $k$ and $\ell$, respectively. Assume that the Wigner matrix $W$ satisfies Assumption~\ref{as-Wigner2} and pick spectral parameters $z_1,\dots,z_{k+\ell}$ with $|\Im z_j|\gtrsim1$  and $\max_j|z_j|\leq N^{100}$ as well as deterministic matrices $A_1,\dots,A_{k+\ell}$ with $\|A_j\|\lesssim 1$. Then, for any $\eps>0$,
\begin{align*}
&N^2\E X^{(k,a)}_\alpha X^{(\ell,b)}_\beta=\m[\alpha|\beta]+\cO\Big(\frac{N^\eps}{\sqrt{N}}\Big),
\end{align*}
with $\m[\cdot|\cdot]$ as introduced in Definition~\ref{def-M}.
\end{lemma}

\begin{proof}[Proof of Lemma~\ref{lem-covapprox2}]
The proof is analogous to that of~\cite[Lem.~2.5]{JRmain} with one key difference. When evaluating the cumulant expansion
\begin{align*}
N^2\E\big(\langle \underline{WT_{[1,k]}}\rangle X^{(\ell)}_{\beta}\big)&=N\sum_{n\geq1}\sum_{x,y\in[N]}\sum_{\nu\in\{xy,yx\}^n}\frac{\kappa(xy,\nu)}{n!}\E \partial_{\nu}\Big((T_{[1,k]})_{yx}X^{(\ell)}_{\beta}\Big),
\end{align*}
the second moment structure of the model in Assumption~\ref{as-Wigner2} has to be taken into account for evaluating the $n=1$ term. We obtain
\begin{align*}
&N\sum_{x,y\in[N]}\sum_{\nu\in\{xy,yx\}}\kappa(xy,\nu)\E \partial_{\nu}\Big((T_{[1,k]})_{yx}X^{(\ell)}_{\beta}\Big)\\
&=-\sum_{j=1}^\ell\E\Big(\langle T_{[1,k]}T_{[k+j,k+\ell]}T_{[1,k+j\rangle}G_{k+j}\rangle+\sigma\langle T_{[1,k]}(T_{[k+j,k+\ell]}T_{[1,k+j\rangle}G_{k+j})^t\rangle\\
&\quad\quad\quad\quad+\widetilde{\omega_2}\langle \diag{T_{[1,k]}}\diag{T_{[k+j,k+\ell]}T_{[1,k+j\rangle}G_{k+j}}\rangle\Big),
\end{align*}
which, compared to the computation in the proof of~\cite[Lem.~2.5]{JRmain}, additionally involves $\sigma$ and $\widetilde{\omega_2}$. The deterministic approximation of the terms follow from~\eqref{eq-multiGaveraged}, Theorem~\ref{thm-skewedLL} and ~\eqref{eq-multiGisotropic}, respectively, which yields the source terms $\mathfrak{s}_{GUE}$, $\mathfrak{s}_\sigma$, and $\mathfrak{s}_\omega$ in~\eqref{eq-Mrecursion}. The rest of the expansion evaluates exactly as its counterpart in~\cite{JRmain} since the term does not involve any second moments of the entries of $W$.
\end{proof}

The rest of the proof of Theorem~\ref{thm-resolventCLT2} is analogous to that of~\cite[Thm.~3.6]{JRmain}, i.e., we apply induction on the number of factors in~\eqref{eq-resolventCLT2} using $\E X_\alpha=0$ and Lemma~\ref{lem-covapprox2} as base cases.
\qed

\section{Additional Proofs and Computations}
\subsection{Proof of Lemma~\ref{lem-doubledindex}}\label{app-doubledindex}
Using~\eqref{eq-mgraphs} for the ordered multi-set $S=\{z_1,\dots,z_j,\dots,z_k,z_j\}$, rewrite the left-hand side of~\eqref{eq-doubledindex} as
\begin{equation}\label{eq-doubledindex2}
m[1,\dots,j,\dots,k,j]=m_1\dots m_{j-1}m_j^2m_{j+1}\dots m_k\sum_{\Gamma\in NCG[\{1,\dots,k,j\}]}\prod_{(a,b)\in E(\Gamma)}q_{a,b}.
\end{equation}
Recalling that $m[S]$ is invariant under any permutation of the elements of $S$, we pick an ordering in which the two $j$'s occur in two consecutive positions and visualize the corresponding non-crossing graphs by equidistantly arranging the vertices on a circle. In this picture, the edge $e$ connecting both $j$'s cannot be involved in any crossing, even in an arbitrary planar graph on the given vertices (see left of Fig.~\ref{fig-cross}). Hence, for any non-crossing graph with edges $\{e_1,\dots,e_n\}\niton e$, the graph with edge set $\{e_1,\dots,e_n,e\}$ is also non-crossing. In particular, every non-crossing graph that involves $e$ has a counterpart that does not. Next, note that there is at most one vertex $l$ among $1,\dots,j-1,j+1,\dots,k$ that is connected to both copies of $j$, as having two distinct vertices $l,l'$ with this property results in a crossing (see right of Fig.~\ref{fig-cross}).

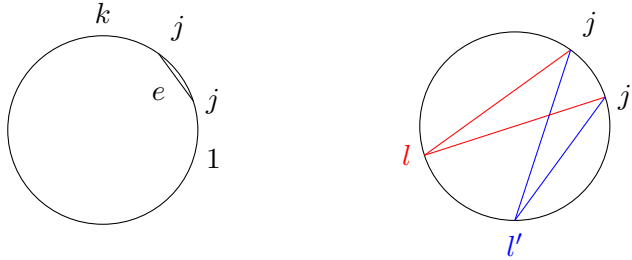
\begin{figure}[H]
\begin{center}
\begin{tikzpicture}[scale=1.25]
\draw (0,1) node[above=1pt] {$k$};
\draw (0.5878,0.809) node[above right=1pt] {$j$};
\draw (0.951,0.309) node[right=1pt] {$j$};
\draw (0.951,-0.309) node[right=1pt] {1};
\draw (0.7694,0.559) node[below left=0.1pt] {$e$};
\draw (0,-1) node[below=1pt]  {\color{white}0\color{black}};

\draw (0,0) circle (1cm);

\draw[black] (0.5878,0.809) -- (0.951,0.309);
\end{tikzpicture}\hspace{2cm}
\begin{tikzpicture}[scale=1.25]
\draw (0,1) node[above=1pt] {\color{white}$0$\color{black}};
\draw (0.5878,0.809) node[above right=1pt] {$j$};
\draw (0.951,0.309) node[right=1pt] {$j$};
\draw (0,-1) node[below=1pt] {\color{blue}$l'$\color{black}};
\draw (-0.951,-0.309) node[left=1pt] {\color{red}$l$\color{black}};

\draw (0,0) circle (1cm);

\draw[red] (-0.951,-0.309) -- (0.5878,0.809);
\draw[red] (-0.951,-0.309) -- (0.951,0.309);
\draw[blue] (0,-1) -- (0.5878,0.809);
\draw[blue] (0,-1) -- (0.951,0.309);
\end{tikzpicture}
\end{center}
\captionof{figure}{An edge between two consecutive vertices never results in a crossing (left), but connecting two distinct vertices to both copies of $j$ always does (right).}\label{fig-cross}
\end{figure}

Consider the disjoint decomposition
\begin{displaymath}
NCG(\{1,\dots,j,\dots,k,j\})=\cS_1\cup\Big(\bigcup_{l\in[k]\setminus\{j\}}\cS_2^{(l)}\Big)
\end{displaymath}
where $\cS_2^{(l)}$ contains all $\Gamma\in NCG(\{1,\dots,j,\dots,k,j\})$ for which the vertex $l\in[k]\setminus\{j\}$ is connected to both copies of $j$ and $\cS_1$ contains any remaining $\Gamma$. This implies
\begin{equation}\label{eq-decomp}
\sum_{\Gamma\in NCG[\{1,\dots,k,j\}]}\prod_{(a,b)\in E(\Gamma)}q_{a,b}=\sum_{\Gamma\in \cS_1}\prod_{(a,b)\in E(\Gamma)}q_{a,b}+\sum_{l\in [k]\setminus\{j\}}\sum_{\Gamma\in\cS_2^{(l)}}\prod_{(a,b)\in E(\Gamma)}q_{a,b}.
\end{equation}
Next, merge both vertices with the label $j$. The resulting graphs are either an element of $NCG(k)$ or arise from an element of $NCG(k)$ by adding a loop $(j,j)$. Recalling that
\begin{equation}\label{eq-qjj}
q_{j,j}=\frac{m_j^2}{1-m_j^2}=m_j'
\end{equation}
where the second equality follows from~\eqref{eq-mselfconderived}, we can factor out $(1+m_j')$ on the right-hand side of~\eqref{eq-decomp} and reduce to summation over $NCG(k)$ without further restriction. Note that this also results in the edge $(l,j)$ doubling whenever $\smash{\Gamma\in\cS_2^{(l)}}$, i.e., we obtain an extra factor of $q_{j,l}$ in this case. Hence,
\begin{displaymath}
\sum_{\Gamma\in NCG[\{1,\dots,k,j\}]}\prod_{(a,b)\in E(\Gamma)}q_{a,b}=(1+m_j')\Big(\sum_{\Gamma\in NCG[\{1,\dots,k\}]}\prod_{(a,b)\in E(\Gamma)}q_{a,b}\Big)\Big(1+\sum_{l\in [k]\setminus \{j\}}q_{j,l}\Big).
\end{displaymath}
Noting that~\eqref{eq-qjj} is equivalent to $m_j(1+m_j')=\frac{m_j'}{m_j}$ and applying~\eqref{eq-mgraphs} to recover $m[1,\dots,k]$ yields~\eqref{eq-doubledindex}. \qed

\subsection{Proof of Corollary~\ref{cor-fpapplication}}\label{app-ncpairings}
We start by evaluating $\mathrm{sc}[i_1,\dots,i_n]$ and $\mathrm{sc}[i_1,\dots,i_n|i_{n+1},\dots,i_{n+m}]$ before considering the associated free cumulant functions. First, note that
\begin{align*}
\mathrm{sc}[i_1,\dots,i_n]=\int_{-2}^2x^n\rho_{sc}(x)\dx x=\begin{cases}0,\quad&\text{if $n$ odd},\\C_{n/2},\quad&\text{if $n$ even},\end{cases}
\end{align*}
where $C_0,C_1,C_2,\dots$ denote the Catalan numbers. In particular, $\mathrm{sc}[i_1,\dots,i_n]$ coincides with the number of non-crossing pairings of the set $[n]$. It readily follows that $\mathrm{sc}_\circ[i_1,i_2]=1$ and that $\mathrm{sc}_\circ[i_1,\dots,i_n]=0$ whenever $n$ is odd. Moreover, $\mathrm{sc}_\circ[i_1,\dots,i_n]=0$ for any even $n\geq4$. The latter follows inductively from~\eqref{eq-mcrelation1} by writing
\begin{equation}\label{eq-cancelcirc}
\mathrm{sc}_\circ[i_1,\dots,i_n]=\mathrm{sc}[i_1,\dots,i_n]-\sum_{\pi\in NCP(n)\setminus\{[n]\}}\prod_{B\in\pi}\mathrm{sc}_\circ[B],
\end{equation}
and observing that only pairings contribute to the sum on the right-hand side of~\eqref{eq-cancelcirc}, i.e., the two terms cancel.

\medskip
Next, we compute $\mathrm{sc}[i_1,\dots,i_n|i_{n+1},\dots,i_{n+m}]$. Observe that by~\eqref{eq-mcrelation2} and Theorem~\ref{thm-functCLT},
\begin{displaymath}
\mathrm{sc}[i_1,\dots,i_n|i_{n+1},\dots,i_{n+m}]=\lim_{N\rightarrow\infty}\E\big[(\Tr W^n-\E W^n)(\Tr W^m-\E W^m)\big].
\end{displaymath}
The limit on the right-hand side is well-known in the free probability literature (see, e.g.,~\cite{MaleMingoPecheSpeicher2020}) and hence readily identified as the number of non-crossing pairings of the $(n,m)$-annulus. Solving~\eqref{eq-mcrelation2} for $\mathrm{sc}_{\circ\circ}[i_1,\dots,i_n|i_{n+1},\dots,i_{n+m}]$, it follows inductively that $\mathrm{sc}_{\circ\circ}[i_1,\dots,i_n|i_{n+1},\dots,i_{n+m}]=0$ for any $n,m$. Hence, $\Phi_{\pi_1\times\pi_2, U_1\times U_2}(f_1,\dots,f_{k+\ell})=0$ and
\begin{align*}
\Phi_{\pi}(f_1,\dots,f_{k+\ell})=\begin{cases}1,\quad &\text{if }\pi\in\NCA_2(k,\ell),\\ 0,\quad &\text{otherwise},\end{cases}
\end{align*}
which is the claim. Note that the error in~\eqref{eq-covfunctionsgeneral} evaluates to $\cO(N^{\eps-1/2})$ as $\|f_i\|_{H^p}$ for $i=1,\dots,k$ resp. $\|f_j\|_{H^q}$ for $j=k+1,\dots,k+\ell$ are $N$-independent constants in the macroscopic regime.
\qed

\renewcommand*{\bibname}{References}

\let\oldthebibliography\thebibliography
\let\endoldthebibliography\endthebibliography
\renewenvironment{thebibliography}[1]{
  \begin{oldthebibliography}{#1}
    \setlength{\itemsep}{0.5em}
    \setlength{\parskip}{0em}
}
{
  \end{oldthebibliography}
}

\bibliographystyle{plain}
\bibliography{References}

\begin{thebibliography}{10}

\bibitem{BaiYao2005}
Z.~D. Bai and J.~Yao.
\newblock On the convergence of the spectral empirical process of {W}igner
  matrices.
\newblock {\em Bernoulli}, 11:1059--1092, 2005.

\bibitem{BaoHe2021}
Z.~Bao and Y.~He.
\newblock Quantitative {CLT} for linear eigenvalue statistics of {W}igner
  matrices.
\newblock {\em Preprint, arXiv:2103.05402}, 2021.

\bibitem{BorotGuionnet2013}
G.~Borot and A.~Guionnet.
\newblock Asymptotic expansion of $\beta$ matrix models in the one-cut regime.
\newblock {\em Commun. Math. Phys.}, 317:447--483, 2013.

\bibitem{CES-ETH}
G.~Cipolloni, L.~Erd\H{o}s, and D.~Schröder.
\newblock Eigenstate thermalization hypothesis for {W}igner matrices.
\newblock {\em Commun. Math. Phys.}, 388(2):1005–1048, 2021.

\bibitem{CES-optimalLL}
G.~Cipolloni, L.~Erd\H{o}s, and D.~Schröder.
\newblock Optimal multi-resolvent local laws for {W}igner matrices.
\newblock {\em Electron. J. Probab.}, 27:1--38, 2022.

\bibitem{CES-thermalization}
G.~Cipolloni, L.~Erd\H{o}s, and D.~Schröder.
\newblock Thermalization for {W}igner matrices.
\newblock {\em J. Funct. Anal.}, 282(8), 2022.

\bibitem{CES-functCLT}
G.~Cipolloni, L.~Erd\H{o}s, and D.~Schröder.
\newblock Functional central limit theorems for {W}igner matrices.
\newblock {\em Ann. Appl. Probab.}, 33(1):447--489, 2023.

\bibitem{CollinsMingoSniadySpeicher2007}
B.~Collins, J.~Mingo, P.~\'{S}niady, and R.~Speicher.
\newblock Second order freeness and fluctuations of random matrices: {III}.
  higher order freeness and free cumulants.
\newblock {\em Documenta Math.}, 12:1--70, 2007.

\bibitem{DiazJaramilloPardo2022}
M.~Diaz, A.~Jaramillo, and J.~C. Pardo.
\newblock Fluctuations for matrix-valued {G}aussian processes.
\newblock {\em Ann. Henri Poincaré}, 58(4):2216--2249, 2022.

\bibitem{DiazMingo2022}
M.~Diaz and J.A. Mingo.
\newblock On the analytic structure of second-order non-commutative probability
  spaces and functions of bounded {F}réchet variation.
\newblock {\em Random Matrices: Theory Appl.}, page 2250044, 2022.

\bibitem{Guionnet2002}
A.~Guionnet.
\newblock Large deviation upper bounds and central limit theorems for
  non-commutative functionals of {G}aussian large random matrices.
\newblock {\em Ann. Inst. H. Poincar\'{e} Probab. Stat.}, 38:341--384, 2002.

\bibitem{HeKnowles2020}
Y.~He and A.~Knowles.
\newblock Mesoscopic eigenvalue density correlations of {W}igner matrices.
\newblock {\em Probab. Theory Relat. Fields}, 177:147--216, 2020.

\bibitem{Johansson1998}
K.~Johansson.
\newblock On fluctuations of eigenvalues of random {H}ermitian matrices.
\newblock {\em Duke Math. J.}, 91(1):151--204, 1998.

\bibitem{KhorunzhyKhoruzhenkoPastur1995}
A.~M. Khorunzhi, B.~A. Khoruzhenko, and L.~A. Pastur.
\newblock On the {1/N} corrections to the {G}reen functions of random matrices
  with independent entries.
\newblock {\em J. Phys. A Math. Gen.}, 28:L31, 1995.

\bibitem{KhorunzhyKhoruzhenkoPastur1996}
A.~M. Khorunzhi, B.~A. Khoruzhenko, and L.~A. Pastur.
\newblock Asymptotic properties of large random matrices with independent
  entries.
\newblock {\em J. Math. Phys.}, 37:5033--5060, 1996.

\bibitem{LandonSosoe2022}
B.~Landon and P.~Sosoe.
\newblock Almost-optimal bulk regularity conditions in the {CLT} for {W}igner
  matrices.
\newblock {\em Preprint, arXiv:2204.03419}, 2022.

\bibitem{Lytova2013}
A.~Lytova.
\newblock On non-{G}aussian limiting laws for certain statistics of {W}igner
  matrices.
\newblock {\em Zh. Mat. Fiz. Anal. Geom.}, 9:536--581, 2013.

\bibitem{LytovaPasturCLT}
A.~Lytova and L.~Pastur.
\newblock Central limit theorem for linear eigenvalue statistics of the
  {W}igner and the sample covariance random matrices.
\newblock {\em Metrika}, 69:153--172, 2009.

\bibitem{LytovaPastur2009}
A.~Lytova and L.~Pastur.
\newblock Fluctuations of matrix elements of regular functions of {G}aussian
  random matrices.
\newblock {\em J. Stat. Phys.}, 134:147--159, 2009.

\bibitem{Male2021}
C.~Male.
\newblock Freeness over the diagonal and global fluctuations of complex
  {W}igner matrices.
\newblock {\em Preprint, arXiv:2104.06157}, 2021.

\bibitem{MaleMingoPecheSpeicher2020}
C.~Male, J.~A. Mingo, S.~Peché, and R.~Speicher.
\newblock Joint global fluctuations of complex {W}igner and deterministic
  matrices.
\newblock {\em Random Matrices: Theory Appl.}, 11(2):2250015, 2022.

\bibitem{MSBook}
J.~A. Mingo and R.~Speicher.
\newblock {\em Free Probability and Random Matrices}.
\newblock Vol. 35, Fields Institute Research Monographs, Springer, New York,
  2017.

\bibitem{MingoNica2004}
J.A. Mingo and A.~Nica.
\newblock Annular noncrossing permutations and partitions, and second-order
  asymptotics for random matrices.
\newblock {\em Int. Math. Res. Not.}, 2004(28):1413--1460, 2004.

\bibitem{MingoSpeicher2006}
J.A. Mingo and R.~Speicher.
\newblock Second order freeness and fluctuations of random matrices {I}:
  {G}aussian and {W}ishart random matrices and cyclic {F}ock spaces.
\newblock {\em J. Funct. Anal.}, 235(1):226--270, 2006.

\bibitem{Redelmeier2012}
C.~E.~I. Redelmeier.
\newblock Real second-order freeness and the asymptotic real second-order
  freeness of several real matrix models.
\newblock {\em Int. Math. Res. Not.}, 2014(12):3353--3395, 2012.

\bibitem{Redelmeier2018}
C.~E.~I. Redelmeier.
\newblock Real and quaternionic second-order free cumulants and connections to
  matrix cumulants.
\newblock {\em Preprint, arXiv:1808.10589v2}, 2018.

\bibitem{JRmain}
J.~Reker.
\newblock Multi-point functional central limit theorem for {W}igner matrices.
\newblock {\em Preprint}, 2023.

\bibitem{Vova2023}
V.~Riabov.
\newblock Mesoscopic eigenvalue statistics for {W}igner-type matrices.
\newblock {\em Preprint, arXiv:2301.01712}, 2023.

\bibitem{Shcherbina2011}
M.~Shcherbina.
\newblock Central limit theorem for linear eigenvalue statistics of the
  {W}igner and sample covariance random matrices.
\newblock {\em Zh. Mat. Fiz. Anal. Geom.}, 7:176--192, 2011.

\bibitem{Shcherbina2013}
M.~Shcherbina.
\newblock Fluctuations of linear eigenvalue statistics of $\beta$ matrix models
  in the multi-cut regime.
\newblock {\em J. Stat. Phys.}, 151:1004--1034, 2013.

\bibitem{SosoeWong2013}
P.~Sosoe and P.~Wong.
\newblock Regularity conditions in the {CLT} for linear eigenvalue statistics
  of {W}igner matrices.
\newblock {\em Adv. Math.}, 249:37--87, 2013.

\bibitem{Wigner1955}
E.~Wigner.
\newblock Characteristic vectors of bordered matrices with infinite dimensions.
\newblock {\em Annals of Mathematics}, 62:548--564, 1955.

\end{thebibliography}

\end{document}